\documentclass[reqno]{amsart}
\usepackage{amsmath,amssymb,verbatim,hyperref,mathrsfs,graphicx}

\renewcommand{\geq}{\geqslant}
\renewcommand{\leq}{\leqslant}

\newcommand{\new}[1]{\emph{#1}}

\newcommand{\mesname}{\nu}
\makeatletter
\long\def\tlist@if@empty@nTF #1{%
\expandafter\ifx\expandafter\\\detokenize{#1}\\%
\expandafter\@firstoftwo
\else
\expandafter\@secondoftwo
\fi
}

\newcommand{\mes}[2][]{
\mesname{\tlist@if@empty@nTF{#1}{}{_{#1}}(#2)
}
}
\newcommand{\essp}[2][]{
\rho{
\tlist@if@empty@nTF{#1}{}{_{#1}}
}(#2
)
}

\newcommand{\norm}[2][]{%
\|#2\|\tlist@if@empty@nTF{#1}{}{_{#1}}
}

\makeatother
\chardef \atcode = \the \catcode `\@
\catcode `\@ = 11
\catcode246=13  \def^^f6{\"o}  
\catcode `\@ = \the \atcode
\newcommand{\citespectral}[2]{{\cite[{{#1}~\ref{#2}}]{AGN-collatz}}}
\newenvironment{rajout}{}{}

\newcommand{\br}{\begin{rajout}}
\newcommand{\er}{\end{rajout}}
\newenvironment{mycomment}{\footnotesize \bf Comments: }{}
\newcommand{\bc}{\begin{mycomment}}
\newcommand{\ec}{\end{mycomment}}

\numberwithin{equation}{section}

\newcommand{\der}[3]{(#1)'_{#2}(#3)}
\newcommand{\derz}[1]{#1'(x)}
\newcommand{\dom}{\operatorname{dom}}
\newtheorem{theorem}{Theorem}[section]
\newtheorem{lemma}[theorem]{Lemma}
\newtheorem{proposition}[theorem]{Proposition}
\newtheorem{corollary}[theorem]{Corollary}
\theoremstyle{definition}

\theoremstyle{remark}
\newtheorem{remark}[theorem]{Remark}

\newcommand{\tho}{\bar d}

\newcommand{\continuous}{\mathscr{C}}

\newcommand{\actionsa}{A}

\newcommand{\R}{\mathbb{R}}

\newcommand{\Z}{\mathbb{Z}}
\newcommand{\leqc}{\leq_{C}}
\newcommand{\NEW}[1]{{\em #1}\index{#1}}
\newcommand{\interior}{\mrm{int}\,}
\newcommand{\Cint}{\interior C}

\newcommand{\mrm}[1]{\text{\rm #1}}
\newcommand{\set}[2]{\{#1\mid\,#2\}}
\newcommand{\comp}{\circ}
\newcommand{\con}{\operatorname{conv}}
\newcommand{\clo}[1]{\operatorname{clo}#1}

\newcommand{\PF}{\mrm{(F)}}
\newcommand{\WKR}{\mrm{WKR}}
\newcommand{\id}{\mrm{Id}}
\newcommand{\seq}[2]{\langle #1\mid #2\rangle}
\newcommand{\stf}[2]{M(#2\, /\,#1)}
\newcommand{\sbf}[2]{m(#2\, /\,#1)}
\def\M(#1/#2){M(#1\,/\,#2)}
\def\m(#1/#2){m(#1\,/\,#2)}

\newcommand{\bonsall}[1]{\tilde{r}_{#1}}
\newcommand{\supeigen}[1]{\operatorname{cw}_{#1}}

\newcommand{\ant}[1]{#1^-}

\newcommand{\Ev}{E}
\newcommand{\nmv}{\nmvf{\cdot}}
\newcommand{\nmvf}[1]
{{\left\vert\kern-0.3ex\left\vert\kern-0.3ex%
\left\vert #1 \right\vert\kern-0.3ex\right\vert\kern-0.3ex\right\vert}}

\makeindex
\begin{document}
\title[Fixed points of nonexpansive semidifferentiable maps]{Uniqueness of the fixed point of nonexpansive semidifferentiable maps}
\author{Marianne Akian}
\address{Marianne Akian, 
INRIA and CMAP. Address: CMAP, \'Ecole Polytechnique, 91128 Palaiseau Cedex, France}
\email{marianne.akian@inria.fr}
\author{St\'ephane Gaubert}
\thanks{The first two authors were partially supported by the Arpege  programme of the French National Agency of Research (ANR), project ``ASOPT'', number ANR-08-SEGI-005}
\address{St\'ephane Gaubert, 
INRIA and CMAP. Address: CMAP, \'Ecole Polytechnique, 91128 Palaiseau Cedex, France}
\email{stephane.gaubert@inria.fr}
\author{Roger Nussbaum}
\address{Roger Nussbaum, Mathematics Department, Hill Center,
Rutgers University, 110 Frelinghuysen Road,
Piscataway, New Jersey, U.S.A. 08854-8019}
\thanks{The third author was partially supported by NSFDMS 0701171 and by NSFDMS 1201328.}
\email{nussbaum@math.rutgers.edu}
\date{\today}

\begin{abstract}
We consider semidifferentiable (possibly nonsmooth)
maps, acting on a subset of a Banach space, that
are nonexpansive either in the norm of the space or in the Hilbert's or
Thompson's metric inherited from a convex cone. 
We show that the global
uniqueness of the fixed point of the map, as well as the geometric
convergence of every orbit to this fixed point, can be inferred
from the semidifferential of the map at this point. In particular,
we show that
the geometric convergence rate of the orbits to the fixed point can be bounded
in terms of Bonsall's nonlinear spectral radius of the semidifferential.
We derive similar results concerning the uniqueness of the eigenline
and the geometric convergence of the orbits to it,
in the case of positively homogeneous maps acting on the interior of a cone,
or of additively homogeneous maps acting on an AM-space with unit.
This is motivated
in particular by the analysis of dynamic programming operators
(Shapley operators) of zero-sum stochastic games.
\end{abstract}
\maketitle

\section{Introduction}
Nonlinear maps acting on a subset of a Banach space,
that are nonexpansive either
in the norm of the space, or in metrics inherited from 
a convex cone, like Hilbert's or Thompson's metric,
arise in a number of fields,
including population dynamics~\cite{perthame},
entropy maximization and scaling problems~\cite{menonSchneider69,borweinlewisnussbaum}, renormalization operators and fractal diffusions~\cite{sabot,metz}, mathematical economy~\cite{morishima}, 
mathematical biology~\cite{angeli08}, optimal filtering~\cite{bougerol}, 
optimal control~\cite{crandall,spectral2} and 
zero-sum games~\cite{kolokoltsov92,FilarVrieze,rosenbergsorin,neymansurv,AGGut10}.
In particular, the dynamic programming operators, known
as Bellman operators in control theory, or as 
Shapley operators in game theory,
turn out to be, under standard assumptions,
sup-norm nonexpansive maps defined on a space
of continuous functions.

Such nonexpansive maps include as a special case the nonnegative matrices and positive
linear operators arising in Perron-Frobenius theory. A number
of works, including~\cite{krut48,birkhoff57,birkhoff62,hopf,krasnoselskii,birkhoff67,potter,bushell73,bushell86,nussbaum88,nussbaummemoir89,krause,nussbaumlunelmemoir,arxiv1,AGLN,hujiang10a,hujiang10b,GV10,lemmensnussbaum}, have dealt with the extension of Perron-Frobenius theory to the nonlinear case. 

In particular, a central problem is to give conditions allowing one to check whether a given fixed point is globally unique,
and whether all the orbits of the maps converge to it.

This problem was considered by Nussbaum in~\cite{nussbaum88},
who gave general conditions in the case of differentiable
maps. Several results in~\cite{nussbaum88},
show, in various settings, that
a given fixed point of the map is globally unique
as soon as the differential of the map at this fixed point
does not have any nontrivial fixed point, provided some mild compactness
condition is satisfied. Similarly, the global geometric convergence
of all the orbits of the map to this fixed point is guaranteed
if the spectral radius of the same differential is strictly less than
one. Analogous results are derived in~\cite{nussbaum88}
for nonlinear eigenvectors. 

In a number of applications, including the
ones arising from control and games, differentiability assumptions
turn out to be too restrictive, since Bellman
or Shapley operators are typically nonsmooth
(smoothness being related to the uniqueness
of the optimal action). However,
such operators often satisfy a weaker condition, {\em semidifferentiability},
which was first introduced by Penot~\cite{penot}, and which has become a basic
notion in variational analysis, 
see in particular the book by Rockafellar and Wets~\cite{rock98}.

The case of nondifferentiable convex Shapley operators
(corresponding to one player stochastic games) was studied by Akian and Gaubert
in~\cite{spectral2}. It was shown there that the dimension
of the fixed point set can be bounded 
by considering the subdifferential of the operator
at any fixed point (in particular, the uniqueness
of this fixed point, perhaps up to an additive
constant, can be guaranteed by considering
this subdifferential), and that the
asymptotic behavior of the orbits, could also
be inferred from this subdifferential.

In this paper, we show that the general principle,
developed in~\cite{nussbaum88} in the differentiable
case and in~\cite{spectral2} in the convex nondifferentiable
case carries over to the nonconvex semidifferentiable case:
the global uniqueness of the fixed point of a nonexpansive map, and the geometric convergence of the orbits to it, can be guaranteed
by considering only the infinitesimal behavior of the map at the fixed
point.

Our main results are Theorem~\ref{theo0}, dealing with the uniqueness
of the fixed point, and Theorem~\ref{th-geom}, dealing with the geometric
convergence of the orbits. We also derive in~\S\ref{sec-eig-nonexp}
similar results for the uniqueness of the normalized eigenvector, and
for the convergence of the orbits, in the case of homogeneous maps
acting on the interior of a normal cone and nonexpansive in Hilbert's or Thompson's metric. Note that the spectral
radius arising in the differentiable case
is replaced by Bonsall's cone spectral radius in the semidifferentiable case.
Moreover, to state the uniqueness result, a nonlinear Fredholm-type compactness
is required. In this way, we generalize to the case
of semidifferentiable maps 
(Theorem~\ref{theo01}, and Corollaries~\ref{theo1} and~\ref{cor-add})
some of the main results of~\cite{nussbaum88}, including Theorem~2.5 there.
Note that the present results are already new
in the finite dimensional case. Then, the technical compactness conditions
are trivially satisfied.

We remark in passing that there are interesting classes of cone maps $f:C\to C$
 which are differentiable on the interior of the cone $C$ but only semidifferentiable on the boundary of $C$.  See, for example, the class $\mathcal{M}_{-}$
treated in~\cite{nussbaummemoir89}. Properties equivalent to semidifferentiability (see e.g.\ Theorem~3.1 there) have been used there to prove existence of eigenvectors in the interior of $C$ for some such maps $f$. However, we shall not pursue this line of applications here.

As an illustration of our results, we analyze a simple zero-sum
stochastic repeated game (\S\ref{subsec-games}). (The application
to zero-sum games will be discussed in more detail 
in a further work.) We note in this respect that the
spectral radius of the semidifferential of a Shapley
operator can be computed using the results of the companion work~\cite{AGN-collatz}, which yields explicit formul\ae\ for the geometric contraction rate.

The paper is organized as follows. Some basic definitions and results concerning
convex cones, Hilbert's and Thompson's metric, are recalled in~\S\ref{sec-intro-hilbert}. Preliminary results concerning semidifferentiable
maps are established in~\S\ref{sec-semidifferentiable}.
Then, in~\S\ref{sect-semider}, we examine the semidifferentials
of order preserving or nonexpansive maps. The needed
results concerning the different notions of nonlinear spectral radius, 
and in particular Bonsall's cone spectral radius, are recalled
in~\S\ref{sec-spec}
(these results are
essentially taken from~\cite{Nuss-Mallet} and~\cite{AGN-collatz}).
Then, the main results of the paper, concerning
the uniqueness of the fixed point, and the geometric convergence
of the orbits, are established in~\S\ref{sec-main}. The corollaries of these
results concerning nonlinear eigenvectors of homogeneous
maps leaving invariant the interior of a cone
are derived in~\S\ref{sec-eig-nonexp}.
Note that the proofs in~\S\ref{sec-main} and~\ref{sec-eig-nonexp} are based on
metric fixed point technique (approximation of the map by a strictly contracting map). In particular, in the
case of Hilbert's or Thompson's metric, the need to work with a complete
metric space limits the scope of these results to {\em normal} cones. 
Hence, alternative uniqueness results, using degree theory arguments, 
which are valid more generally in proper (closed, convex, and pointed) cones, 
are presented in~\S\ref{sec-fix-proper}. Then, in~\S\ref{sec-eig-nonexp-diff}, the results are
specialized to differentiable maps, and compared with the ones
of~\cite{nussbaum88}. Finally, in~\S\ref{sec-order}, the present results are adapted
to the additive setting (order preserving and additively homogeneous
maps acting on an AM-space with unit, or equivalently, on a space of
continuous functions $\continuous(K)$ for some
compact set $K$). They are illustrated in \S\ref{subsec-games} by the analysis
of an example of zero-sum stochastic game (a variant of the Richman~\cite{loeb} or stochastic tug-of-war games~\cite{peres}).

\section{Preliminary results about Hilbert's and Thompson's metric}
\label{sec-intro-hilbert}
In this section, we recall classical notions
about cones, and recall or establish
results which will be used in the following sections.
See \cite[Chapter 1]{nussbaum88} and \cite[Section 1]{nussbaum94}
for more background.

\subsection{Hilbert's and Thompson's metric}\label{sec21}
In this paper,
a subset $C$ of a real vector space $X$ 
is called a \NEW{cone} (with vertex 0) if 
$tC:=\set{tx}{x\in C} \subset C$ for all $t\geq 0$. 
If $f$ is a map from a cone $C$ of a vector space $X$ to a cone
$C'$ of a vector space $Y$,
we shall say that $f$ is (positively) \NEW{homogeneous} (of degree $1$)
if $f(t y)=t f(y)$, for all $t>0$ and $y\in C$.
We say that the cone $C$ is \NEW{pointed}
if $C \cap (-C) = \{0\}$.
A convex pointed cone $C$ of $X$
induces on $X$ a partial ordering $\leqc$, which
is defined by $x\leqc y$ iff $y - x \in C$. 
If $C$ is obvious, we shall write $\leq$ instead of $\leqc$.
When $X$ is a topological vector space, we say that
$C$ is \NEW{proper} if it is closed convex and pointed.
Note that in \cite{nussbaum88}, 
a {\em cone} is by definition what we call here a proper cone.
We next recall the definition of Hilbert's and Thompson's
metrics associated to a proper cone $C$ of 
a topological vector space $X$.

Let $x\in C\setminus\{0\}$ and $y\in X$. 
We define
$\M(y/x)$ by
\begin{align}\label{def-M}
\M(y/x) := \inf \set{b \in \R}{ y \leq bx }\enspace ,
\end{align}
where the infimum of the emptyset is by definition equal to
$+\infty$. 
Similarly, 
we define $\m(y/x)$ by
\begin{align}\label{def-m}
\m(y/x) := \sup \set{a \in \R}{ax \leq y }\enspace,
\end{align}
where the supremum of the emptyset is by definition equal to
$-\infty$. 
We have $\m(y/x)=-\M(-y/x)$ and if in addition 
$y\in C\setminus\{0\}$, $\m(y/x)=1/ \M(x/y)$ (with $1/(+\infty)=0$).
Since $C$ is pointed and closed, we have $\M(y/x)\in \R\cup\{+\infty\}$, and
$y \leq \M(y/x) x$ as soon as $\M(y/x)<+\infty$. 
Symmetrically $m(y/x)\in \R\cup\{-\infty\}$ and $\m(y/x) x \leq y$, as soon as 
$m(y/x)>-\infty$. 

We shall say that two elements $x$ and $y$ in $C$ are \NEW{comparable}
and write $x \backsim y$ 
if there exist positive constants $a > 0$ and $b > 0$ such that
$ax \leq y \leq bx$. 
If $x,y \in C\setminus \{0\}$ are comparable,
we define
\begin{align*}
d(x,y) = &  \log \M(y/x) -\log \m(y/x) \enspace ,
\\
\tho(x,y) = &  \log \M(y/x) \vee -\log \m(y/x)\enspace .
\end{align*}
We also define $d(0,0)=\tho(0,0)=0$.

If $u \in C$, we define
\[
C_u = \{x \in C \mid x \backsim u\}\enspace .
\]
If $C$ has nonempty interior $\Cint$, and $u \in \Cint$, 
then $C_u = \Cint$. To see this,
take a neighborhood $V$ of $0$
such that $u+V\subset C$. If $x\in \Cint$, we
can find a neighborhood $W$ of $0$ such that $x+W\subset C$.
Let $a,b>0$ be such that $-a u \in W$
and $-b x\in V$. Then, $u-b x \in C$ and $x-a u\in C$,
and so $b x\leq u\leq a^{-1} x$, showing
that $\Cint \subset C_u$. Conversely, if $x\in C_u$,
we have $u\leq bx$ for some $b>0$.
It follows that $x+b^{-1}V= (x-b^{-1}u)+b^{-1}(u+V)\subset
C+C=C$, showing that $x\in \Cint$.

In general, $C_u \cup \{0\}$ is a pointed convex cone
but $C_u \cup \{0\}$ is not closed. 
The map $\tho$ 
is a metric on $C_u$,
called \NEW{Thompson's metric}.
The map $d$ 
is called the \NEW{Hilbert's projective metric}
on $C_u$. The term ``projective metric''
is justified by the following properties:
for all $x,y,z \in C_u$, $d(x,z) \leq d(x,y) + d(y,z)$, 
$d(x,y) = d(y,x)\geq 0$ and $d(x,y) = 0$
iff $y = \lambda x$ for some $\lambda > 0$.

{From} now, we will assume that $X=(X,\|\cdot\|)$ is a Banach space.
We denote by $X^*$ the space of continuous linear forms over $X$,
and by $C^*:=\set{\psi\in X^*}{\psi(x)\geq 0\; \forall x\in C}$ 
the \NEW{dual cone} of $C$.
If $f$ is a map between two ordered sets $(D,\leq)$ and $(D',\leq)$, 
we shall say that $f$ is \NEW{order preserving} if $f(x)\leq f(y)$ for all   
$x,y\in D$ such that $x\leq y$. Then, any element of $C^*$ is a
homogeneous and order preserving  map from $(X,\leq_C)$ to $[0,+\infty)$.
Since $C$ is proper, the Hahn-Banach theorem implies 
that for all $u\in C\setminus\{0\}$, 
there exists $\psi\in C^*$ such that $\psi(u)>0$.
For such a $\psi$, we have $\psi(x)>0$ for all $x\in C_u$.
More generally, if $q:C_u\to (0,+\infty)$ is homogeneous and
order preserving, we denote
\begin{align}
\label{e-def-sigmau}
\Sigma_u=\set{x\in C_u}{ q(x)=q(u)}
\enspace .
\end{align}
Then, $d$ and $\tho$ are equivalent metrics on $\Sigma_u$.
Indeed, as shown in~\cite[Remark~1.3, p. 15]{nussbaum88}:
\begin{align}
\frac 1 2 d(x,y)\leq \tho(x,y) \leq d(x,y),\; \forall x,y\in \Sigma_u
\enspace.
\label{e-equivm}
\end{align}
More precisely:
\begin{align}
1\leq \M(y/x) \leq e^{d(x,y)}\quad \forall x,y\in \Sigma_u\enspace.
\label{mleqd}
\end{align}
To see this, let us apply $q$ to the inequality $y\leq \M(y/x)\; x$.
Using that $q$ is order preserving and homogeneous, and
that $q(x)=q(y)$, we get $\M(y/x)\geq 1$.
By symmetry, $\M(x/y)\geq 1$, hence 
$\log \M(y/x)= d(x,y)-\log\M(x/y)\leq d(x,y)$.

We say that a cone $C$ is 
\NEW{normal} if $C$ is proper and 
there exists a constant $M$ such that $\|x\| \leq M \|y\|$ 
whenever $0 \leq x \leq y$. Every proper cone $C$ in a
finite dimensional Banach space $(X,\|\cdot\|)$ is necessarily normal.
We shall need the following result of Thompson.
\begin{proposition}[{\cite[Lemma~3]{thompson}}]\label{prop-thompson}
Let $C$ be a normal cone in a Banach space $(X,\|\cdot\|)$.
For all $u\in C\setminus\{0\}$, $(C_u,\tho)$ 
is a complete metric space.
\end{proposition}
The following result follows from a general result
of Zabre{\u\i}ko, Krasnosel{$'$}ski{\u\i} and Pokorny{\u\i}~%
\cite{zabreiko} (see~\cite[Theorem~1.2 and Remarks~1.1 and 1.3]{nussbaum88}
and a previous result of Birkhoff~\cite{birkhoff62}).
When $q\in C^*$, it follows from 
Proposition~\ref{prop-thompson},~\eqref{e-equivm} and the property that
$\Sigma_u$ is closed in the topology of the Thompson's metric $\tho$.
\begin{proposition}\label{prop-birkhoff}
Let $C$ be a normal cone in a Banach space $(X,\|\cdot\|)$.
Let $u\in C\setminus\{0\}$ and $q:C_u\to (0,+\infty)$,
which is homogeneous and order preserving with respect to $C$. 
Let $\Sigma_u=\set{x\in C_u}{ q(x)=1}$. Then,  
$(\Sigma_u,d)$ and $(\Sigma_u,\tho)$ are complete metric spaces. 
\end{proposition} 

\subsection{The local Banach space $X_u$}
Given $u \in C\setminus \{0\}$, we define the linear space 
$X_u=\set{x \in X}{\exists a > 0,\; -au \leq x \leq au}$.
We will show (in Lemma~\ref{lem-allineq} and Proposition~\ref{prop-normequiv} below) 
that $X_u$ is equipped with
a norm $\|\cdot\|_u$, and a seminorm $\omega_u$,
such that $\tho$ and $d$ behave locally, near $u$,
as $\|\cdot\|_u$ and $\omega_u$, respectively.

Let $M$ and $m$ be defined as in~\eqref{def-M} and \eqref{def-m}.
We equip $X_u$ with the norm:
\begin{equation}\label{eq6}
\|x\|_u = \stf{u}{x}\vee -\sbf{u}{x} = 
\inf \set{a > 0}{- au \leq x \leq au} \enspace .
\end{equation}
We also define the {\em oscillation} of $x\in X_u$:
\begin{align}
\label{def-omega}
\omega_u(x) = \stf{u}{x}-\sbf{u}{x} = \inf\set{b-a}{au\leq x\leq bu} \enspace.
\end{align}
The map $\omega_u: x\mapsto \omega_u(x)$ is a 
seminorm on $X_u$, which satisfies 
$\omega_u(x)=0\iff x\in \R u$. Note also
that $\omega_u(x+au)=\omega_u(x)$ for all $x\in X_u$ and $a\in\mathbb{R}$.
Since
$\stf{u}{x}\vee -\sbf{u}{x} \geq \frac 1 2 (\stf{u}{x}-\sbf{u}{x})$,
we get
\begin{align}
\label{equiv-omega0}
\frac{1}{2} \omega_u(x)\leq \|x\|_u,
\qquad \forall x\in X_u \enspace .
\end{align}
Moreover, when $\psi$ is an element of $C^*$ and $\psi(u)=1$,
$\psi(x)=0$ and $x\leq au $ implies $0\leq a$, 
so that $\stf{u}{x} \geq 0$, and dually, $\sbf{u}{x} \leq 0$.
Hence, $\stf{u}{x} \vee -\sbf{u}{x}\leq \stf{u}{x}- \sbf{u}{x}$, 
which gives
\begin{align}
\label{equiv-omega1}
\|x\|_u \leq \omega_u(x)\quad \forall x\in 
X_u,\text{ such that }\psi(x)=0\enspace .
\end{align}
Gathering~\eqref{equiv-omega0}
and~\eqref{equiv-omega1}, 
we see that the restriction of $\omega_u$ to the subspace
$\set{x\in X_u}{\psi(x)=0}$
of $X_u$ is a norm equivalent to $\|\cdot\|_u$.

We shall need the following variants or refinements of some inequalities
of~\cite[Proposition~1.1]{nussbaum94}. 
\begin{lemma}\label{lem-allineq}
Let $C$ be a proper cone, and $u\in C\setminus \{0\}$. For 
all $x,y\in C_u$, we have:
\begin{align}
\|x -y\|_u &\leq (e^{\tho(x,y)}-1)e^{\tho(x,u)\wedge \tho(y,u)} \enspace,\label{e-e1}\\
\tho(x,y) &\leq \log (1+ \|x-y\|_u\,  e^{\tho(x,u)\vee \tho(y,u)}) \enspace. \label{e-e2}
\end{align}
Moreover, 
defining $\Sigma_u$ as in~\eqref{e-def-sigmau},
with $q:=\psi\in C^*$ and $\psi(u)=1$, we get that
for all $x,y\in \Sigma_u$, 
\begin{align}
\label{e-f1}
\omega_u(x-y) &\leq
(e^{d(x,y)}-1)e^{\tho(x,u)\wedge \tho(y,u)} \enspace,
 \\\label{e-f2}
d(x,y)&\leq \omega_u(x-y) e^{\tho(x,u)\vee \tho(y,u)}
\enspace .
\end{align}
\end{lemma}
\begin{proof}
We first prove~\eqref{e-e1}.
By definition of $\tho$, we can write
\begin{align}
\label{ed2} e^{-\tho(y,u)}u&\leq y\leq e^{\tho(y,u)}u\\
\label{ed3} e^{-\tho(x,y)}y&\leq x\leq e^{\tho(x,y)}y \enspace .
\end{align}
Using~\eqref{ed3} and the second inequality in~\eqref{ed2}, 
we get
\begin{align*}
& (e^{-\tho(x,y)}-1)e^{\tho(y,u)}u
\leq (e^{-\tho(x,y)}-1)y  \\
& \qquad \leq x-y 
\leq (e^{\tho(x,y)}-1)y\leq 
(e^{\tho(x,y)}-1)e^{\tho(y,u)}u \enspace. \nonumber
\end{align*}
Hence, 
\begin{align*}
\|x-y\|_u\leq
\big((e^{\tho(x,y)}-1)\vee (1-e^{-\tho(x,y)})\big)e^{\tho(y,u)} 
=(e^{\tho(x,y)}-1)e^{\tho(y,u)} 
\enspace .
\end{align*}
Together with 
the symmetrical
inequality obtained by exchanging $x$ and $y$, 
this yields~\eqref{e-e1}.

We next prove~\eqref{e-e2}. By definition of $\|\cdot \|_u$, we can write
\begin{align*}
 x -y\leq \|x-y\|_u u 
\end{align*}
for all $x,y\in X_u$. If $x,y\in C_u$, using the first inequality in~\eqref{ed2}, we deduce
\begin{align*}
x\leq y+ \|x-y\|_u u\leq (1+\|x-y\|_u e^{\tho(y,u)})y \enspace,
\end{align*}
and by symmetry,
\begin{align*}
y\leq  (1+\|x-y\|_u e^{\tho(x,u)})x \enspace.
\end{align*}
It follows that 
\begin{align*}
\tho(x,y)&\leq 
\log(1+\|x-y\|_u  e^{\tho(y,u)})
\vee 
\log(1+\|x-y\|_u  e^{\tho(x,u)})\\
&= \log(1+ \|x-y\|_u e^{\tho(y,u)\vee \tho(x,u)}) \enspace,
\end{align*}
which shows~\eqref{e-e2}. 

We now prove~\eqref{e-f1}. By definition of $d$, we can write
\begin{align}
d(x,y) &= \log \beta-\log \alpha, \;\mrm{with}\;
\alpha,\beta>0 \;\mrm{and}\\
\label{ed10}
\alpha y&\leq x\leq \beta y\enspace .
\end{align}
When $x,y\in \Sigma_u$, applying $\psi$ to~\eqref{ed10},
we get
\(\alpha \leq 1\leq \beta\).
Combining the second inequality in~\eqref{ed10} 
with $\beta-1\geq 0$ and the second 
inequality in~\eqref{ed2}, we get
\begin{align}\label{ed12}
x-y\leq (\beta-1)y \leq (\beta-1)e^{\tho(y,u)} u \enspace. 
\end{align}
Combining the first inequality in~\eqref{ed10}
with $\alpha-1\leq 0$ and the second inequality in~\eqref{ed2},
we get
\begin{align}\label{ed13}
x-y\geq (\alpha-1)y \geq (\alpha-1)e^{\tho(y,u)} u \enspace. 
\end{align}
Gathering~\eqref{ed12} and~\eqref{ed13}, we
get 
\[
\omega_u (x-y)\leq (\beta-\alpha) e^{\tho(y,u)} \enspace .
\]
By symmetry, 
$\omega_u (x-y)\leq (\beta-\alpha) e^{\tho(x,u)}$,
so that 
\begin{align*}
\omega_u (x-y)\leq (\beta-\alpha) e^{\tho(x,u)\wedge \tho(y,u)} \enspace .
\end{align*}
Since 
$\beta\geq 1\geq \alpha$, we get
\begin{align*}
\omega_u (x-y)\leq (\frac{\beta}{\alpha}-1)
 e^{\tho(x,u)\wedge \tho(y,u)} = (e^{d(x,y)}-1)
 e^{\tho(x,u)\wedge \tho(y,u)} \enspace ,
\end{align*}
which shows~\eqref{e-f1}.

We finally prove~\eqref{e-f2}. By definition of $\omega_u$, 
we have
\begin{align}
\label{ed16}
au\leq x-y\leq bu
\end{align}
for some $a,b\in \R$, with $b-a=\omega_u(x-y)$.
Applying $\psi$ to~\eqref{ed16}, we get $a\leq 0\leq b$.
Combining the second inequality in~\eqref{ed16} with
the first inequality in~\eqref{ed2}, we get
\begin{align}\label{ed17}
x\leq y+bu\leq (1+be^{\tho(y,u)}) y
\end{align}
and by symmetry (\eqref{ed16} is equivalent to $-bu\leq y-x\leq -au$)
\begin{align}\label{ed18}
y\leq x-au\leq (1-ae^{\tho(x,u)}) x
\end{align}
Gathering~\eqref{ed17} and~\eqref{ed18}, we get
\begin{align*}
d(x,y)\leq \log (1+be^{\tho(y,u)})+ \log(1- ae^{\tho(x,u)}) \enspace .
\end{align*}
Since $\log (1+x)\leq x$ for all $x\geq 0$ and $b\geq 0\geq a$,
we get
\begin{align*}
d(x,y)\leq be^{\tho(y,u)}- ae^{\tho(x,u)}
\leq (b-a) e^{\tho(x,u)\vee \tho(y,u)} \enspace ,
\end{align*}
which shows~\eqref{e-f2}.
\end{proof}
\begin{proposition}\label{prop-normequiv}
Let $C$ be a normal cone in a Banach space $(X, \| \cdot \|)$.
Consider $u\in C\setminus\{0\}$, and $\psi\in C^*$
with $\psi(u)>0$.
Then $(X_u, \| \cdot \|_u)$ is a Banach space.
Moreover, $\|\cdot\|_u$ and $\tho$ induce
the same topology on $C_u$, and
\begin{align}
\|x-y\|_u\thicksim \tho(x,y),\;\mrm{when}\; x,y\to u
\;\mrm{in}\; (C_u,\tho) \enspace ,\label{e-normequivdbar}
\end{align}
where by~\eqref{e-normequivdbar}, we mean that for all $\lambda>1$, there
exists a neighborhood $U$ of $u$ in $(C_u,\tho)$ such that 
\[
\lambda^{-1} \tho(x,y)\leq \|x-y\|_u \leq \lambda \tho(x,y),\;\mrm{for all}\; x,y\in U \enspace.
\]
We also have
\begin{align}
\omega_u(x-y) \thicksim d(x,y),\; \mrm{when}\; x,y\to u
\;\mrm{in}\; (\Sigma_u,\tho) \enspace .
\label{e-normequivosc}
\end{align}
Also $C \cap X_u$ is a normal cone in $(X_u, \| \cdot \|_u)$ and 
has nonempty interior given by $C_u$. 
If $C$ has nonempty interior in $(X, \| \cdot \|)$ and $u \in \Cint$, 
then
$X_u = X$ and $\| \cdot \|$ and $\| \cdot \|_u$ are equivalent norms on $X$.
\end{proposition}
\begin{proof}
It is proved in~\cite[Proposition~1.1]{nussbaum94}
that $\|\cdot\|_u$ and $\tho$ induce the same
topology on $C_u$ (this indeed follows from~\eqref{e-e1} and~\eqref{e-e2}).
The fact that $X_u$ is a complete
metric space when $C$ is normal is a result of 
Zabre{\u\i}ko, Krasnosel{$'$}ski{\u\i}, and Pokorny{\u\i}~\cite{zabreiko}
(see also Theorem~1.2 and Remark~1.1 in ~\cite{nussbaum88}). 
This can also be derived from the completeness
of $(C_u,\tho)$ (Proposition~\ref{prop-thompson} above)
and of the equivalence~\eqref{e-normequivdbar}, which
follows readily from~\eqref{e-e1} and~\eqref{e-e2}.
The equivalence~\eqref{e-normequivosc} follows
readily from~\eqref{e-f1} and~\eqref{e-f2}.
The fact that $C\cap X_u$ is a normal cone in $(X_u, \| \cdot \|_u)$
and has $C_u$ as interior is remarked
in~\cite[Remark~1.1]{nussbaum94}. 
Indeed, since $C$ is a proper cone and
$X_u$ is a vector space, $C\cap X_u$ is a proper cone.
Moreover, from~\eqref{eq6}, $\|\cdot\|_u$ is order preserving on
$C\cap X_u$, hence $C\cap X_u$ is normal.
Also, by definition of $C_u$ and \eqref{eq6}, we get that $C_u$
is open (if $a u\leq x\leq b u$ with $a,b>0$, then 
$B(x,\frac{a}{2})\subset C_u$) and equal to the interior of
$C\cap X_u$ (if $B(x,\epsilon)\subset C\cap X_u$, then 
$\epsilon u \leq x\leq \|x\|_u u$, and $x\in C_u$).
Finally, it is proved in~\cite[Proposition~1.1]{nussbaum94}
that $\|\cdot\|_u$ and $\|\cdot\|$
are equivalent norms on $X=X_u$ when $u$ is in the
interior of $C$.
\end{proof}

We say that a cone $C$ of a Banach space $(X,\|\cdot\|)$ is
\NEW{reproducing} if $X=C-C:=\set{x-y}{x, y\in C}$.
It is known that a cone $C$ in a Banach space $(X, \| \cdot \|)$
with nonempty interior is reproducing (see for 
instance~\citespectral{Proposition}{sp-normalempty}).

\begin{remark}
It is shown in~\cite{nussbaum94} that the Thompson's metric 
coincides with the Finsler metric
arising when considering the interior of the cone $C$ as a manifold
and equipping the tangent space at point $u$ with the local norm $\|\cdot \|_u$.
The Hilbert's metric arises in a similar way, when equipping
the tangent space with the seminorm $\omega_u$.
\end{remark}
\subsection{AM-spaces with units}
\label{am-spaces-def}
Recall that an ordered set $(X,\leq)$ is a \NEW{sup-semilattice} 
(resp.\ \NEW{inf-semilattice}) if 
any nonempty finite subset $F$ of $X$ admits a least upper bound 
(resp.\ greatest lower bound) in $X$,
denoted $\vee F$ (resp.\ $\wedge F$). We shall also use the infix notation
$x\vee y= \vee\{x,y\}$ and $x\wedge y=\wedge\{x,y\}$. 
We say that $X$ is a \NEW{lattice}
when it is both a sup-semilattice and an inf-semilattice.
A \NEW{Banach lattice}
is a Banach space $(X,\|\cdot\|)$ 
equipped with an order relation, $\leq$, 
such that $(X,\leq)$ is a lattice, $x\leq y\Rightarrow x+z\leq y+z$ 
and $\lambda x\leq \lambda y$, for all  $x,y,z\in X$ and $\lambda\geq 0$,
and $|x|\leq |y|\Rightarrow \|x\|\leq \|y\|$, 
for all $x,y\in X$, where $|x|:=x\vee (-x)$.
Note that if $X$ is a Banach lattice,
then the lattice operations are continuous
in the norm topology~\cite[Prop.~1.1.6]{meyernieberg}. Moreover,
$X^+:=\set{x\in X}{0\leq x}$ is closed
and it is a reproducing normal cone.
An \NEW{AM-space with unit}
is a Banach lattice $(X,\|\cdot\|)$
such that $\|x\vee y\|=\|x\|\vee \|y\|$ for all
$x,y\in X^+$,
and such that the closed unit ball of $X$ has a maximal element $e$,
which is called the {\em  unit}
(see for instance~\cite{schaefer} or \cite{aliprantis} for definitions 
and results about AM-spaces).
Equivalently, $X$ is an AM-space with unit, if $X$ is
a Banach lattice equipped with a distinguished
element $e\geq 0$ (the unit), which is such that 
$\|x\|=\inf\set{a\geq 0}{ -a e \leq x\leq a e }$
(see~\cite[Section 8.4]{aliprantis}, note however that in this
reference, a slightly different definition of a unit is used, but it
is shown that there exists another norm equivalent to $\|\cdot\|$
satisfying all the above conditions).
The fundamental example of AM-space with unit is given by the
space $\continuous(K)$  of continuous functions on a compact set $K$,
equipped with the sup-norm, $\|\cdot\|_\infty$, the
pointwise order $\leq$ and the unit $\mathbf{1}$, where
$\mathbf{1}$ is the constant function $K\to \R,\, t\mapsto 1$.
In fact, the Kakutani-Krein theorem (\cite[Theorem~8.29]{aliprantis}
or \cite[Chapter II, Theorem~7.4]{schaefer}) shows 
that an AM-space with unit is isomorphic as an AM-space with unit
(that is lattice isomorphic and isometric)
to $\continuous(K)$, for some compact, Hausdorff set $K$.

Any AM-space with unit $X$ can be put in isometric correspondence with
the interior of a normal cone equipped with Thompson's metric, 
thanks to the following construction.
Let $\imath: X\to \continuous(K)$ be
the isomorphism (of AM-spaces with unit, where $K$ is  a compact set) 
given by Kakutani-Krein theorem. Consider $C=\continuous^+(K)$ the set of
 nonnegative continuous functions on $K$.
Then, $C$ is a normal cone, and the order $\leq_C$ associated to $C$
is nothing but the pointwise order in $\continuous(K)$. The interior $\Cint$ of
$C$ is the set of continuous functions $K\to \R$ that are bounded from below
by a positive constant, or equivalently (since $K$ is compact),
the set of continuous functions $K\to \R$ that are positive everywhere.
Consider the map $\log: \Cint \to \continuous(K)$, 
which sends the continuous positive map $t\in K \mapsto y(t)$ 
to the continuous map $t\in K \mapsto \log y(t)$, and denote by
$\exp=\log^{-1}$ its inverse. Then, the map
$\imath^{-1} \comp \log : \Cint\to X$ is an isometry between
$\Cint$ endowed with the Thompson's metric and $X$ endowed with the
distance associated with the norm $\|\cdot\|$.

\section{Semidifferentiable maps}\label{sec-semidifferentiable}
In this section, we consider (with a slight modification) the notion
of semidifferentiability introduced by Penot in~\cite{penot} in the case
of maps defined on Banach spaces. We also
refer the reader to~\cite{rock98} for the case of maps defined on $\R^n$.
We establish here some properties of semidifferentiable maps defined on
normed vector  spaces,
which will be needed in the proof of the main results
and in the applications.
\subsection{Semidifferentiable maps on normed vector spaces}
\label{sec-semidiff}

Let $(X,\|\cdot\|)$ and $(Y,\|\cdot\|)$ be normed vector spaces,
$G$ be a subset of $X$, and $v\in G$.
Recall that a map $h$ from a cone $C$ of $X$ to $Y$ 
is \new{homogeneous} if $f(t y)=t f(y)$, for all $t>0$ and $y\in C$.
We say that a map $f : G \rightarrow Y$ is 
\NEW{semidifferentiable} at $v$ with respect to 
a cone $C\subset X$ if 
$G$ contains a neighborhood of $v$ in
$v+C:=\set{v+x}{x\in C}$ (that is, if there exists
$\epsilon>0$ such that all the elements of the
form $v+x$ with $x\in C$ and  $\|x\|\leq \epsilon$ belong to $G$),
and if there exists a continuous (positively) homogeneous map 
$h : C \rightarrow Y$ such that
\begin{equation}\label{eq1'}
f(v+x) = f(v) + h(x) + o(\|x\|), \text{ when } x \rightarrow 0, \; x \in C
\enspace .
\end{equation}  
If $C=X$ (and $v$ is in the interior of $G$), we say shortly that
$f$ is semidifferentiable at $v$.
The map $h$, if it exists, is unique, since 
then, for all $x\in C$, the classical (one sided) \NEW{directional derivative}:
\begin{equation}\label{eq2'}
f'_v(x) :=\lim_{t \rightarrow 0^+} 
\frac{f(v + tx) - f(v)}{t} 
\end{equation}
exists and coincides with $h(x)$. In fact, 
the definition \eqref{eq2'} of directional derivatives is obtained by 
specializing \eqref{eq1'} to a cone of the form $C=\set{ tx}{t \geq 0}$
 and the classical definition of Frechet derivatives is obtained by 
specializing~\eqref{eq1'} to $C = X$, and requiring $h$ to be linear.
We call $h$ the \NEW{semidifferential} of $f$ at $v$ with respect to $C$,
and we denote it by $f'_v$.
Note that in general the set $D$ of $x\in X$ such that $f'_v(x)$ exists
is a cone which may contain strictly $C$, for the map
$h:D\to Y,\; x\mapsto f'_v(x)$ is not necessarily continuous on $D$
or it does not satisfy Equation~\eqref{eq1'} on $D$.
However we shall only use the notation $f'_v$ for the restriction
of $x\mapsto f'_v(x)$ to a cone $C$ satisfying the above conditions.

The following characterization of 
semidifferentiable maps, which is a mere
rephrasing of property~\eqref{eq1'}, 
illuminates the requirements
that semidifferentiability adds to the existence
of directional derivatives.
\begin{lemma}\label{lem2}
Let $(X,\|\cdot\|)$ and $(Y,\|\cdot\|)$ be normed vector spaces, 
$G$ be a subset of $X$, $v\in G$, 
and $C$ be a cone of $X$ such that $G$ contains a neighborhood of
$v$ in $v+C$. 
Then, a map $f:G\to Y$ is semidifferentiable at $v$ with respect to $C$ if,
and only if,
\begin{subequations}
\label{eq3'}
\begin{align}
\label{eq3'a}
 \frac{f(v + tx) - f(v)}{t} \rightarrow f'_v(x) \text{ when } t \rightarrow 0^+, \\
 \text{uniformly for $x$ in bounded sets of $C$, } \label{eq3'aa}\\
\nonumber
\text{and the map $C \rightarrow Y,\; x\mapsto f'_v(x)$ is continuous.}
\end{align}
\end{subequations}
\end{lemma}

The following lemma shows in particular
that when $f$ is locally Lipschitz continuous
and $X$ is finite dimensional, the last two
properties are implied by the existence of directional
derivatives.
\begin{lemma}\label{lem2b}
Let $(X,\|\cdot\|)$ and $(Y,\|\cdot\|)$ be normed vector spaces,
$G$ be a subset of $X$, $v\in G$, 
and $C$ be a cone of $X$ such that $G$ contains a neighborhood of
$v$ in $v+C$. 
Assume that $f : G \rightarrow Y$ is Lipschitz continuous in a
neighborhood of $v$,
and that $f$ has directional derivatives $f'_v(x)$
at $v$ with respect to all $x \in C$.
Then $C \rightarrow Y,\; x\mapsto f'_v(x)$ is Lipschitz continuous
with same Lipschitz constant as $f$. Moreover,
\begin{subequations}\label{eq4'}
\begin{align}
\frac{f(v + tx) - f(v)}{t} \rightarrow f'_v (x) \text{ when } t \rightarrow 0^+, \\
\text{uniformly for $x$ in compact sets of $C$,}
\end{align}
\end{subequations}
and
\begin{equation}
\label{eq5'}
\lim_{\substack{t \rightarrow 0^+\\ x' \rightarrow x ,\; x' \in C}}
\frac{f(v + tx') - f(v)}{t} = f'_v (x) \quad\forall x\in C\enspace.
\end{equation}
In particular, if $X$ is finite dimensional and $C$ is closed in $X$,
then $f$ is semidifferentiable at $v$ with respect to $C$.
\end{lemma}
\begin{proof}
Since $G$ contains a neighborhood of $v$ in $v+C$ and $f$ is
Lipschitz continuous in a neighborhood of $v$, we can find $\epsilon > 0$ and 
$M \geq 0$ such that
$$(z,z'\in C\mrm{ and } \|z\|,\|z'\|\leq \epsilon)
\implies \|f(v+z) - f(v + z')\| \leq M \|z - z'\|
\enspace .$$
Let $K$ denote a compact subset of $C$, let 
$R = \max \set{\|x\|}{x \in K},$ 
and, for $0 < t \leq R^{-1} \epsilon$,
consider the map $g_{v,t}:\, K \rightarrow Y$,
$$g_{v,t}(x) = \frac{f(v + tx) - f(v)}{t}\enspace .$$
For all $x,x'\in K$ and $0 < t \leq R^{-1} \epsilon$, we have
\begin{equation}\label{eq7'}
\|g_{v,t} (x) - g_{v,t}(x')\| \leq M \|x - x'\|.
\end{equation}
The family $\{g_{v,t}\}_{0<t \leq R^{-1} \epsilon}$ is an
equicontinuous family of maps converging pointwise to the map
$K\to Y,\; x\mapsto f'_v(x)$, hence for all $x\in K$, the set
$\set{g_{v,t}(x)}{0<t \leq R^{-1} \epsilon}$ is relatively compact.
This implies, by Ascoli's theorem, that $g_{v,t}(x)$ converges 
to $f'_v(x)$ as  $t \rightarrow 0^+$, uniformly in $x\in K$, which 
shows~\eqref{eq4'}. 
Of course, by~\eqref{eq7'}, $x\mapsto f'_v(x)$ is $M$-Lipschitz on $K$, 
and since this holds for all compact subsets $K \subset C$, 
$x\mapsto f'_v(x)$ is $M$-Lipschitz
 on $C$. Finally, to prove \eqref{eq5'}, it is enough to show that
\begin{align*}
\lim_{k \rightarrow \infty} \frac{f(v + t_k x_k) - f(v)}{t_k} = f'_v(x)
\end{align*}
holds for all sequences $\{x_k\}_{k \geq 1}$, $x_k \in C$, $x_k \rightarrow x$, 
and $\{t_k\}_{k \geq 1}$, $t_k > 0$, $t_k \rightarrow 0$. Since $\{x_k \mid k 
\geq 1\} \cup \{x \}$ is compact, it follows from \eqref{eq4'} that
$$\lim_{k \rightarrow \infty} \left(\frac{f(v + t_k x_k) - f(v)}{t_k} - f'_v(x_k)
\right) = 0.$$

Since $x\mapsto f'_v(x)$ is continuous on $C$,
$f'_v(x_k) \rightarrow f'_v(x)$, and we 
get~\eqref{eq5'}.
\end{proof}
\begin{remark}\label{rem2}
When $G= X = C = \R^n$ and $Y=\R$,
Rockafellar and Wets~\cite[Chap. 7, \S D]{rock98}
define semidifferentiable maps $f : \R^n \rightarrow \R$ at $v$ by requiring 
that~\eqref{eq5'} holds, for all $x \in \R^n$. Assume now that $X$ and
$Y$ are arbitrary normed vector spaces,
 that $C \subset X$ is a cone, that $G \subset X$,
 and $f : G \rightarrow X$. Then \eqref{eq5'} holds if, and only if,
$x\mapsto f'_v(x)$ is continuous and satisfies \eqref{eq4'}. 
(Indeed, we showed the ``if'' part in the proof of Lemma \ref{lem2b},
and the ``only if'' part of the result is not difficult.)
This equivalence is a special case of a general property,
saying that a sequence of functions $f_n$ (defined on a metric space $X$) converges
 {\em continuously} to a function $f$, i.e.\ satisfies  
$f_n(x_n) \rightarrow f(x)$, for all convergent sequences 
$x_n \rightarrow x$, if, and only if, $f_n$ converges to $f$ uniformly 
on compact sets and $f$ is continuous, see~\cite[Theorem 7.14]{rock98}. 
(The theorem of \cite{rock98} is stated when $X = \R^n$
and $Y=\R$, but the property holds
 for arbitrary metric spaces $X$ and $Y$.)
When $X= C = \R^n$ (and $Y=\R$), Rockafellar and Wets show 
that \eqref{eq5'} is equivalent to \eqref{eq1'}~\cite[Theorem 7.21]{rock98},
so that our definition of semidifferentiable maps is consistent with the 
one of \cite{rock98}. 
By comparing \eqref{eq3'}, which is equivalent to \eqref{eq1'},
with \eqref{eq4'}  together with the requirement that $x\mapsto f'_v(x)$ is continuous, 
which is equivalent to \eqref{eq5'}, we see that for general normed
vector spaces,
our definition \eqref{eq1'} of semidifferentiable maps becomes stronger
than the one we would obtain by taking 
the definition \eqref{eq5'} of \cite{rock98}.
\end{remark}
We shall need the notion
of \NEW{norm} of a continuous homogeneous map
$h$ from a cone $C$ of a normed vector space $(X,\|\cdot\|)$ to 
a normed vector space $(Y,\|\cdot\|)$:
\begin{equation}\label{definormh}
\|h\|_C:=\sup_{x\in C\setminus\{0\}} \frac{\|h(x)\|}{\|x\|} 
<+\infty  \enspace .
\end{equation}
When $C$ is obvious, and in particular when $C=X$,
we will simply write $\|h\|$ instead of $\|h\|_C$.
Since $h(0)=0$
and $h$ is continuous,
there exists $\delta>0$ such that $\|h(x)\|\leq 1$ for all 
$x\in C$ such that $\|x\|\leq \delta$. Hence, by homogeneity of $h$,
$\|h\|_C\leq 1/\delta<+\infty $, as claimed in~\eqref{definormh}.
In particular, when $f$ is  semidifferentiable at $v$ with respect to $C$,
there exists $\gamma\geq 0$ such that
\begin{equation}\label{fpvbound}
\|f'_v(x)\|\leq \gamma \|x\|\quad \forall x\in C\enspace,
\end{equation}
since by definition, $f'_v$ is homogeneous and continuous.
When $f'_v$ is a linear map on $X$,
that is, when $f$ is differentiable,
\eqref{fpvbound} implies that $f'_v$ is  Lipschitz continuous.
Another condition which guarantees
the Lipschitz continuity of $f'_v$ was given in Lemma~\ref{lem2b}.

We next prove a chain rule for semidifferentiable maps:
\begin{lemma}[Chain rule]\label{lemma-chain}
Let $(X,\|\cdot\|)$, $(Y,\|\cdot\|)$ and $(Z,\|\cdot\|)$
be normed vector spaces, 
let $G_1$ and $G_2$ be subsets of $X$ and $Y$ respectively, 
let $f:G_1\to Y$, $g:G_2\to Z$ be two maps such that
$f(G_1)\subset G_2$, and let $v\in G_1$.
Assume that
\begin{enumerate}
\renewcommand{\theenumi}{\rm (A\arabic{enumi})}
\renewcommand{\labelenumi}{\theenumi}
\item\label{as1} $f$  is semidifferentiable at $v$ 
with respect to a cone $C_1$;
\item\label{as2} $g$  is semidifferentiable at $f(v)$ 
with respect to a cone $C_2$;
\item\label{as3} $f'_v(C_1)\subset C_2$ and $f(G_1\cap (v+C_1))\subset f(v)+C_2$;
\item\label{as4} $g'_{f(v)}$ is uniformly continuous
on bounded sets of $C_2$.
\end{enumerate}
Then, $g\comp f:G_1\to Z$ is semidifferentiable at $v$ 
with respect to $C_1$, and
\begin{align*}
(g\comp f)'_v=g'_{f(v)}\comp f'_v \enspace .
\end{align*}
\end{lemma}
\begin{proof}
Using Assumptions~\ref{as1} and~\ref{as2}, we can write:
\begin{align}
\label{sdif1} f(v+x)    &= f(v) + f'_v(x)+ \|x\|\epsilon_1(x)\\
\label{sdif2} g(f(v)+y) &= g(f(v))+ g'_{f(v)}(y) + \|y\|\epsilon_2(y)
\end{align}
where $\epsilon_1$ (resp.\ $\epsilon_2$) is a map defined
on a neighborhood of $0$ in $C_1$ (resp.\ $C_2$),
with $\epsilon_1(x)\to 0$ when $\|x\|\to 0$
(resp.\ $\epsilon_2(y)\to 0$ when $\|y\|\to 0$).
Using~\eqref{sdif1},~\eqref{sdif2} and Assumption~\ref{as3},
we get, 
\begin{align}
\label{sdif3}
g\comp f(v+x) = g\comp f(v)
+g'_{f(v)}\big[f'_v(x)+\|x\|\epsilon_1(x)\big]+ \eta(x)\enspace,
\end{align}
where 
$\eta(x)=\|f'_v(x)+\|x\|\epsilon_1(x)\|\epsilon_2
\big[f'_v(x)+\|x\|\epsilon_1(x)\big]$.
Using \eqref{fpvbound}, we can write 
\begin{align}
\eta(x) = \|x\|\epsilon_3(x) \enspace,
\label{sdif4}
\end{align}
where $\epsilon_3$ is a map defined on a 
neighborhood of $0$ in $C_1$, such that
$\epsilon_3(x)\to 0$ when $\|x\|\to 0$. 
Using the homogeneity of $g'_{f(v)}$ and
$f'_v$, we get:
\[
g'_{f(v)}\big[f'_v(x)+\|x\|\epsilon_1(x)\big]= \|x\|
g'_{f(v)}\big[f'_v(\|x\|^{-1} x)+\epsilon_1(x)\big]
\enspace. 
\]
Hence, the uniform continuity assumption for
$g'_{f(v)}$ (Assumption~\ref{as4}) implies that there is 
a map $\epsilon_4$ defined on a neighborhood
of $0$ in $C_1$ such that
\begin{align}
g'_{f(v)}\big[f'_v(x)+\|x\|\epsilon_1(x)\big]&= \|x\|
\big[g'_{f(v)}\comp f'_v(\|x\|^{-1} x)+\epsilon_4(x)\big]\nonumber\\
&= g'_{f(v)}\comp f'_v(x)+\|x\|\epsilon_4(x) \enspace ,
\label{sdif5}
\end{align}
with $\epsilon_4(x)\to 0$ when $\|x\|\to 0$.
Gathering~\eqref{sdif3},\eqref{sdif4}, and~\eqref{sdif5},
we get
\[
g\comp f(v+x)= g'_{f(v)}\comp f'_v(x) + \|x\|\epsilon_5(x) \enspace,
\]
where $\epsilon_5=\epsilon_3+ \epsilon_4$, which
concludes the proof of the lemma.
\end{proof}
\begin{remark}
When $C_2$ is closed, Assumption~\ref{as3} reduces to the condition
$f(G_1\cap (v+C_1))\subset f(v)+C_2$. 
Also, in this condition, $G_1$ can be replaced
by a neighborhood of $v$ in $G_1$.
\end{remark} 
\begin{remark}
By homogeneity of $g'_{f(v)}$, Assumption~\ref{as4}
of Lemma~\ref{lemma-chain}
is equivalent to the uniform continuity of $g'_{f(v)}$
on the intersection of $C_2$ with the unit ball of $Y$.
\end{remark} 
\begin{remark}
Since a continuous map on a compact subset 
of a metric space is uniformly continuous,
Assumption~\ref{as4} of Lemma~\ref{lemma-chain}
automatically holds when $C_2$
is closed and $Y$ is finite dimensional. Therefore,
Lemma~\ref{lemma-chain} extends the result
stated in~\cite[Exercise 10.27,(b)]{rock98}
in the case of finite dimensional vector spaces. 
\end{remark}

\subsection{Semidifferentiability of sups}
In applications to control and game theory, one
needs to consider maps from $\R^n$ to $\R$ of the form
\begin{align}
f=\sup_{a\in \actionsa} f_a \enspace,
\label{e-d-sup}
\end{align}
where $(f_a)_{a \in\actionsa}$
is a family of maps $\R^n\to \R$, and the supremum is taken for
the pointwise ordering of functions. Dually, $f$ may be defined as the infimum
of a family of maps.

We next give conditions which guarantee that a map of the
form~\eqref{e-d-sup} has directional derivatives
or that it is semidifferentiable.
Recall that a real valued map $g:\actionsa\to\R$ is \NEW{upper-semicontinuous}
(resp.\ \NEW{sup-compact}) if 
for all $\lambda\in\R$, the \NEW{upper level set}
$S_\lambda(g)=\set{a\in\actionsa}{g(a)\geq \lambda}$
is closed (resp.\ compact). The following theorem gives a general version,
for semidifferentiable maps, of the rule of ``differentiation'' of a supremum,
which has appeared in the literature in various guises, see Remarks~\ref{rk-convex},~\ref{rk-finite} below.
\begin{theorem}\label{th-semi}
Let $X$ denote a normed vector space, let $v,x\in X$, let $V$
denote a neighborhood of $v$, and let $f:V\to \R$
be given by~\eqref{e-d-sup},
where $(f_a)_{a\in\actionsa}$ is a family
of maps $V\to\R$ and $\actionsa$ is a Hausdorff topological space.
Make the following assumptions:
\begin{enumerate}
\renewcommand{\theenumi}{\rm (A\arabic{enumi})}
\renewcommand{\labelenumi}{\theenumi}
\item\label{i-attained} There exists $b\in\actionsa$ such that
$f(v)=f_b(v)$;
\item\label{i-exists} For all $a\in\actionsa$, the directional derivative
of $f$ at $v$ in the direction $x$, $\der{f_a}vx$, exists;
\item\label{i-usc} The maps $a\mapsto f_a(v)$
and $a \mapsto (f_a)'_v(x)$ are upper-semicontinuous;
\item\label{i-newcompact} There exists a real number $t_0>0$ such that the map
\[
a \mapsto f_a(v)+t_0\der{f_a}vx 
\]
is sup-compact.
\item\label{i-above} We have
\begin{align*}
f_a(v+tx)\leq f_a(v)+t\der{f_a}vx + t\epsilon_x(t)
\quad\forall a\in\actionsa\;\forall t\in [0,t_1],
\end{align*}
for some $t_1>0$ and some
function $\epsilon_x:[0,t_1]\to \R$ independent of $a$,
such that $\epsilon_x(t)\to 0$ when $t\to 0^+$;
\end{enumerate}
Then, the directional derivative of $f$ at $v$ in the direction $x$
exists, and:
\begin{align}
f'_v(x)= \max_{a\in \actionsa,\; f_a(v)=f(v)} (f_a)'_v(x) \enspace .
\label{e-derivativeofsup}
\end{align}
Moreover, if the previous assumptions are satisfied for all $x\in X$,
if $X$ is finite dimensional and if $f$ is Lipschitz continuous 
in a neighborhood of $v$, then, $f$ is semidifferentiable at point $v$.
\end{theorem}
We shall see in Remarks~\ref{rk-i-exists}--\ref{rk-finite}
below that Assumptions~\ref{i-exists},\ref{i-above}, or~\ref{i-newcompact}
are implied by several standard assumptions, and that 
Theorem~\ref{th-semi} extends several known results.

Before proving Theorem~\ref{th-semi}, we make the following observation.
\begin{lemma}\label{lem-t0-t1}
Let~\ref{i-newcompact} and~\ref{i-usc} be as in Theorem~\ref{th-semi}.
If Assumption~\ref{i-newcompact} is satisfied for some $t_0>0$,
and if Assumption~\ref{i-usc} is satisfied, then
Assumption~\ref{i-newcompact} is still satisfied if we replace $t_0$ by any $t\in (0,t_0)$.
\end{lemma}
\begin{proof}
We first remark that if $g,h$ are two
maps $\actionsa\to \R$, such that $g$ is sup-compact,
$h$ is upper-semicontinuous, and $h\leq g$, then
$h$ is also sup-compact. Indeed, for all $\lambda\in\R$,
$S_\lambda(h)$ is included in $S_\lambda(g)$,
$S_\lambda(h)$ is closed since $h$ is upper-semicontinuous,
and $S_\lambda(g)$  is compact because $g$
is sup-compact, and thus, $S_\lambda(h)$ is compact.

Consider now $h_0:a\mapsto f_a(v)+t_0\derz{(f_a)}$.
By Assumption~\ref{i-newcompact}, $h_0$ is sup-compact.
Take $t\in (0,t_0)$,
and define $h_1: a\mapsto f_a(v)+t\derz{(f_a)}$,
which is upper-semicontinuous thanks to Assumption~\ref{i-usc}.
We have $h_1(a)=
\frac{t}{t_0}h_0(a)+(1-\frac{t}{t_0})f_a(v)
\leq \frac{t}{t_0}h_0(a)+(1-\frac{t}{t_0})f(v)$
which shows that $h_1$ is bounded from above by 
a sup-compact map. Thus, $h_1$ is sup-compact.
\end{proof}
\begin{proof}[Proof of Theorem~\ref{th-semi}]
The first part of the theorem is a property of the
one real variable functions $g(t)=f(v+tx)$
and $g_a(t)=f_a(v+t x)$, that
is one can assume without loss of generality that
$X=\R$, $v=0$ and $x=1$, $g=f$ and $g_a=f_a$, and omit $x$
when possible.

By Assumption~\ref{i-above}, we have for all $t\in (0,t_1]$
and for all $a\in\actionsa$,
\begin{align}
\frac{g_a(t)-g(0)}t
&=\frac{g_a(t)-g_a(0)+g_a(0)-g(0)}{t}
\leq h_a(t)+ \epsilon(t)
\label{ineq-falpha}
\end{align}
where
\[
h_a(t):=\frac{t\der{g_a}01+
g_a(0)-g(0)}t 
\enspace .
\]
Taking the sup of the inequalities~\eqref{ineq-falpha},
we get
\begin{align*}
\frac{g(t)-g(0)}t
&\leq 
\sup_{a\in\actionsa} 
h_a(t) +\epsilon(t)\enspace .
\end{align*}
Since $g_a(0)\leq g(0)$,
$h_a(t)$ is a nondecreasing function of $t$, 
so that 
\begin{align}
\limsup_{t\to 0^+}\frac{g(t)-g(0)}t 
&\leq 
\lim_{t\to 0^+}
\sup_{a\in\actionsa}
h_a(t) +\epsilon(t)
=\inf_{t>0} \sup_{a\in\actionsa}
h_a(t)\enspace .
\label{ineq-falpha3}
\end{align}
For all $t\in (0,t_0)$,
the map $a\mapsto h_a(t)$, 
which is sup-compact by Lemma~\ref{lem-t0-t1},
attains its sup at some point $a_t\in\actionsa$.
Let $\lambda$ denote the value of the right hand side of~\eqref{ineq-falpha3}.
We have, $h_{a_t}(t)\geq \lambda$, so that the upper level
set $S_\lambda(a \mapsto h_a(t))$ is nonempty. Since a nonincreasing
intersection of nonempty compact sets is nonempty,
we can find $c\in \cap_{t\in (0,t_0)}S_\lambda(a \mapsto h_a(t))$.
Then, for all $t\in(0,t_0)$, 
\begin{align}
\frac{t\der{g_c}01 + g_c(0)-g(0)}t&=
\der{g_c}01 +\frac{g_c(0)-g(0)}t
\geq \lambda \enspace.
\label{e-conclusion}
\end{align}
Observe that $\lambda$ is finite, because $\lambda\geq \der{g_b}01$
with $g_b(0)=g(0)$ as in Assumption~\ref{i-attained}.
Thus, multiplying~\eqref{e-conclusion}
by $t$ and letting $t\to 0^+$, we get 
$g_c(0)-g(0)\geq 0$, and since
the other inequality is obvious, $g_c(0)=g(0)$,
which shows that $c\in  \set{a\in A}{g_a(0)=g(0)}$.
Combining~\eqref{e-conclusion} and~\eqref{ineq-falpha3}, 
we get
\begin{align}
\limsup_{t\to 0^+}\frac{g(t)-g(0)}t 
\leq 
\inf_{t>0} \sup_{a\in\actionsa}
h_a(t)
=\lambda\leq \der{g_c}01 \leq \sup_{a:\;g_a(0)=g(0)}\der{g_a}01 \enspace .
\label{ineq-falpha4}
\end{align}

Conversely, for all $a$ such that $g_a(0)=g(0)$, we have
$g(t)-g(0)=g(t)-g_a(0)\geq g_a(t)-g_a(0)$,
since $g\geq g_a$, so that
\begin{align*}
\liminf_{t\to 0^+} \frac{g(t)-g(0)}t
\geq \liminf_{t\to 0^+} \frac{g_a(t)-g_a(0)}t
=\der{g_a}01 \enspace.
\end{align*}
Thus,
\begin{align}\label{e-theboundfff}
\liminf_{t\to 0^+} \frac{g(t)-g(0)}t\geq
\sup_{a\in\actionsa,\;g_a(0)=g(0)}\der{g_a}01 \enspace.
\end{align}
Gathering~\eqref{ineq-falpha4} and~\eqref{e-theboundfff},
we get that the directional derivative of $f$ at $v$ in the direction
$x$, $f'_v(x)$, exists, is finite, and is given by~\eqref{e-derivativeofsup}.
Moreover, the sup in~\eqref{e-derivativeofsup} is a max since
it is attained by taking $a=b$.

Now the last assertion of Theorem~\ref{th-semi} follows from Lemma~\ref{lem2b}.
\end{proof}

\begin{remark}\label{rk-i-newcompact}
Assumption~\ref{i-newcompact} is implied by
standard assumptions. 

First, if $a \mapsto \der{f_a}vx$ is sup-compact,
then Assumption~\ref{i-newcompact} follows
from Assumption~\ref{i-usc}.
Indeed, for all $t_0>0$, $f_a(v) +t_0 \der{f_a}vx
\leq f(v)+t_0\der{f_a}vx$. Hence,
$a\mapsto f_a(v) +t_0 \der{f_a}vx$
is sup-compact since it is upper-semicontinuous
and bounded from above by a sup-compact map.

Symmetrically, if  $a \mapsto f_a(v)$ is sup-compact and
$a \mapsto \der{f_a}vx$ is upper-semi\-con\-ti\-nuous and 
bounded from above, then again Assumption~\ref{i-newcompact} is satisfied
for all $t_0>0$, since $f_a(v) +t_0 \der{f_a}vx
\leq f_a(v) +t_0\lambda $ for some $\lambda\in \R$.

Other assumptions which imply Assumption~\ref{i-newcompact}
are the following. Assume that $a \mapsto f_a(v)$
is sup-compact, that Assumption~\ref{i-usc} is satisfied,
and that $a \mapsto f_a(v)+t_1\der{f_a}vx$
is bounded from above, for some $t_1>0$. 
Then, we claim that 
$a \mapsto f_a(v)+t_0\der{f_a}vx$ is sup-compact,
for all $t_0\in (0,t_1)$. Indeed, let 
$\lambda=\sup_{a\in \actionsa}f_a(v)+t_1\der{f_a}vx$.
Then, $f_a(v)+t_0\der{f_a}vx 
= (1-\frac{t_0}{t_1}) f_a(v)+ \frac{t_0}{t_1} (f_a(v)+
t_1\der{f_a}vx)
\leq (1-\frac{t_0}{t_1})f_a(v)+\frac{t_0}{t_1}\lambda$, and since
an upper-semicontinuous
map bounded from above by a sup-compact map is sup-compact,
the claim is proved.
\end{remark}

\begin{remark}\label{rk-i-affine}
If $X=\R^n$, and if for all $a\in\actionsa$,
$f_a: x\mapsto p_a \cdot x +r_a$ 
for some $p_a\in\R^n$ and $r_a\in \R$, then $f_a$
is affine, thus differentiable at any $v\in\R^n$
with $(f_a)'_v(x)=p_a \cdot x$ for all $x\in\R^n$. 
Hence, $f$ is convex, Assumption~\ref{i-exists} is satisfied, 
Assumption~\ref{i-above} is satisfied with $\epsilon\equiv 0$
and Formula~\eqref{e-derivativeofsup} becomes
\[
f'_v(x)=\max_{a\in \actionsa,\; f_a(v)=f(v)} p_a \cdot x \enspace.
\]
Assumption~\ref{i-usc} is satisfied when $a\mapsto p_a$ is 
continuous and $a\mapsto r_a$ is upper-semicontinuous.
Finally, Assumption~\ref{i-newcompact} is satisfied if, and only if, 
$a \mapsto f_a(v+t_0 x)$ is sup-compact. These assumptions
are thus satisfied for instance when $a\mapsto r_a$ is 
sup-compact and $a\mapsto p_a$ is bounded and continuous. 
\end{remark}

\begin{remark}\label{rk-convex}
{From} Remark~\ref{rk-i-affine}, one can see that
Theorem~\ref{th-semi} extends the classical rule
for the directional derivative of a convex function.
When $f:\R^n\to\R\cup\{+\infty\}$ is convex,
we can write, by Legendre-Fenchel
duality~\cite[Th.~12.2]{rock}:
\[
f(x)=\sup_{p\in \dom f^*} p\cdot x-f^*(p) \enspace ,
\]
where $f^*(p)=\sup_{x\in\R^n}p\cdot x-f(x)$ denotes the Legendre-Fenchel
transform of $f$, evaluated at $p\in\R^n$, and 
for any map $g:\R^n\to\R\cup\{+\infty\}$,
$\dom g=\set{x\in\R^n}{g(x)<\infty}$ denotes the effective domain of $g$.
So, the restriction of $f$ to $\dom f$
can be written as~\eqref{e-d-sup} with
$\actionsa=\dom f^*$, $a=p$, and $f_a(x)=f_p(x)=p\cdot x-f^*(p)$.
Let us denote by $V$ the interior of $\dom f$, assume that $v\in V$.
Assumption~\ref{i-attained} is satisfied, because
$f(v)=p\cdot v-f^*(p)=f_p(v)$ holds for all
$p$ in the subdifferential $\partial f(v)$ of $f$ at $v$,
and $\partial f(v)$ is nonempty because $f$ is convex and 
$v$ is in $V$ (see again~\cite[Th.~23.4]{rock}).
We are in the conditions of Remark~\ref{rk-i-affine},
with $p_a=a$ and $r_a=-f^*(a)$. 
Hence, Assumptions~\ref{i-exists} and~\ref{i-above} are satisfied.
Since $a\mapsto p_a$ is continuous and $a\mapsto
r_a$ is upper-semicontinuous, Assumption~\ref{i-usc} is satisfied.
Moreover, Assumption~\ref{i-newcompact} is satisfied since for any 
$w\in V$, $a \mapsto f_a(w)$ is sup-compact and
since $v+t_0 x\in V$ when $v\in V$, $x\in \R^n$ and $t_0>0$ is small enough. 
Indeed, $f_p(w)=p\cdot w-f^*(p)=p\cdot w-\sup_{x\in\R^n}(p\cdot x-f(x))
\leq p\cdot (w-x)+f(x)$, for all $x\in \R^n$.
Taking $\epsilon>0$ and $x=w+\epsilon \frac{p}{\|p\|}$, we get $f_p(w)
\leq  - \epsilon p\cdot \frac{p}{\|p\|}+f(w+\epsilon\frac{p}{\|p\|})=
-\epsilon\|p\|+f(w+\epsilon\frac{p}{\|p\|})$. Since 
$f$ is convex and $w\in V$, $f$ is continuous on a neighborhood of 
$w$~\cite[Th.\ 10.1]{rock} and thus for $\epsilon>0$ small enough
$f(w+\epsilon\frac{p}{\|p\|})$
can be bounded independently of $p$, which implies that
the upper level sets of $p\mapsto f_p(w)$
are bounded. Since these level sets are closed subsets
of $\R^n$, they are compact, so $p\mapsto  f_p(w)$
is sup-compact. Hence all the assumptions of Theorem~\ref{th-semi} are
satisfied. Finally, Formula~\eqref{e-derivativeofsup} becomes
\[
f'_v(x)=\max_{p\in\partial f(v)} p\cdot x \enspace,
\]
which coincides with the classical formula
of Theorem~23.4 of~\cite{rock}. 
\end{remark}
\begin{remark}\label{rk-i-exists}
When, in Theorem~\ref{th-semi}, $X=\R^n$, and $f_a:V\to \R$ is concave,
the directional derivative
\begin{align*}
\der{f_a}vx =\lim_{t\to 0^+}
\frac{f_a(v+tx)-f_a(v)}t =\sup_{t>0} 
\frac{f_a(v+tx)-f_a(v)}t
\end{align*}
exists and is finite~\cite[Th~23.1 and~23.4]{rock}
(see also Remark~\ref{rk-convex}),
so that Assumption~\ref{i-exists} is satisfied.
Then, Assumption~\ref{i-above} is satisfied with $\epsilon\equiv 0$.
\end{remark}
\begin{remark}\label{rk-finite}
When $\actionsa$ is finite, only Assumption~\ref{i-exists} has to be checked
in the first part of Theorem~\ref{th-semi} and the second part of 
Theorem~\ref{th-semi}
shows that semidifferentiable maps $\R^n\to \R$ are
stable by max, a well known fact~\cite[Exercise~10.27]{rock98}.
More generally, the part of~\cite[Theorem~10.31]{rock98}
which asserts that lower-$\continuous_1$ maps are semidifferentiable
may be recovered from Theorem~\ref{th-semi} 
(in~\cite{rock98},
a map $f:\R^n\to\R$ is said to be \NEW{lower-$\continuous_1$} if it can
be written as~\eqref{e-d-sup} with $\actionsa$ compact, $x\mapsto f_a(x)$
of class $\continuous_1$, and $(a,x)\mapsto (f_a(x),(f_a)'_x)$
continuous.)
\end{remark}
\begin{remark}
The following counter example shows that
the upper-semicontinuity assumptions are useful
in Theorem~\ref{th-semi}.
Take any function $f:\R\to\R$ of Lipschitz constant $1$,
which is not semidifferentiable (for instance, 
$f(x)=\frac x2\sin (\log|x|)$, which is not semidifferentiable
at $0$). Then,
$f:]-1,1[\to \R$ can be written as~\eqref{e-d-sup} with $\actionsa=[-1,1]$ and
\[
f_a(x)=f(a)-|x-a| \enspace. 
\]
All the assumptions of Theorem~\ref{th-semi} are satisfied,
except the requirement in~\ref{i-usc} and \ref{i-newcompact}
that $a\mapsto \der{f_a}vx$ and $a \mapsto f_a(v)+t_0\der{f_a}vx$
be upper semicontinuous.
Indeed, 
\begin{align*}
\der{f_a}vx=\begin{cases} -x & \mrm{if $v>a$}\\
x&\mrm{if $v<a$}\\
-|x| &\mrm{if $v=a$,}
\end{cases}\end{align*}
so that $\der{f_v}vx=-|x|<\limsup_{a\to x} \der{f_a}vx=+|x|$,
unless $x=0$.
\end{remark}

\section{Semidifferentials of order preserving and nonexpansive maps}
\label{sect-semider}
In this section, we consider various classes of
maps on cones, and establish auxiliary results concerning
their semidifferentials.

Let $(X,\|\cdot\|)$ be a Banach space endowed with a partial 
ordering $\leq$, and $f$ be a map 
from a subset $D\subset X$, to $X$.
Recall that $f$ is \NEW{order-preserving} if for all $x,y\in D$, 
$x\leq y\implies f(x)\leq f(y)$.
We shall say that $f$ is \NEW{convex} 
if $f((1-t)x +t y)\leq (1-t) f(x)+t f(y)$ for all $0\leq t\leq 1$ and
 $x,y\in D$, such that $(1-t) x+ty\in D$.
We shall say that $f$ is \NEW{subhomogeneous} 
if $t f(y)\leq f(t y)$ for all $0\leq t\leq 1$ and $y\in D$,
such that $ty\in D$.
Also, we shall say that $f$ satisfies a property (for instance is 
order preserving,
or homogeneous,\ldots) in a  neighborhood of a point $v$ of $D$,
if there exists a neighborhood $V$ of $v$ in $D$ such that 
$f|_V$ satisfies this property. We shall use 
systematically the following well known elementary properties
(see for instance~\cite{nussbaum88,thompson,bushell73,potter}):

\begin{lemma}\label{lemma-nonexpan}
Let $C$ be a proper cone, $u\in C\setminus\{0\}$,
$\psi\in C^*\setminus\{0\}$ such that $\psi(u)>0$ and let
$\Sigma_u=\set{x\in C_u}{\psi(x)=\psi(u)}$.
\begin{enumerate}
\renewcommand{\theenumi}{\rm (\roman{enumi})}
\renewcommand{\labelenumi}{\theenumi}
\item \label{lemma-nonexpan-1}
If $f:C_u\to C$ is order-preserving and homogeneous,
then $f(C_u)\subset C_{f(u)}$ and 
$f$ is nonexpansive with respect to $d$ and $\tho$.
\item 
If $f:C_u\to C$ is order preserving and subhomogeneous, then
$f(C_u)\subset C_{f(u)}$ and 
$f$ is nonexpansive with respect to $\tho$, and the restriction
$f|_{\Sigma_u}$ of $f$ to $\Sigma_u$ is nonexpansive with respect 
to $d$.
\end{enumerate}
\end{lemma}

We now introduce additive analogues of (sub-) homogeneity:
we shall say that a self map $h$ of a Banach space $(X,\|\cdot\|)$
is \NEW{additively homogeneous} with respect to 
$v\in X$ if,  for all $x\in X$ and $t\in\R$, $h(x+tv)=h(x)+t v$;
similarly, we shall say that $h$ is \NEW{additively subhomogeneous} with 
respect to $v$ if,  for all $x\in X$ and $t\geq 0$, $h(x+tv)\leq h(x)+t v$.
The following lemma relates the properties of $f$ with
those of $f'$. 
\begin{lemma}\label{preli-fpv}
Let $C$ be a proper cone with nonempty interior  in 
a Banach space $(X,\|\cdot\|)$.
Let $G$ be an open subset of $(X,\|\cdot\|)$ included in $C$ and $f: G\rightarrow \Cint$.
Let $v\in  G$ be a fixed point of $f$: $f(v)= v$. 
Let $\psi\in C^*\setminus\{0\}$ be such that $\psi(v)=1$ and denote
$\Sigma=\set{x\in \Cint}{\psi(x)=1}$.
Assume that $f$ is semidifferentiable at $v$.
The following implications hold:
\begin{enumerate}
\renewcommand{\theenumi}{\rm (\roman{enumi})}
\renewcommand{\labelenumi}{\theenumi}
\item \label{preli-fpv1} If $f$ is order preserving in a neighborhood of $v$,
then $f'_v:X\to X$ is order preserving.
\item \label{preli-fpv12} If $f$ is convex in a neighborhood of $v$,
then $f'_v:X\to X$ is convex.
\item  \label{preli-fpv2} If $f$ is homogeneous in a neighborhood of $v$,
 then $f'_v$ is additively homogeneous with respect to  $v$.
\item  \label{preli-fpv3} If $f$ is subhomogeneous in a neighborhood of $v$,
 then $f'_v$ is additively subhomogeneous with 
respect to  $v$.
\item  \label{preli-fpv34} Assume that there exists $\delta>0$  such that 
 $f(tv)\leq tf(v)$ for all $1\leq t \leq 1+\delta$.
Then, $f'_v(v)\leq v$.
\item  \label{preli-fpv32} Assume that there exists $\delta>0$  such that 
 $\delta\leq 1$ and  $t f(v)\leq f(tv)$ for all $1-\delta\leq  t\leq 1$.
Then, $f'_v(-v)\geq -v$.
\item \label{preli-fpv4} If $f$ is nonexpansive with respect to $\tho$
in a neighborhood of $v$, then $f'_v$ is nonexpansive with respect to 
$\|\cdot \|_v$.
\item \label{preli-fpv5} If $f|_{G\cap \Sigma}$ is nonexpansive with 
respect to $d$ in a neighborhood of $v$, then $f'_v|_{\psi^{-1}(0)}$  
is nonexpansive with respect to $\omega_v$.
\end{enumerate}
\end{lemma}
\begin{proof} We shall only need the definition~\eqref{eq2'} of $f'_v$.

\ref{preli-fpv1}: If $f$ is order preserving in a neighborhood of $v$,
then for all $x,y\in X$ such that $x\leq y$, we have
$f(v+tx)\leq f(v+ty)$ for all $t>0$ small enough, hence,
from~\eqref{eq2'}, $f'_v(x)\leq f'_v(y)$,
 which shows that $f'_v$ is order preserving.

\ref{preli-fpv12}: If $f$ is convex in a neighborhood of $v$,
then for all $x,y\in X$ and $s\in [0,1]$, we have
$f(v+t((1-s)x+s y))=f((1-s)(v+t x)+s (v+ty))\leq (1-s)f(v+tx)+
s f(v+ty)$ for all $t>0$ small enough, hence,
from~\eqref{eq2'}, $f'_v((1-s)x+s y)\leq (1-s) f'_v(x)+ sf'_v(y)$,
 which shows that $f'_v$ is convex.

\ref{preli-fpv2}: If $f$ is homogeneous in a neighborhood of $v$,
then for all $x\in X$, $s\in\R$, 
\[ f(v+t(x+sv))=f((1+ts) (v+\frac{t}{1+ts} x))=(1+ts) f(v+\frac{t}{1+ts} x)
\] 
for all $t>0$ small enough. Since $f(v)=v$, this leads to
\begin{equation}\label{ine3}
 \frac{f(v+t(x+sv))-f(v)}{t}=\frac{1+ts}{t} \left( f(v+\frac{t}{1+ts} x)
-f(v)\right) +s v\end{equation}
for all $t>0$ small enough. Using~\eqref{eq2'} and \eqref{ine3}, we get
\[ f'_v(x+sv)=f'_v(x)+s v\]
which shows that $f'_v$ is additively homogeneous with respect to  $v$.

\ref{preli-fpv3}: If $f$ is subhomogeneous in a neighborhood of $v$,
then $f(tx)\leq t f(x)$ for all $x\in X$ and $t\geq 1$
such that $x$ and $tx$ are sufficiently close to $v$.
By the same arguments as for \ref{preli-fpv2}, we obtain that 
for all $x\in X$, $s\geq 0$,
\[ f(v+t(x+sv))\leq (1+ts) f(v+\frac{t}{1+ts} x) \] 
and
\begin{equation}\label{ine4}
 \frac{f(v+t(x+sv))-f(v)}{t}\leq \frac{1+ts}{t} \left( f(v+\frac{t}{1+ts} x)
-f(v)\right) +s v\end{equation}
for all $t>0$ small enough. Using~\eqref{eq2'} and \eqref{ine4}, we get
\[ f'_v(x+sv)\leq f'_v(x)+s v\]
for all $s\geq 0$, 
which shows that $f'_v$ is additively subhomogeneous with respect to  $v$.

\ref{preli-fpv34}: Assume that there exists $\delta>0$  such that 
 $f(tv)\leq tf(v)$ for all $1\leq t \leq 1+\delta$.
Then, using $f(v)=v$, we get for all $0< s\leq \delta$
\[ \frac{f(v+sv)-f(v)}{s}\leq f(v)=v\]
and passing to the limit when
$s$ goes to $0$, we obtain $f'_v(v)\leq v$.

\ref{preli-fpv32}: Assume that there exists $\delta>0$  such that 
 $\delta\leq 1$ and  $t f(v)\leq f(tv)$ for all $1-\delta\leq  t\leq 1$.
Then, similarly to case \ref{preli-fpv34}, we get for all $0< s\leq \delta$
\[ \frac{f(v-sv)-f(v)}{s}\geq -f(v)=-v\]
and passing to the limit when
$s$ goes to $0$, we obtain $f'_v(-v)\geq -v$.

\ref{preli-fpv4}: If $f$ is nonexpansive with respect to $\tho$
in a neighborhood of $v$, then for all $x,y\in X$ such that $x\neq y$,
\begin{equation}\label{ine5}
 \frac{\tho(f(v+tx),f(v+ty))}{\tho(v+tx,v+ty)}\leq 1\end{equation}
for all $t>0$ small enough. In particular $f(v+tx)$ and
$f(v+ty)$ tend to $v$ in $(\Cint, \tho)$ when $t\to 0^+$.
Using~\eqref{e-normequivdbar} and \eqref{ine5}, we get
\begin{align}\label{ine1}
\lim_{t\to 0^+}  \frac{\|f(v+tx)-f(v+ty)\|_v}{\|tx-ty\|_v}\leq 1
\enspace .\end{align}
Since for any proper cone $C$ and $v\in \interior C$,
there exists a constant $\kappa$ such that $\|\cdot\|_v\leq \kappa \|\cdot \|$,
the convergence in~\eqref{eq2'} holds
not only for the topology of the norm $\|\cdot \|$
but also for the topology of $\|\cdot\|_v$.
Gathering~\eqref{ine1} with~\eqref{eq2'} for 
the topology of $\|\cdot\|_v$, we deduce
\[ \frac{\|f'_v(x)-f'_v(y)\|_v}{\|x-y\|_v}\leq 1\]
which shows that $f'_v$ is nonexpansive with respect to 
$\|\cdot \|_v$.

\ref{preli-fpv5}: If $f|_{G\cap \Sigma}$ is nonexpansive with 
respect to $d$ in a neighborhood of $v$, then,
for all $x,y\in \psi^{-1}(0)$ such that $x\neq y$,
\[ \frac{d(f(v+tx),f(v+ty))}{d(v+tx,v+ty)}\leq 1\]
for all $t>0$ small enough (since $v+tx$ and $v+ty\in G\cap \Sigma$
for all $t>0$ small enough). Hence,  using~\eqref{e-normequivosc}, we get
\begin{align}\label{ine2}
\lim_{t\to 0^+}  \frac{\omega_v(f(v+tx)-f(v+ty))}{\omega_v(tx-ty)}\leq 1
\enspace .\end{align}
As said before, the convergence in~\eqref{eq2'} holds
for the topology of $\|\cdot\|_v$,
hence by~\eqref{equiv-omega0}, we get that
\[ \lim_{t \rightarrow 0^+} \omega_v( \frac{f(v + tz) - f(v)}{t} -f'_v(z) 
)=0 \enspace, \]
for all $z\in \psi^{-1}(0)$.
Using this convergence together with~\eqref{ine2},
we deduce
\[ \frac{\omega_v(f'_v(x)-f'_v(y))}{\omega_v(x-y)}\leq 1\]
which shows that $f'_v|_{\psi^{-1}(0)}$  
is nonexpansive with respect to $\omega_v$.
\end{proof}
We shall also need:
\begin{lemma}\label{preli-fpv33}
Let $h$ denote a self-map of a Banach space $(X,\|\cdot\|)$
endowed with a partial ordering $\leq$, and $v\in X$.
If $h$ is convex, homogeneous, and if $h(v)\leq v$,  then
$h$ is additively subhomogeneous with respect to  $v$.
\end{lemma}
\begin{proof}
If $h$ is convex and $h(v)\leq v$,  then
using the homogeneity of $h$, we get
for all $t\geq 0$ and $x\in X$,
\[h(x+tv)\leq \frac{1}{2}( h(2x) +h(2t v))=h(x)+t h(v)
\leq h(x)+ t v\enspace ,\]
which shows that $h$ is additively subhomogeneous with respect to  $v$.
\end{proof}
Symmetrically, if $h:X\to X$ is
concave, homogeneous and such that $h(-v)\geq -v$,
then $h$ is additively subhomogeneous with respect to  $v$.
Lemma~\ref{preli-fpv33} covers only a special case:
a homogeneous and additively subhomogeneous map
need not be convex or concave
(consider for instance $h:\R^2\to \R^2, h(x_1,x_2)=
((x_1\vee x_2)\wedge x_1/2, x_2)$, which is homogeneous
and additively subhomogeneous with respect to $(1,1)$).

If $C,G,f$ are as in Lemma~\ref{preli-fpv},
to study the eigenvectors of $f$, 
we shall pick a linear form $\psi\in C^*\setminus\{0\}$,
define the set $\Sigma=\set{x\in \Cint}{\psi(x)=1}$,
and consider the map 
\begin{equation}
\label{defg}
\tilde{f} :G \rightarrow \Sigma,  \quad \tilde{f}(x)=\frac{f(x)}{\psi(f(x))}
\enspace .
\end{equation}
The following lemma states some basic properties of $\tilde{f}$.
\begin{lemma}\label{ftog}
Let $C$ be a proper cone with nonempty interior in 
a Banach space $(X,\|\cdot\|)$.
Let $G$ be an open subset of $(X,\|\cdot\|)$ included in $C$ and  $f: G\rightarrow \Cint$.
Let $\psi\in C^*\setminus\{0\}$, 
$\Sigma=\set{x\in \Cint}{\psi(x)=1}$, let $\tilde{f}$ be defined 
by~\eqref{defg} and $g=\tilde{f}|_{G\cap \Sigma}$.
If $f|_{G\cap \Sigma}$ is nonexpansive with 
respect to $d$, so is $g$.
Let $v\in  G\cap \Sigma$ be a fixed point of $f$:
$f(v)= v$. Assume that $f$ is semidifferentiable at $v$.
Then, $\tilde f$ is semidifferentiable at $v$, and
\begin{equation}\label{gprimv}
\tilde{f}'_v(x)= f'_v(x)-\psi(f'_v(x)) v\quad \forall x\in X\enspace .
\end{equation}
Moreover, if  $f|_{G\cap \Sigma}$ is nonexpansive with 
respect to $d$ in a neighborhood of $v$, then 
$g'_v=\tilde{f}'_v|_{\psi^{-1}(0)}$  is nonexpansive
 with respect to $\omega_v$.
\end{lemma}
\begin{proof}
By definition of $d$, 
we have $d(\lambda x,\mu y)=d(x,y)$
for all $\lambda,\mu>0$ and $x,y\in \Cint$. Hence, if $f|_{G\cap \Sigma}$
is nonexpansive with respect to $d$, then for all $x,y\in G\cap \Sigma$,
\[ d(g(x),g(y))=d\Big(\frac{f(x)}{\psi(f(x))},\frac{f(y)}{\psi (f(y))}
\Big) =d(f(x),f(y))\leq d(x,y)\enspace ,\]
which shows that $g$ is nonexpansive with respect to $d$.

Consider the map $R:\Cint \to \Sigma,\,
y\mapsto \frac{y}{\psi(y)}$. We have $\tilde{f}=R\circ f$.
Since $\psi$ is linear, thus differentiable at any point
$w\in X$, and $\psi(y)>0$ for all $y\in \Cint$, it follows that
$R$ is differentiable at any $w\in \Cint$, with
\begin{align}\label{diff-R}
 R'_w(y)=\frac{y}{\psi(w)}-\frac{\psi(y) w}{\psi(w)^2}\enspace .
\end{align}
In particular $R$ is differentiable at $v\in\Cint$, thus $R$
is semidifferentiable at $f(v)=v$ and 
$R'_{f(v)}$ is uniformly continuous on all bounded sets.
Moreover $f$ is semidifferentiable at $v$, hence 
applying Lemma~\ref{lemma-chain}, we get 
that $\tilde{f}$ is semidifferentiable  at $v$ and that
$\tilde{f}'_v=R'_{v}\circ f'_v$. From~\eqref{diff-R} and $\psi(v)=1$,
we get~\eqref{gprimv}.

If $f|_{G\cap \Sigma}$ is nonexpansive with respect to $d$
in a neighborhood of $v$, 
then by the same arguments as above, $g=\tilde{f}|_{G\cap \Sigma}$ is
nonexpansive with respect to $d$ 
in a neighborhood of $v$, and by
Lemma~\ref{preli-fpv},~\ref{preli-fpv5}
applied to $\tilde{f}$, $g'_v=\tilde{f}'_v|_{\psi^{-1}(0)}$ is
nonexpansive with respect to $\omega_v$.
\end{proof}
The following result, which 
controls $x-h(x)$ in terms of the fixed point set of an order
preserving additively subhomogeneous map $h$,
will play a key role in the proof of the uniqueness theorem
for eigenvectors (Theorem~\ref{theo01} below),
\begin{lemma}\label{lemma-ess}
Let $C$ be a proper cone with nonempty interior  in 
a Banach space $(X,\|\cdot\|)$ and let $v\in \Cint$.
Let $h:X\to X$ be order preserving and additively subhomogeneous 
with respect to $v$. Assume that the set $S=\set{y\in X}{h(y)=y}$
of fixed points of $h$ is nonempty. We have, for all $x\in X$
\begin{align}
\label{cw1}
\sbf{v}{x-h(x)} &\leq \inf_{y\in S} \sbf{v}{x-y}\vee 0\\
\stf{v}{x-h(x)} &\geq \sup_{y\in S} \stf{v}{x-y}\wedge 0 \enspace .
\label{cw2}
\end{align}
\end{lemma}
\begin{proof}
We first show that if $h(0)=0$, then
\begin{equation}\label{ess1}
\stf{v}{x-h(x)}\geq \stf{v}{x}\wedge 0 \quad\forall x\in X\enspace .
\end{equation}
Let $x\in X$ and denote $\beta=\stf{v}{x}$ and $z=x-h(x)$.
Then, $x\leq \beta v\leq (\beta\vee 0) v$. Since $h(0)=0$ and
$h$ is order preserving and  additively subhomogeneous 
with respect to $v$, we get
$h(x)\leq h((\beta\vee 0) v)\leq h(0)+ (\beta\vee 0) v=(\beta\vee 0) v$.
Hence,
\begin{equation}\label{ess2}
x=z+h(x)\leq \stf{v}{z} v +(\beta\vee 0) v= (\stf{v}{z} +\beta\vee 0) v
\enspace .
\end{equation}
Since $\beta=\inf\set{a\in \R}{ x\leq a v}$, we deduce from~\eqref{ess2}
that $\beta \leq \stf{v}{z} +\beta\vee 0$, thus $\beta\wedge 0=\beta-
\beta\vee 0\leq \stf{v}{z} $, which shows~\eqref{ess1}.

Applying~\eqref{ess1} to $-x$ and replacing $h$ by $\ant{h}(x)=-h(-x)$,
which is order preserving,  additively subhomogeneous 
with respect to $v$ and satisfies $\ant{h}(0)=0$, we get
$\stf{v}{-x+h(x)}\geq \stf{v}{-x}\wedge 0$. Since $\sbf{v}{x}=-\stf{v}{-x}$,
this shows that 
\begin{equation}\label{ess3}
\sbf{v}{x-h(x)}\leq \sbf{v}{x}\vee 0 \quad\forall x\in X\enspace .
\end{equation}

Let $y\in S$ and consider $h_y:X\to X,\; x\mapsto h_y(x)=h(y+x)-y$.
Then,  $h_y(0)=0$ and $h_y$ is order preserving and additively subhomogeneous 
with respect to $v$: $h_y(x+tv)=h(y+x+tv)-y\leq h(y+x)+tv-y=h_y(x)+tv$
for all $t\geq 0$.
Replacing $h$ by $h_y$ in~\eqref{ess1}, we get
\[ \stf{v}{x+y-h(x+y)}\geq \stf{v}{x}\wedge 0 \quad\forall x\in X\enspace ,\]
then, replacing $x$ by $x-y$, we obtain
\begin{align*}
 \stf{v}{x-h(x)}\geq \stf{v}{x-y}\wedge 0 \quad\forall x\in X\enspace .
\end{align*}
Taking the supremum with respect to $y\in S$ 
leads to~\eqref{cw2}.
Similarly, applying~\eqref{ess3} to $h_y$ and replacing $x$ by $x-y$, 
we obtain~\eqref{cw1}.
\end{proof}

Let $(X,\|\cdot\|)$ be an AM-space with unit, denoted by $e$,
and let $f: X\to X$ be a map.
 We shall say that $f$ is \NEW{additively homogeneous} 
(resp.\ \NEW{additively subhomogeneous})
if $f$ is additively homogeneous (resp.\ subhomogeneous) with respect to
$e$, that is (see above)
$f(x+te)=f(x)+te$ (resp.\ $f(x+te)\leq f(x)+te$), 
for all $x\in X$ and $t\in \R$ (resp.\ $t\geq 0$).
We call an \NEW{additive eigenvector} of $f$ a
vector $x\in X$ such that $f(x)= x+\lambda e$,
for some $\lambda \in \R$. 
We denote by $\omega$ the seminorm $\omega_e$ defined in~\eqref{def-omega}.

Additive versions of some results of
Lemmas~\ref{lemma-nonexpan} and \ref{preli-fpv} are easy to check
and are given without proof
(a variant of Lemma~\ref{lemma-nonexpan-add} can be found in~\cite{crandall}).

\begin{lemma}[Compare with~\cite{crandall}]\label{lemma-nonexpan-add}
Let $(X,\|\cdot\|)$ be an AM-space with unit, denoted by $e$, and
$\omega$ denotes the seminorm $\omega_e$ defined in~\eqref{def-omega}.
Let $\psi\in (X^+)^*\setminus\{0\}$.
\begin{enumerate}
\renewcommand{\theenumi}{\rm (\roman{enumi})}
\renewcommand{\labelenumi}{\theenumi}
\item If $F:X\to X$ is order-preserving and additively homogeneous,
then $F$ is nonexpansive with respect to $\|\cdot\|$ and $\omega$.
\item 
If $F:X\to X$ is order preserving and additively subhomogeneous, then
$F$ is nonexpansive with respect to $\|\cdot\|$, and the restriction
$F|_{\psi^{-1}(0)}$ of $F$ to $\psi^{-1}(0)$ is nonexpansive with respect 
to $\omega$.
\end{enumerate}
\end{lemma}

\begin{lemma}\label{preli-fpv-add}
Let $(X,\|\cdot\|)$ be an AM-space with unit.
Let $G$ be an open subset of $X$ and  $F: G\rightarrow X$.
Let $v\in  G$ be a fixed point of $F$: $F(v)= v$. 
Assume that $F$ is semidifferentiable at $v$.
The following implications hold:
\begin{enumerate}
\renewcommand{\theenumi}{\rm (\roman{enumi})}
\renewcommand{\labelenumi}{\theenumi}
\item \label{preli-fpv-add1} If $F$ is order preserving in a neighborhood 
of $v$, then $F'_v:X\to X$ is order preserving.
\item \label{preli-fpv-add12} If $F$ is convex in a neighborhood of $v$,
then $F'_v:X\to X$ is convex.
\item  \label{preli-fpv-add2} If $F$ is additively 
homogeneous in a neighborhood of $v$,
 then $F'_v$ is additively homogeneous.
\item  \label{preli-fpv-add3} If $F$ is additively 
subhomogeneous in a neighborhood of $v$,
 then $F'_v$ is additively subhomogeneous.
\item \label{preli-fpv-add4} If $F$ is nonexpansive with respect to 
$\|\cdot\|$ in a neighborhood of $v$, then $F'_v$ is nonexpansive with 
respect to $\|\cdot\|$.
\end{enumerate}
\end{lemma}

\section{Spectral radius notions and nonlinear Fredholm property}
\label{sec-spec}
Some of our main results rely on some
mild compactness (nonlinear Fredholm type) condition. In order
to discuss it, we first recall the definition of several notions of nonlinear
spectral radius,
as well as some results of~\cite{Nuss-Mallet,AGN-collatz}, 
on generalized measures of noncompactness.

\subsection{Spectral radius, measures of noncompactness, and essential spectral radius}
Let $C$ be a cone of a Banach space $(X, \|\cdot\|)$
and $h$ be a map, homogeneous (of degree 1) and continuous, from $C$ to $C$. 
Following~\cite{Nuss-Mallet}, we define:
\begin{align}
\bonsall C (h) & =
\lim_{k \rightarrow \infty} \|h^k\|_C^{1/k} 
=\inf_{k\geq 1} \|h^k\|_C^{1/k} \qquad \text{ where } \|h^k\|_C \text{ is defined in \eqref{definormh},} \label{eq31-bonsallcspr}
\end{align}
When $C$ is obvious and in particular when $C=X$, $C$ will be 
omitted in the previous notations.
Since $h$ is continuous at $0$, we have $0\leq \bonsall C (h) <+\infty $.
The equality of the limit and the infimum in~\eqref{eq31-bonsallcspr}
follows from $\|h^{k+\ell}\| \leq \|h^k\| \|h^\ell\|$.
The number $\bonsall C (h)$ is called the
\NEW{Bonsall's cone spectral radius} of $h$.
If now $C$ has nonempty interior $\Cint$, we define another
spectral radius:
\begin{align}\label{def-collatz-number}
\supeigen  C (h) & =\inf\, \set{ \lambda >0}{\exists x \in \Cint, \; h(x) \leq
 \lambda x } \enspace .
\end{align}
In~\cite{Nuss-Mallet} and then in~\cite{AGN-collatz} these spectral radii 
are compared with other notions of spectral radius like the
cone spectral radius and the cone eigenvalue spectral radius.
We recall below some of the results of~\cite{AGN-collatz}
that are needed in the following sections.

\begin{lemma}[\citespectral{Lemma}{sp-lem-add5}] \label{collatz-lem72}
Let $C$ be a proper cone of a Banach space $(X, \norm{\cdot})$, 
with nonempty interior,
and let  $h:C\to C$ be a continuous, homogeneous
and order-preserving map.
Then, 
\[ \bonsall C (h)
\leq\supeigen  C (h) \enspace . \]
\end{lemma}
A map $\nu$ from the set of bounded subsets of $X$ to the set of real
nonnegative numbers is called a \NEW{homogeneous generalized measure of noncompactness} if for all bounded subsets $A,B$ of $X$ and for all
real scalars $\lambda$,
\begin{subequations}\label{gmnc}
\begin{align}
\nu(A) = 0 &\Leftrightarrow \clo{A} \text{ is compact}
 \label{gmnc1}\\
\nu(A+B) &\leq \nu(A) + \nu (B) \label{gmnc2}\\
\nu(\clo{\con(A)}) &= \nu (A) \label{gmnc3} \\
\nu(\lambda A) &= |\lambda| \nu(A) \label{gmnc5}\\
A\subset B&\implies \nu(A)\leq \nu(B) \enspace .
\label{gmnc6}
\end{align}
\end{subequations}
We use the notation $\clo{A}$ for the closure
of a set $A$ and $\con{A}$ for its convex hull.
Note that equations~\eqref{gmnc} imply that $\nu(A\cup K)=\nu(A)$ whenever $A$ is a bounded subset of $X$ and $K$ is a compact subset of $X$,
see e.g.\ Prop.~2.2 of~\cite{MPN11}.
They also imply that $\nu(A + K) = \nu(A)$, in particular
\begin{equation}\label{gmnc7}
\nu(A + x) = \nu(A)\quad \text{for all}\; x\in X\enspace.
\end{equation}

For every bounded subset $A$ of $X$,
let $\alpha(A)$ denote the infimum
of all $\delta>0$ such that there
exists an integer $k$ and
$k$ subsets $S_1,\ldots,S_k\subset A$
of diameter at most $\delta$,
such that $A = S_1\cup\cdots\cup S_k$.
The map $\alpha$, introduced by Kuratowski and further studied by Darbo
(see~\cite{Nuss-Mallet} for references), is a particular case
of a homogeneous generalized measure of noncompactness.

If $h$ : ${\mathcal D} \subset X \rightarrow X$ is a map
sending bounded sets to bounded sets,
$\nu$ is a homogeneous generalized  measure
of noncompactness, and $D\subset {\mathcal D}$,
we define
\[
\nu_D(h) =
\inf\set{\lambda > 0}{\nu(h(A)) \leq \lambda \nu(A),\; \mrm{ for all bounded sets } A \subset D} \enspace .
\]
If in addition $h(D) \subset D$, we define :
\[
\rho_D(h) = \lim_{k \rightarrow \infty} (\nu_D(h^k))^{1/k} = \inf_{k \geq 1} \nu_D(h^k)^{1/k}
\enspace .
\]
If $C$ is a cone, $h$: $C \rightarrow C$
is homogeneous and Lipschitz continuous with constant $\kappa$, 
and $\alpha$ is the Kuratowski-Darbo generalized measure of noncompactness
described above, 
then $\alpha_C(h) \leq \kappa$.
A general map $h$: $D \subset X \rightarrow X$, such that 
$\nu_D(h) \leq k < 1$ is called a \NEW{$k$-set contraction}
(with respect to $\nu$).
If $C$ is a cone and $h$ : $C \rightarrow C$ is homogeneous, $\rho_C(h)$ 
is called the \NEW{cone essential spectral radius} of $h$
associated to the homogeneous generalized measure of noncompactness $\nu$.
(We note that a more refined notion of cone essential spectral
radius, which avoids various pathologies which can occur with the above definition,  has been recently developed in~\cite{nussbaummp10}. See also~\cite{nussbaummp11}.)

We saw in Lemma~\ref{preli-fpv} that several elementary properties (including monotonicity) of a map carry over to its semidifferential. The same turns out to be true for properties involving $k$-set contractions.

\begin{proposition}\label{prop2}
Let $f : G \rightarrow X$ be a map,
and let $v \in G$ be a fixed point of 
$f : f(v) = v$. 
Assume that $f$ is semidifferentiable at $v$ with respect
to a closed cone $C$. Then, for any relative neighborhood $U$ of $v$ in $v+C$ such that $U\subset G$ and for any generalized measure of noncompactness $\nu$ on $X$, we have
\begin{subequations}
\begin{align}
\nu_C(f'_v)&\leq \nu_U(f) \enspace. \label{e-ftofp1}
\end{align}
Moreover, if there exists a neighborhood $U$ of $v$ in $v+C$, 
such that $U\subset G$ and $f(U)\subset U$, then $f'_v$ sends $C$ to itself, 
and if in addition $f'_v$ is uniformly continuous on bounded sets of $C$, 
we have
\begin{align}
\rho_C(f'_v)  &\leq \rho_U(f)  \enspace .
\label{e-ftofp2}
\end{align}
\end{subequations}
\end{proposition}
\begin{proof}
Let us prove~\eqref{e-ftofp1}.
Let $S$ be a bounded subset of $C$.
Since $f$ is semidifferentiable at $v$
with respect to $C$, there exists a neighborhood $U$ of $v$ in $v+C$, 
such that $U\subset G$ and for any such a neighborhood $U$,
there exists $t_0>0$ such that $v + tS \subset U$ for $0\leq t\leq t_0$,
since $S$ is bounded in $C$.
Moreover, from Lemma~\ref{lem2} it follows that
\[
\epsilon(t) :=
\sup_{x \in S} t^{-1}\|f(v + tx)- f(v)- tf'_v(x)\|
\to 0\mrm{ when } t \rightarrow 0^+
\enspace .
\]
Then for all $t>0$ such that $t\leq t_0$, we have
\[
tf'_v(x) - f(v + tx) + f(v) \in B(0, t \epsilon(t)),
\mrm{ for all } x \in S\enspace,
\]
which implies that
\[
tf'_v(S) \subset f(v + tS) - f(v) + B(0, t \epsilon (t))
\enspace .\]
Applying $\nu$, and using Properties~\eqref{gmnc5}, \eqref{gmnc6},~\eqref{gmnc7}, \eqref{gmnc2}, 
we get
\begin{align} 
t \nu (f'_v(S))& = 
\nu (tf'_v(S)) \nonumber \\
& \leq  \nu (f(v + tS) - f(v) + B(0, t \epsilon (t))) \nonumber\\
& \leq \nu (f(v + tS)) + 
\nu (B(0, t \epsilon(t))) \nonumber\\
 & \leq \nu (f(v + tS)) + t \epsilon (t)\nu(B(0,1))\enspace .
\label{+1}
\end{align}
Since $v+tS \subset U$ for $t\leq t_0$, we obtain, by definition 
of $\nu_U(f)$,
\[
\nu(f(v + tS)) \leq \nu_U(f) \nu (v + tS)\enspace.
\]
By~\eqref{gmnc7} and~\eqref{gmnc5},
$\nu(v+tS)=\nu(tS)=t\nu(S)$.
We now get from~\eqref{+1}:
\[
\nu(f'_v(S)) \leq \nu_U(f) \nu(S) + \epsilon (t)\nu(B(0,1))
\]
for all $t$ small enough, 
hence $\nu(f'_v(S)) \leq \nu_U(f) \nu(S)$,
which shows~\eqref{e-ftofp1}.
 
Assume now that there exists a neighborhood $U$ of $v$ in $v+C$, 
such that $U\subset G$ and $f(U)\subset U$.
Then, $f(U) \subset U \subset v + C$, and since $C$ is closed,
we deduce that $f'_v(C) \subset C$.
Assume in addition that $f'_v$ is uniformly continuous on bounded sets of $C$
and let us prove~\eqref{e-ftofp2}.
Since $f'_v$ is uniformly continuous on bounded sets of $C$,
so is $(f'_v)^n$, since, by~\eqref{fpvbound},
$f'_v$ sends bounded sets to bounded sets. 
Using the chain rule (Lemma \ref{lemma-chain}),
we obtain by induction that $f^n$ is semidifferentiable at $v$
with respect to $C$, and that $(f^n)'_v = (f'_v)^n$.
Then, applying~\eqref{e-ftofp1} to $f^n$,
we get $\nu_C((f'_v)^n) \leq \nu_U(f^n)$. 
Taking the $1/n$ power and then the infimum over all $n\geq 1$,
we get $\rho_C(f'_v) \leq \rho_U(f)$.
\end{proof}

\subsection{A nonlinear Fredholm-type property}\label{sec-fredholm}

We now introduce a nonlinear Fred\-holm-type property, 
also considered in~\cite{AGN-collatz}, which
will be required to establish the uniqueness
result for fixed points and eigenvectors in 
Sections~\ref{sec-main}--\ref{sec-eig-nonexp}.
If $(X,\|\cdot\|)$ and $(Y,\|\cdot\|)$ are Banach spaces,
$D$ is a subset of $X$, and $g:D\to Y$ is
a map, we shall say that
$g$ has {\em Property~\PF} when
\begin{enumerate}
\item[(F)] any sequence 
$\seq{x_j \in D}{j \geq 1}$, bounded in $X$,
and such that $g(x_j) \rightarrow_{j\to \infty} 0$,
has a convergent subsequence in $X$.
\end{enumerate}
In the point set topology literature, Property~\PF\ corresponds to the 
property that the restriction of $g$ to any closed bounded set of $X$ is
proper at $0$.
If $X$ is finite dimensional, any 
map $g:X\to Y$ has Property~\PF.
When $g$ is linear, $g$ has Property~\PF\ if, and only if,
$g$ is a semi-Fredholm linear operator with index in 
$\Z\cup\{-\infty\}$, which means that $g$ has a finite dimensional
kernel and a closed range,
see for instance~\cite[Proposition~19.1.3]{hormander}
or~\cite[Chapter IV, Theorems~5.10 and~5.11]{kato}.
In the sequel, $\id$ denotes the identity map over any set.

\begin{lemma}[\citespectral{Lemma}{sp-lem-ksetcontract}]
\label{lem-ksetcontract}
If $D$ is a subset of a Banach space $(X,\|\cdot\|)$, and if $h:D\to X$
is a map sending bounded sets to bounded sets and 
such that $\nu_D(h)<1$, then $\id -h$ has
Property~\PF.
\end{lemma}
In the particular case where the homogeneous
generalized measure of noncompactness $\nu$ is equal to 
the Kuratowski-Darbo generalized measure of noncompactness $\alpha$,
and where $h$ is continuous, 
Lemma~\ref{lem-ksetcontract} is a consequence of Corollary~2 in Section~E
of~\cite{MR0312341}, which says more generally that the restriction of
$\id -h$  to any closed bounded set of $X$ is proper (at any point).

\begin{proposition}[\citespectral{Proposition}{sp-prop1}]
\label{prop1}
If $C$ is a cone of a Banach space $(X,\|\cdot\|)$, 
if $h:C \rightarrow C$ is homogeneous
and uniformly continuous on bounded sets of $C$,
and if either $\rho_C(h)< 1$ or $\bonsall C (h)< 1$,
then $\id-h$ has Property $\PF$ on $C$.
Moreover,
when $\bonsall C (h)< 1$, $0$ is the unique fixed point of $h$ in $C$.
\end{proposition}

\section{General uniqueness and convergence results}\label{sec-main}
In this section, we establish the main results of this paper, concerning
the uniqueness of the fixed point of a semidifferentiable
map and the convergence of the orbits to it. The results
are already useful in the finite dimensional case, hence, the reader
might wish at the first reading to ignore the technical compactness
assumptions regarding Property \PF, which are trivially satisfied
in finite dimension.
\subsection{Uniqueness of the fixed point}
We study the uniqueness of the fixed point
$v$ of a nonexpansive map defined on a metric space $(V,d)$. 
We shall need the following assumptions:
\begin{enumerate}
\renewcommand{\theenumi}{\rm (B\arabic{enumi})}
\renewcommand{\labelenumi}{\theenumi}
\item\label{h-1} There exists a Banach space $(\Ev,\nmv)$,
such that $V$ is an open subset of $\Ev$ and that for any $U$ contained in $V$,
$U$ is open in the $d$-metric topology iff it is open in the norm topology
of $\Ev$. In addition, we require that 
\begin{equation}\label{eq-h-1}
d(x,y)\thicksim \nmvf{x-y} \text{ when } x,y\to v
\;\mrm{in}\; V\enspace,
\end{equation}
where by \eqref{eq-h-1}, we mean that for all $\lambda>1$, there exists
a neighborhood $U$ of $v$ in $(V,d)$ such that
\begin{equation}\label{equiv0}
 \frac{1}{\lambda} \nmvf{x-y}\leq d(x,y)\leq \lambda  \nmvf{x-y}\quad
\forall x,y\in U \enspace ;
\end{equation}
\item\label{h-2} The balls of $(V,d)$ are convex in $\Ev$;
\item\label{h-3}  For all $w\in V$, and $s\in (0,1)$,
\begin{equation}\label{gammas}
\Gamma_s:=\set{z\in V}{ d(z,v)= s d(v,w),\; d(z,w)= (1-s) d(v,w)}\neq 
\emptyset\enspace . 
\end{equation}
\end{enumerate}
To interpret assumption~\ref{h-3}, 
let us
recall that a metric space $(V,d)$ is \NEW{strongly metrically convex}
or is a \NEW{geodesic space} if for all $x,y\in V$, 
there exists a (minimal) \NEW{geodesic} from $x$ to $y$,
that is, a continuous path, $z:[0,1]\to V$ such that
$z(0)=x$, $z(1)=y$, and $d(z(s),z(t))=|s-t|d(x,y)$
for all $s,t\in [0,1]$.
Assumption~\ref{h-3} holds when $(V,d)$ is strongly metrically
convex. Conversely, if $(V,d)$ is a complete metric space and if 
Assumption~\ref{h-3} holds for all $v\in V$, or more generally,
if $V$ is \NEW{metrically convex}, which means that 
for all $x,y\in V$, there
exists $w\in V$ such that $x\neq w$, $y\neq w$ and
$d(x,y)=d(x,w)+d(w,y)$, then a theorem
of K. Menger (see~\cite[Th.~14.1, page 41]{blumen})
asserts that $(V,d)$ is strongly metrically convex.

Assumptions~\ref{h-1}, \ref{h-2} and \ref{h-3} are trivially satisfied
when $V=X$ and $d(x,y)=\|x-y\|$ for some Banach space $(X,\|\cdot\|)$.

Almost all the uniqueness results of this paper
will be applications of the following general theorem. 
\begin{theorem}\label{theo0}
Let $(V,d)$ be a complete metric space satisfying~\ref{h-1}--\ref{h-3},
$G$ be an open subset of $V$,
$f:G\to V$ be a nonexpansive map
and $v\in G$ be a fixed point of $f$: $f(v)=v$.
Make the following assumptions:
\begin{enumerate}
\renewcommand{\theenumi}{\rm (A\arabic{enumi})}
\renewcommand{\labelenumi}{\theenumi}
\item\label{h-4} $f:G\to \Ev$ is semidifferentiable at $v$;
\item\label{h-5} The map $\id-f'_v: \Ev\to \Ev$ has Property~\PF;
\item\label{h-6} The fixed point of $f'_v:\Ev\to \Ev$ is unique:
$f'_v(x) = x,\, x \in \Ev\Rightarrow x =0$.
\end{enumerate}
Then, the fixed point of $f$ in $G$ is unique: 
$f(w) = w,\, w \in G\Rightarrow w =v$.
\end{theorem}
Theorem~\ref{theo0} can be remembered
by saying that 
{\em the uniqueness of the fixed point of $f'_v$
implies the uniqueness of the fixed point of $f$},
under assumptions which,
as we will see in Sections~\ref{sec-eig-nonexp} 
and~\ref{sec-eig-nonexp-diff}, are fulfilled
in many situations relative to cones.
In many applications, the map
 $f'_v$ is much ``simpler'' than $f$, and the uniqueness
of the fixed point, \ref{h-6}, can be proved by direct
algebraic or combinatorial means.  
An example of such a situation will be given in~Section~\ref{subsec-games}.

\begin{remark}
Assumption~\ref{h-1} implies that there exists $\epsilon>0$,
such that $U=B(v,\epsilon)$ satisfies~\eqref{equiv0}, 
where in any metric space $(X,d)$, we denote by
$B(x,\rho):=\set{y\in X}{d(y,x)\leq \rho}$ the closed ball 
of $X$ with center $x\in X$ and radius $\rho\geq 0$.
Thus, $B(v,\epsilon)$ is a neighborhood of $v$ in
$(\Ev,\nmv)$, and the completeness of $B(v,\epsilon)$ for $d$ is equivalent to
its completeness for $\nmv$. But for a normed
vector space, the completeness of any closed subset with nonempty interior
is equivalent to the completeness of the space. 
Therefore, the completeness
of $B(v,\epsilon)$ for $d$ is equivalent
to the completeness of $(\Ev,\nmv)$,
and in the assumptions of Theorem~\ref{theo0}, one may either omit
the completeness of $(V,d)$ or that of  $(\Ev,\nmv)$.
\end{remark}

\begin{remark}\label{remark63}
In Theorem~\ref{theo0}, Assumption~\ref{h-5} may be restrictive. Consider
$X=\continuous([0,1])$ the Banach space of continuous functions
from $[0,1]$ to $\R$, endowed with the sup-norm $\|\cdot\|_\infty$,
and let $e$ denote the function identically $1$ on $X$.
Then $X$ is an AM-space with unit $e$ and
pointwise order $\leq$ (see Section~\ref{am-spaces-def}).
Let $F:X\to X$ be defined by $F(x)(t)=x(\frac{t}{2}) \vee (x(t)-1)$.
One checks that $F$ is additively homogenous
($F(x+\alpha e)=F(x)+\alpha e$ for all $\alpha \in\R$), and
order preserving ($F(x)\leq F(y)$ whenever $x\leq y$).
Then, from Lemma~\ref{lemma-nonexpan-add}, $F$ is nonexpansive with respect to
the sup-norm.
Let now $Y=\continuous_0([0,1])\subset X$ be the Banach subspace
of maps $x\in X$ such that $x(0)=0$. The restriction $f:Y\to Y$ 
of $F$ to $Y$ is also nonexpansive with respect to the sup-norm.
Moreover, $v\equiv 0$ is the unique fixed point of $f$,
$f$ is differentiable at $v$ with semidifferential 
$f'_v(x)(t)=x(\frac{t}{2})$, $f'_v$ has $0$ as
a unique fixed point,
whereas $\id -f'_v$ does not have Property~\PF.
Indeed, $x_n(t)=t^{1/n}$ is such that $x_n-f'_v(x_n)$ tends
to $0$ when $n$ goes to infinity, whereas 
$\seq{x_n }{n \geq 1}$ has no convergent subsequence 
(see~\cite{bonsall}).
\end{remark}

As a corollary of Lemma~\ref{lem2b}, we obtain the following
proposition, which gives
in particular a sufficient condition for Assumption~\ref{h-4} of
Theorem~\ref{theo0} to hold.
\begin{proposition}\label{prop-sc-a1hold}
Let $(V,d)$ be a metric space satisfying Assumption~\ref{h-1},
let $G$ be an open subset of $V$, let $v\in V$,
and let $f:G\to V$ be a nonexpansive
map that has directional derivatives at $v$
with respect to all $x\in \Ev$. Then, $\Ev\to \Ev,\; x\mapsto f'_v(x)$
is nonexpansive, and if $\Ev$ is finite
dimensional, $f$ is semidifferentiable at $v$.
\end{proposition}
\begin{proof}
We first prove that the map $x\mapsto f'_v(x)$ 
(defined by~\eqref{eq2'}) is nonexpansive with respect to $\nmv$.
By taking $X=C=\Ev$ and a neighborhood $U$ of $v$ together
with $\lambda>1$ as in~\eqref{equiv0}, 
so that $f$ is Lipschitz of constant
$M=\lambda^2$ in $U$, we get from Lemma~\ref{lem2b}
that $x\mapsto f'_v(x)$ is $\lambda^2$-Lipschitz
$(\Ev,\nmv)\to (\Ev,\nmv)$.
Since this holds for all $\lambda>1$, 
$x\mapsto f'_v(x)$ is nonexpansive $(\Ev,\nmv)\to (\Ev,\nmv)$.
The remaining part of the proposition follows
from Lemma~\ref{lem2b}.
\end{proof}
The proof of Theorem~\ref{theo0} relies on the following
general local uniqueness result, which does not
require $f$ to be nonexpansive.
\begin{lemma}\label{uni-local}
Let $(X,\|\cdot\|)$ be a normed vector space,
$f$ be a map from a subset $G\subset X$
to $X$, and $v$ be a fixed point of $f$ belonging to the
interior of $G$. Make the following assumptions:
\begin{enumerate}
\renewcommand{\theenumi}{\rm (A\arabic{enumi})}
\renewcommand{\labelenumi}{\theenumi}
\item\label{hl4} $f:G\to X$ is semidifferentiable at $v$.
\item\label{hl5} The map $\id-f'_v: {X}\to X$ has Property~\PF;
\item\label{hl6} The fixed point of $f'_v:X\to X$ is unique:
$f'_v(x) = x,\, x \in {X}\Rightarrow x =0$.
\end{enumerate}
Then, there does not exist a sequence $\seq{v_n \in X\setminus\{v\}}{n \geq 1}$
such that
\begin{equation}\label{almost-fixed}
\lim_{n\to \infty}  \|v_n-v\|=0, \mrm{ and } \lim_{n\to \infty}
\frac{\| f(v_n)-v_n\|}{\|v_n-v\|}=0\enspace .
\end{equation}
In particular  $v$ is isolated in the set of fixed points of $f$.
\end{lemma}
\begin{proof}
Let $\seq{v_n \in X}{n \geq 1}$ be
as in~\eqref{almost-fixed}.
Writing $v_n := v + \epsilon_n x_n$ with $\epsilon_n=\|v_n- v\|$,
we get $\lim_{n\to \infty} \epsilon_n=0$ and $\|x_n\|=1$.
Since the sequence $\seq{x_n}{n\geq 1}$ is bounded, it follows from~\eqref{eq3'a}
and~\eqref{eq3'aa},
which are satisfied thanks to~\ref{hl4}, that
\begin{equation}\label{eq36l}
\lim_{n\to \infty} \left\|\frac{f(v_n) - f(v)}{\epsilon_n} - f'_v(x_n)
 \right\| = 0.
\end{equation}
Recalling that $f(v)=v$ and using the second equation in~\eqref{almost-fixed}
together with~\eqref{eq36l}, we get
\begin{equation}\label{eq37l}
\lim_{n\to \infty} \| x_{n} - f'_v(x_{n}) \| = 0\enspace .
\end{equation}
Since the sequence $\seq{x_n }{n \geq 1}$ is bounded,
Assumption~\ref{hl5} implies that
$x_n$ has a subsequence converging to a point $x\in {X}$.
Using the continuity of the semidifferential $f'_v$,
we obtain, from~\eqref{eq37l},
$x - f'_v(x) = 0$. By Assumption~\ref{hl6}, this implies that $x=0$, 
a contradiction with $\|x\|=1$.
\end{proof}

\begin{remark}
Let us give an example where all the assumptions 
of Lemma~\ref{uni-local} are fulfilled except Assumption~\ref{hl5} and
where the fixed point $v$ is not isolated.
Consider $X=\continuous_0([0,1])$ the Banach space of continuous functions
$x$ from $[0,1]$ to $\R$ such that $x(0)=0$, endowed with the
sup-norm, and $f:X\to X$ defined by
$f(x)(t)=x(\frac{t}{2}) +x(1) x(t)$. 
Note that $f$ is not nonexpansive on $X$.
For all $\gamma\in [0,1]$, let $x_\gamma\in X$ be such that
$x_\gamma(t)= (1-(\frac{1}{2})^\gamma) t^\gamma$ for all $t\in [0,1]$.
Then, $x_\gamma$ is a fixed point of $f$ for all $\gamma\in [0,1]$,
$v:=x_0\equiv 0$ and $x_\gamma$ tends to $v$ when $\gamma$ tends to $0$.
Hence, $v$ is a non isolated fixed point of $f$.
The map $f$ is differentiable at $v$ with semidifferential 
$f'_v(x)(t)=x(\frac{t}{2})$, and we already pointed out above
that $f'_v$ has $0$ as a unique fixed point,
whereas $\id -f'_v$ does not have Property~\PF.
\end{remark}

\begin{proof}[Proof of Theorem~\ref{theo0}]
Suppose, by way of contradiction, that there exists $w\in G$ 
such that $f(w)=w$, $w \neq v$. Let $R=d(v,w)>0$, and
for all $s\in (0,1)$, consider the set $\Gamma_s$ defined by~\eqref{gammas}.
Since $G$ is open,  $v\in G$ and $\Gamma_s\subset B(v,Rs)$,
there exists $\bar{s}\in (0,1]$ such that $\Gamma_s\subset G$,
for all $s\leq \bar s$.
By Assumption~\ref{h-3}, for all $s\in (0,1)$,
there exists $z_s\in \Gamma_s$.

Also, we have:
\begin{equation}\label{gammas2}
\Gamma_s=\set{z\in V}{ d(z,v)\leq s d(v,w),\; d(z,w)\leq  (1-s) d(v,w)}\enspace .
\end{equation}
Indeed, if $z\in V$ is such that
$d(z,v)\leq s d(v,w)$ and $d(z,w)\leq  (1-s) d(v,w)$,
we get
\[d(v,w)\leq d(z,v)+d(z,w)\leq s d(v,w)+(1-s) d(v,w)=d(v,w)\enspace,\]
hence $d(z,v)= s d(v,w)$ and $d(z,w)=  (1-s) d(v,w)$.
{From}~\eqref{gammas2} and Assumption~\ref{h-2}, we deduce that
$\Gamma_s$ is convex. 
Because $f$ is nonexpansive with respect to $d$, $f(v) = v$, $f(w) = w$, 
and $\Gamma_s$ is given by~\eqref{gammas2},
we have that $f(\Gamma_s) \subset \Gamma_s$ (for $s\leq\bar{s}$).
This property, together with the convexity of  $\Gamma_s$, allows us
to consider, for $t\in [0,1]$ and $s\in [0,\bar{s}]$,
 the map $f_{s,t}:\Gamma_s\to \Gamma_s$,
\begin{equation}\label{fst}
f_{s,t}(x) = (1-t) f(x) + t z_s \enspace .
\end{equation}

By Assumption~\ref{h-1}, for all $\lambda>1$, there exists 
$0<s_\lambda\leq\bar{s}$
such that:
\begin{equation}\label{equiv}
 \frac{1}{\lambda} \nmvf{x-y}\leq d(x,y)\leq \lambda  \nmvf{x-y}\quad
\forall x,y\in B(v,s R),\;  0<s\leq s_\lambda \enspace .
\end{equation}
Since $\Gamma_{s}\subset B(v,s R)$,
and $f$ is nonexpansive with respect to $d$, we get (using~\eqref{equiv}) that
for all $x,y\in \Gamma_{s}$ and $0<s\leq s_\lambda$,
\begin{eqnarray*}
d(f_{s,t}(x), f_{s,t}(y))&\leq & \lambda \nmvf{f_{s,t}(x)- f_{s,t}(y)}\\
&=& \lambda (1-t) \nmvf{f(x)- f(y)} \\
&\leq & \lambda^2 (1-t) d(f(x),f(y))\\
&\leq &\lambda^2 (1-t) d(x,y)
\enspace .\end{eqnarray*}
Considering
\[t_\lambda=\min (2(1-\frac{1}{\lambda^2}),1)\in (0,1]\enspace ,\]
we get that 
\begin{equation}
\label{tlambda}
 c_\lambda:=\lambda^2 (1-t_\lambda)<1 \text{ and } \lim_{\lambda\to 1^+} 
t_\lambda=0\enspace .\end{equation}
Hence, for all $0<s\leq s_\lambda$ and $t_\lambda\leq t\leq 1$,
$f_{s,t}$ is a contraction mapping
in $(\Gamma_{s},d)$, with contraction factor $c_\lambda$.
Since $\Gamma_{s}$ is closed in $V$, $(\Gamma_{s},d)$
is a complete metric space, hence the contraction mapping principle implies
that $f_{s,t}$ has a unique fixed point 
$v_{s,t}\in \Gamma_{s}$.
By definition,  $v_{s,t}$ satisfies:
\begin{equation}\label{eq35++}
f(v_{s,t}) - v_{s,t}=t (f(v_{s,t})-z_{s})\enspace ,
\end{equation}
which leads to
\begin{eqnarray}
\nmvf{f(v_{s,t}) - v_{s,t}}& =&t \nmvf{f(v_{s,t})-z_{s}}\nonumber\\
& \leq & t \lambda d(f(v_{s,t}),z_{s})\nonumber\\
& \leq & t \lambda ( d(f(v_{s,t}),v) + d(v,z_{s}))\nonumber\\
&\leq &  2 t\lambda  R s \enspace ,\label{ineq}
\end{eqnarray}
since $f(v_{s,t})$ and $z_{s}\in \Gamma_{s}$.

Since $v_{s,t}\in \Gamma_s$ and $\Gamma_s$ is given by~\eqref{gammas},
we get, using~\eqref{equiv}:
\begin{equation}\label{xbound}
\frac{R s}{\lambda} =\frac{d(v,v_{s,t})}{\lambda}\leq
\nmvf{v_{s,t}-v}\leq \lambda d(v,v_{s,t})=\lambda R s \enspace .\end{equation}
The first inequality in~\eqref{xbound}
together with~\eqref{ineq} yield:
\begin{equation}\label{eq35+}
\frac{\nmvf{f(v_{s,t}) - v_{s,t}}}{\nmvf{v_{s,t}-v}} \leq 2t\lambda^2 \enspace .
\end{equation}
Let us choose a sequence $\seq{\lambda_n>1}{n\geq 1}$, 
such that $\lambda_n\to_{n\to\infty} 1$, together with
a sequence $\seq{s_n>0}{n\geq 1}$, such that
 $s_n\leq s_{\lambda_n}$ and $s_n\to 0$.
We set $t_n:=t_{\lambda_n}$
and $v_n:=v_{s_n,t_n}$, to simplify the notation.
By~\eqref{tlambda}, $t_n\to 0$.
Using~\eqref{xbound} and \eqref{eq35+}, we get
\begin{equation}\label{contra}
 \lim_{n\to \infty}  \nmvf{v_n-v}=0,\quad \lim_{n\to \infty}
\frac{\nmvf{ f(v_n)-v_n}}{\nmvf{v_n-v}}=0\enspace .
\end{equation}
Taking $X=\Ev$, $\|\cdot\|=\nmv$, 
we see that Assumptions
\ref{hl4}--\ref{hl6} of 
Lemma~\ref{uni-local} are satisfied due
to Assumptions~\ref{h-4}--\ref{h-6}
of Theorem~\ref{theo0}. 
Thus, the conclusion of Lemma~\ref{uni-local} 
contradicts~\eqref{contra}.
\end{proof}
\begin{remark}
The introduction of $f_{s,t}$ as in~\eqref{fst} and the derivation
of its properties is closely related to
Lemma 2.1 on page 45 of~\cite{nussbaum88}.
\end{remark}

\subsection{Geometric convergence to the fixed point}
When $\bonsall{}(f'_v)<1$,
we can prove a result more precise than Theorem~\ref{theo0}:
the geometric convergence of the orbits of $f$ towards
the fixed point of $f$.
\begin{theorem}\label{th-geom}
Let $(V,d)$ be a complete metric space, $G$ be a 
connected open subset of $V$,
$f:G\to G$ be a nonexpansive map
and $v\in G$ be a fixed point of $f$: $f(v)=v$.
Assume that Assumptions~\ref{h-1} and \ref{h-4} 
of Theorem~\ref{theo0} hold and that 
$f'_v:\Ev\to \Ev$ satisfies $\bonsall{}(f'_v)<1$, where
$\bonsall{}=\bonsall{\Ev}$ is
defined with respect to the $\nmv$ norm. Then, 
\begin{align*}
\limsup_{k\to \infty} d(f^k(x),v)^{1/k}\leq \bonsall{}(f'_v)\quad
\forall x\in G \enspace .
\end{align*}
In particular, the fixed point of $f$ in $G$ is unique.
\end{theorem}
\begin{remark}\label{rem-r1}
Under the assumptions of Theorem~\ref{th-geom}, $f'_v$ is nonexpansive
(by Proposition~\ref{prop-sc-a1hold}), and, a fortiori, 
uniformly continuous on bounded sets of $\Ev$.
Hence, using Proposition~\ref{prop1}, with the assumption that
$\bonsall{}(f'_v)<1$, we get that $\id -f'_v$ has
Property~\PF\ and that the fixed point of $f'_v$
is unique. This shows that Assumptions~\ref{h-5} and~\ref{h-6}
of Theorem~\ref{theo0} are satisfied.
Therefore, Theorem~\ref{theo0} shows the uniqueness of the 
fixed point of $f$ without the connectedness of $G$, as soon as
Assumptions~\ref{h-2} and~\ref{h-3} on $(V,d)$ are satisfied.
\end{remark}
To show Theorem~\ref{th-geom},
we first prove that there is a neighborhood
of $v$ in which all orbits of $f$ converge geometrically
to $v$:
\begin{lemma}
Let $V,d,G,f$ and $v$ be as in Theorem~\ref{th-geom},
and $(\Ev, \nmv)$ be as in~\ref{h-1}.
For all $1>\mu>\bonsall{}(f'_v)$, there exists $\eta>0$ and $m>0$ such that 
\begin{align}
\label{contract}
d(f^m(x),v)\leq \mu^m d(x,v)\quad\forall x\in B(v,\eta) \enspace. 
\end{align}
\end{lemma}
\begin{proof}
Since $\mu> \bonsall{}(f'_v)$, and $\bonsall{}(f'_v)$ is defined
in $(\Ev,\nmv)$, we can chose $m$ such that $\nmvf{(f'_v)^m}< \mu^m$. 
The chain rule for semidifferentiable
maps (Lemma~\ref{lemma-chain}), together with $f(v)=v$,
show that $f^m$ is semidifferentiable at $v$, with $(f^m)'_v=(f'_v)^m$:
\begin{align}
f^m(x)-v=f^m(x)-f^m(v)=(f'_v)^m(x-v)+ o(\nmvf{x-v}) \enspace.
\end{align}
Hence,
\[
\nmvf{f^m(x)-v}\leq \nmvf{(f'_v)^m} \nmvf{x-v}+ o(\nmvf{x-v}) \enspace,
\]
and using~\eqref{eq-h-1}, we get 
that there is a ball $B(v,\eta)$, with $\eta>0$,
in which~\eqref{contract} holds.
\end{proof}

\begin{proof}[Proof of Theorem~\ref{th-geom}]
We get from~\eqref{contract}:
\begin{align}
d(f^{mk}(x),v)\leq \mu^{mk} d(x,v)\quad \forall k\geq 0,\;x\in B(v,\eta)
\enspace .\label{gg}
\end{align}
Moreover, by nonexpansiveness of $f$, 
\begin{align}
d(f^{k+1}(x),v)
=d(f^{k+1}(x),f^{k+1}(v))
\leq d(f^k(x),f^k(v))=
d(f^k(x),v) \enspace.
\label{e-dec}
\end{align}
Combining~\eqref{gg} and~\eqref{e-dec}, we see
that
\begin{align}
\limsup_k d(f^{k}(x),v)^{1/k} \leq \mu\quad\forall x\in B(v,\eta)
\enspace .\label{gg2}
\end{align}
Let
\[\Omega=\set{x\in G}{\lim_k f^k(x)=v}
\enspace .
\]
We claim that $\Omega=G$. 
By~\eqref{gg2}, $B(v,\eta)\subset \Omega$, hence $\Omega\neq\emptyset$.
We will show that $\Omega$ is both
open and closed in $G$. We claim that:
\begin{align}
B(x,\eta/2)\cap G \subset \Omega \quad \forall x\in \Omega \enspace .
\label{fixedball}
\end{align}
Indeed, if $x\in \Omega$, we have
$d(f^n(x),v)<\eta/2$ for some $n$, hence,
for all $y\in G$ such that $d(y,x)\leq \eta/2$,
$d(f^n(y),v)\leq d(f^n(y),f^n(x))+d(f^n(x),v)
<\eta$, and by~\eqref{gg2}, $\lim_k f^k(y)=v$,
which shows~\eqref{fixedball}, and, a fortiori,
that $\Omega$ is open. Property~\eqref{fixedball} 
also implies that $\Omega$ is closed. Indeed,
if $\seq{x_n\in \Omega}{n\geq 1}$
is a sequence converging to $x\in G$, we have $d(x_n,x)<\eta/2$
for some $n\geq 1$, hence, $x\in \Omega$ by~\eqref{fixedball}.
We have shown that $\Omega$ is nonempty, closed and open in $G$,
and since $G$ is connected, $\Omega=G$.
Finally, since $\lim_k f^k(x)=v$ for all $x\in G$,
and since~\eqref{gg2} holds, we have 
$\limsup_k d(f^k(x),v)^{1/k}\leq \mu$,
for all $x\in G$. Since this inequality holds
for all $\mu>\bonsall{}(f'_v)$, we have proved the convergence in
Theorem~\ref{th-geom}, which itself implies that $v$ is the unique 
fixed point of $f$ in $G$.
\end{proof}

\section{Fixed points and eigenvectors of semidifferentiable nonexpansive maps
over normal cones}
\label{sec-eig-nonexp}
In this section, we derive from Theorem~\ref{theo0}
some uniqueness results for fixed points or eigenvectors
of maps acting on cones, as well as geometric convergence
results for the orbits.

As a first application, we get the following uniqueness result
for the fixed point
of a nonexpansive map in the Thompson's metric.
\begin{theorem}[Uniqueness of fixed points]\label{theo00d}
Let $C$ be a normal cone with nonempty interior  in 
a Banach space $(X,\|\cdot\|)$. Let $G$ be an open subset of $(X,\|\cdot\|)$ included in $C$ and
$f : G \rightarrow \Cint$ be a nonexpansive map with 
respect to Thompson's metric $\tho$.
Let $v\in  G$ be a fixed point of $f$: $f(v)=v$.
Make the following assumptions:
\begin{enumerate}
\renewcommand{\theenumi}{\rm (A\arabic{enumi})}
\renewcommand{\labelenumi}{\theenumi}
\item\label{h00d4} $f$ is semidifferentiable at $v$;
\item\label{h00d5} The map $\id-f'_v: X\to X$  has Property \PF;
\item\label{h00d6} The fixed point of $f'_v:X\to X$ is unique:
$f'_v(x) = x,\, x \in X\Rightarrow x =0$.
\end{enumerate}
Then, the fixed point of $f$ in $G$ is unique: 
$f(w) = w,\, w \in G\Rightarrow w =v$.
\end{theorem}
\begin{proof}
We apply Theorem~\ref{theo0}.
Let $V=\Cint$ endowed with $\tho$.
Since $v\in \Cint$, $C_v=\Cint=V$ and by
Proposition~\ref{prop-thompson}, $(V,\tho)$ is a 
complete metric space.
Moreover, by Proposition~\ref{prop-normequiv}, 
$\|\cdot\|$, $\|\cdot\|_v$ and $\tho$ define the same topology on $V$,
hence the set $G$ is open in $(V,\tho)$.
Since the equivalence~\eqref{e-normequivdbar} 
holds, $(V,\tho)$ satisfies~\ref{h-1} with $\Ev=X_v$
and $\nmv=\|\cdot\|_v$.
{From} the definition of $\tho$ and the convexity of $C$,
it follows that the metric space $(V,\tho)$ satisfies Assumption~\ref{h-2}.
Assumption~\ref{h-3} for $(V,\tho)$  follows from the existence
of a (minimal) geodesic for $\tho$
between any two points of $\Cint$, which is
proved in~\cite[Proposition~1.12, page 34]{nussbaum88}.
The map $f$ is nonexpansive with respect to $\tho$
and $v$ is a fixed point of $f$.
Then, Assumptions~\ref{h00d4}--\ref{h00d6}
of Theorem~\ref{theo00d}
correspond to Assumptions~\ref{h-4}--\ref{h-6}
of Theorem~\ref{theo0}, respectively, which
yields the conclusion of the theorem.
\end{proof}
\begin{theorem}[Geometric convergence]\label{theo62}
Let $C$ be a normal cone with nonempty interior
in a Banach space $(X,\|\cdot\|)$,
and let $G$ be an open subset of $(X,\|\cdot\|)$ included in $C$.
Assume that $f: G\to G$ 
is nonexpansive with respect to Thompson's metric $\tho$.
If $f$ has a fixed point $v\in G$, and if $f$ is
semidifferentiable at $v$, with $\bonsall{} (f'_v)<1$, 
where $\bonsall{}$ is defined with respect to the $\|\cdot\|$ or the
$\|\cdot\|_v$ norm, then the fixed point of $f$ in $G$ is unique.
Moreover, if $G$ is connected, we have:
\begin{align*}
\forall x\in G,\;\; \limsup_{k\to \infty} \tho(f^k(x),v)^{1/k}\leq 
\bonsall{} (f'_v) \enspace .
\end{align*}
\end{theorem}
\begin{proof}
By the proof of Theorem~\ref{theo00d}, $(\interior C,\tho)$ satisfies~\ref{h-1}.
By  Proposition~\ref{prop-sc-a1hold}, 
$f'_v$ is nonexpansive, and, a fortiori, uniformly continuous on bounded
sets of $C$. Therefore, Proposition~\ref{prop1} implies 
that Assumptions~\ref{h00d5} and~\ref{h00d6} of Theorem~\ref{theo00d}
are satisfied. Since the other assumptions of Theorem~\ref{theo62}
imply the other assumptions of Theorem~\ref{theo00d}, we get
the first assertion of Theorem~\ref{theo62}. (See also Remark~\ref{rem-r1}.)
The last one is
a direct corollary of Theorem~\ref{th-geom}.
\end{proof}

\begin{theorem}[Uniqueness of fixed points in $\Sigma$]\label{theo00}
\sloppy
Let $C$ be a normal cone with nonempty interior  in 
a Banach space $(X,\|\cdot\|)$, let $\psi\in C^*\setminus\{0\}$, and 
$\Sigma=\set{x\in \Cint}{\psi(x)=1}$.
Let $G$ be a relatively open subset of $\Sigma$ and
$g : G \rightarrow \Sigma$ be a nonexpansive map with 
respect to Hilbert's projective metric $d$
or Thompson's metric $\tho$.
Let $v\in  G$ be a fixed point of $g$: $g(v)=v$.
Make the following assumptions:
\begin{enumerate}
\renewcommand{\theenumi}{\rm (A\arabic{enumi})}
\renewcommand{\labelenumi}{\theenumi}
\item\label{h004} $g$ is semidifferentiable at $v$ with respect to 
$\psi^{-1}(0)$; 
\item\label{h005} The map $\id-g'_v: \psi^{-1}(0)\to \psi^{-1}(0)$
 has Property \PF;
\item\label{h006} The fixed point of $g'_v:\psi^{-1}(0)\to \psi^{-1}(0)$ is unique:
$g'_v(x) = x,\, x \in \psi^{-1}(0)\Rightarrow x =0$.
\end{enumerate}
Then, the fixed point of $g$ in $G$ is unique: 
$g(w) = w,\, w \in G\Rightarrow w =v$.
\end{theorem}
\begin{proof}
We apply Theorem~\ref{theo0}.
Let $V=-v+\Sigma$. We consider on $V$ the two following metrics
$d_v(x,y) :=d(v+x,v+y)$ and $\tho_v(x,y) :=\tho(v+x,v+y)$.
Let $G_v=-v+G$ and consider 
the map $f:G_v\to V$, $f(x)=g(v+x)-v$.
If $g$ is nonexpansive with respect to $d$ (resp.\ $\tho$), then
$f$ is nonexpansive with respect to
$d_v$ (resp.\ $\tho_v$), and satisfies $f(0)=0$. 

We know, by Proposition~\ref{prop-birkhoff}, that $(\Sigma, d)$ and
$(\Sigma,\tho)$, hence $(V,d_v)$ and $(V,\tho_v)$ are complete metric spaces.
Since $v\in \Cint$, Proposition~\ref{prop-normequiv}, together with
\eqref{e-equivm}, \eqref{equiv-omega0}, and~\eqref{equiv-omega1}
show that $\|\cdot\|$, $\omega_v$,  $\|\cdot\|_v$,
$d$ and $\tho$ define the same topology on $\Sigma$.
Hence, $\|\cdot\|$, $\omega_v$, $\|\cdot\|_v$, $d_v$ and $\tho_v$ define
the same topology on $V$.
Since the equivalence~\eqref{e-normequivosc}
(resp.~\eqref{e-normequivdbar}) holds, $(V,d_v)$
(resp.\ $(V, \tho_v)$) satisfies~\ref{h-1} with $\Ev=\psi^{-1}(0)$
and $\nmv=\omega_v$ (resp.\ $\nmv=\|\cdot\|_v$).
Moreover, the set $G_v$ is open in $(V,d_v)$
and $(V,\tho_v)$.
{From} the definition of $d$, $\tho$ and the convexity of $C$,
it follows that the metric spaces $(V,d_v)$ and $(V,\tho_v)$ both satisfy 
Assumption~\ref{h-2}.
Assumption~\ref{h-3} for $(V,d_v)$ (resp.\ $(V,\tho_v)$)
 follows from the existence
of a (minimal) geodesic for $d$ (resp.\ $\tho$) 
between any two points of $\Sigma_v$, which is
proved in~\cite[Proposition~1.9, page 25]{nussbaum88}
(resp.~\cite[Proposition~1.12, page 34]{nussbaum88}).
Then, Assumptions~\ref{h004}--\ref{h006}
of Theorem~\ref{theo00} correspond to
Assumptions~\ref{h-4}--\ref{h-6} of Theorem~\ref{theo0}, which
yields the conclusion of Theorem~\ref{theo00}.
\end{proof}

\begin{remark}
The proof of Theorem~\ref{theo00d} 
remains valid 
if one replaces $\Cint$ by the cone $C_u$ for some $u\in C\setminus\{0\}$
where $C$ is still a normal cone,
while assuming that $G$ is an open subset
of $(C_u,\tho)$, and replacing
$\|\cdot\|$ by $\|\cdot\|_u$. Similarly, the statements
of Theorem~\ref{theo00},
and also of Theorem~\ref{theo01} and Corollaries~\ref{theo1}, \ref{coro1},
and~\ref{coro2} below, have obvious extensions
applying to $C_u$ or $\Sigma_u$.
\end{remark}

In order to obtain
a uniqueness result for an eigenvector of $f$,
we shall apply Theorem~\ref{theo00} 
to the map $g$ introduced in Lemma~\ref{ftog}.
Natural assumptions on $f$ which ensure
the nonexpansiveness of $g$ together with
Assumptions~\ref{h004}--\ref{h006}
in Theorem~\ref{theo00}, are captured in the following result.
\begin{theorem}[Uniqueness of eigenvectors in $\Sigma$]\label{theo01}
\sloppy
Let $C$ be a normal cone with nonempty interior  in 
a Banach space $(X,\|\cdot\|)$, let $\psi\in C^*\setminus\{0\}$, and denote
$\Sigma=\set{x\in \Cint}{\psi(x)=1}$.
Let $G$ be an open subset of $(X,\|\cdot\|)$ included in $C$ and
$f : G  \rightarrow \Cint$ be a map such that
$f|_{G\cap \Sigma}$ is nonexpansive with respect to Hilbert's projective 
metric $d$.
Assume that $v\in  G\cap \Sigma$ is a fixed point of $f$: $f(v)=v$.
Make the following assumptions:
\begin{enumerate}
\renewcommand{\theenumi}{\rm (A\arabic{enumi})}
\renewcommand{\labelenumi}{\theenumi}
\item\label{h04} $f$ is semidifferentiable at $v$;
\item\label{h05} 
The map $(\id-f'_v)|_{\psi^{-1}(0)}: \psi^{-1}(0)\to X$ has Property \PF;
\item\label{h06} The fixed point of $f'_v$ in $\psi^{-1}(0)$ is unique:
$f'_v(x) = x,\, x \in \psi^{-1}(0) \Rightarrow x =0$;
\item\label{h07} $f'_v$ is order preserving;
\item\label{h08} $f'_v$ is additively subhomogeneous with respect to $v$.
\end{enumerate}
Then, the eigenvector of $f$ in $G\cap \Sigma$ is unique: 
$\exists \lambda>0,\; f(w) = \lambda w,\, w \in G\cap \Sigma\Rightarrow w =v$.
\end{theorem}
\begin{proof}
Let $\tilde{f}:G \rightarrow \Sigma$ be defined by~\eqref{defg},
$\tilde{f}(x)=\frac{f(x)}{\psi(f(x))}$,  and 
$g=\tilde{f}|_{G\cap \Sigma}:G\cap \Sigma\to \Sigma$.
We shall prove that $g$ satisfies the assumptions of Theorem~\ref{theo00}
with $G$ replaced by $G\cap\Sigma$.
First, since $f|_{G\cap \Sigma}$ is nonexpansive with respect to $d$, 
so is $g$ (by Lemma~\ref{ftog}).
Since, $f(v)=v$ and $v\in G\cap \Sigma$,  $g(v)=v$.
{From} Assumption~\ref{h04} of Theorem~\ref{theo01}
and Lemma~\ref{ftog}, we get Assumption~\ref{h004}
of Theorem~\ref{theo00}.
It remains to check Assumptions~\ref{h005} and \ref{h006}
of Theorem~\ref{theo00}.

Let us first show that
\begin{equation}\label{eqth3}
\| x-f'_v(x)\|_v\leq 2 \, \| x-g'_v(x)\|_v \quad \forall x\in \psi^{-1}(0)
\enspace .
\end{equation}
Since $f'_v(0)=0$ and, by~\ref{h07} and \ref{h08}, $f'_v$ is
order preserving and additively subhomogeneous with respect to $v$,
$h=f'_v$ satisfies the assumptions of Lemma~\ref{lemma-ess} with $S\ni 0$.
Hence, from~\eqref{cw1} and \eqref{cw2}, we obtain, for all $x\in X$,
\begin{align}
\label{cw1p}
\sbf{v}{x-f'_v(x)} &\leq \sbf{v}{x}\vee 0\\
\stf{v}{x-f'_v(x)} &\geq \stf{v}{x}\wedge 0 \enspace . \label{cw2p}
\end{align}
Since $\psi\in C^*$, $v\in \Sigma$ and $\sbf{v}{x} v\leq x\leq \stf{v}{x} v$,
we get $\sbf{v}{x} \leq \psi(x)\leq \stf{v}{x}$. Hence, 
for all $x\in \psi^{-1}(0)$,  $\sbf{v}{x} \leq 0\leq \stf{v}{x}$
which with \eqref{cw1p} and \eqref{cw2p} leads to 
\begin{equation}\label{eqth1}
\sbf{v}{x-f'_v(x)} \leq 0\leq \stf{v}{x-f'_v(x)}\enspace .
\end{equation}
Denote $y=x-g'_v(x)$. By~\eqref{gprimv}
 and  $g'_v=\tilde{f}'_v|_{\psi^{-1}(0)}$,
we get 
\begin{equation}\label{eqth5}
x-f'_v(x)=y-\psi(f'_v(x)) v\enspace ,
\end{equation}
thus
$\sbf{v}{x-f'_v(x)}=\sbf{v}{y}-\psi(f'_v(x))$ and
$\stf{v}{x-f'_v(x)}=\stf{v}{y}-\psi(f'_v(x))$.
Using~\eqref{eqth1}, we get
$\sbf{v}{y}\leq \psi(f'_v(x))\leq \stf{v}{y}$, hence
\begin{equation}\label{eqth4}
|\psi(f'_v(x))|\leq \|y\|_v\enspace .
\end{equation}
Gathering~\eqref{eqth5} and \eqref{eqth4}, we get
\[ \|x-f'_v(x)\|_v\leq\| y\|_v +|\psi(f'_v(x))|\leq 2\| y\|_v \enspace ,\]
which shows~\eqref{eqth3}.

Using~\eqref{eqth3}, we get that, if $g'_v(x)=x$ and $x\in \psi^{-1}(0)$, 
then $f'_v(x)=x$, whence by Assumption~\ref{h06} in Theorem~\ref{theo01},
$x=0$. This shows Assumption~\ref{h006} in Theorem~\ref{theo00}.
Let now $\seq{x_j}{j\geq 1}$ be a bounded sequence in
$\psi^{-1}(0)$ such that 
$x_j-g'_v(x_j)\rightarrow_{j\to \infty} 0$.
By~\eqref{eqth3}, this implies that 
$x_j-f'_v(x_j)\rightarrow_{j\to \infty} 0$,
whence,  by Assumption~\ref{h05} of Theorem~\ref{theo01},
$\seq{x_j}{j\geq 1}$ admits
a convergent subsequence. This shows Assumption~\ref{h005}
of Theorem~\ref{theo00},
and completes the verification of the assumptions of this theorem.

If $w\in G\cap \Sigma$ satisfies $f(w) = \lambda w$ for some $\lambda>0$,
then $g(w)=w$, hence, by Theorem~\ref{theo00}, $w=v$.
\end{proof}

\begin{remark}
As a consequence of Lemma~\ref{preli-fpv}, Assumption~\ref{h07} 
of Theorem~\ref{theo01} is fulfilled
as soon as $f$ is order preserving in a neighborhood of $v$,
and Assumption~\ref{h08} is fulfilled as soon as $f$ is subhomogeneous in
a neighborhood of $v$.
These properties also imply, by Lemma~\ref{lemma-nonexpan}, 
that $f|_{G\cap \Sigma}$ is nonexpansive with respect to Hilbert's
projective metric $d$.
They will be used in particular in Corollary~\ref{theo1}.
Another way to ensure~\ref{h08} is given by 
Lemma~\ref{preli-fpv33}.
In particular, \ref{h08}  is fulfilled when 
$f'_v$ is linear (that is $f$ is differentiable at $v$)
and $f'_v(v)\leq v$.
Moreover, by  Lemma~\ref{preli-fpv},~\ref{preli-fpv12}, 
Assumption~\ref{h08} 
is fulfilled when $f$ is convex and $f'_v(v)\leq v$.
\end{remark}

\begin{corollary}[Uniqueness of eigenvectors]\label{theo1}
Let $C$ be a normal cone with nonempty interior  in 
a Banach space 
$(X,\|\cdot\|)$. Let $f : \Cint \rightarrow \Cint$ be 
homogeneous and order-preserving.
Let $S = \{x \in \Cint \mid f(x) = x \}$, and assume
that $v \in S$. Make the following assumptions:
\begin{enumerate}
\renewcommand{\theenumi}{\rm (A\arabic{enumi})}
\renewcommand{\labelenumi}{\theenumi}
\item\label{h4} $f$ is semidifferentiable at $v$; 
\item\label{h5} The map $\id-f'_v: X \to X$ has Property \PF;
\item\label{h6} if $f'_v(x) = x$ for some $x \in X$, then $x \in 
\set{ \lambda v}{ \lambda \in \R}$;
\end{enumerate}
Then, $S = \{\lambda v \mid \lambda > 0 \}$.
\end{corollary}
\begin{proof}
Let us first check that $f$ satisfies the assumptions of Theorem~\ref{theo01}.
Consider $G=\Cint$ which is clearly open.
Since $v\in C\setminus\{0\}$, one can choose $\psi\in C^*\setminus\{0\}$
such that $\psi(v)=1$. In that case, $v\in \Sigma=G\cap \Sigma$ and $f(v)=v$.
Since $f$ is order preserving and homogeneous, $f|_\Sigma$ is
nonexpansive with respect to $d$, by Lemma~\ref{lemma-nonexpan},\ref{lemma-nonexpan-1}.
Assumption~\ref{h4} of Corollary~\ref{theo1} corresponds to
Assumption~\ref{h04} of Theorem~\ref{theo01}.
Assumption~\ref{h5} of Corollary~\ref{theo1} implies
Assumption~\ref{h05} of Theorem~\ref{theo01}.
In order to show Assumption~\ref{h06} of Theorem~\ref{theo01},
let us consider $x\in X$ such that
 $f'_v(x)=x$ and $\psi(x)=0$. By Assumption~\ref{h6}
of Corollary~\ref{theo1},
$x=\lambda v$ for some $\lambda\in\R$, and since $\psi(x)=0$,
we get $\lambda=0$, thus $x=0$, which shows Assumption~\ref{h06}
of Theorem~\ref{theo01}. 
Assumptions~\ref{h07} and \ref{h08} of Theorem~\ref{theo01} are deduced from
Lemma~\ref{preli-fpv}, \ref{preli-fpv1} and \ref{preli-fpv2}
 respectively, using the fact that $f$ is order preserving and homogeneous.
This completes the proof of the assumptions of Theorem~\ref{theo01}.

Let $x\in S$, then $\lambda=\psi(x)>0$, and,
since $f$ is homogeneous,  $y=\frac{x}{\lambda}$ satisfies $f(y)=y$
and $y\in \Sigma$.
{From} Theorem~\ref{theo01}, this implies $y=v$, hence $x=\lambda v$.
Conversely, if $x=\lambda v$ for some $\lambda>0$, then $x\in S$,
since $f$ is homogeneous.
\end{proof}

If $f$ is homogeneous, $f(v)=v$ and 
$f$ satisfies Assumption~\ref{h4} of Corollary~\ref{theo1},
we have trivially $f'_v(\lambda v)= \lambda v$, 
for all $\lambda\in\R$. Therefore, Assumption~\ref{h6}
of Corollary~\ref{theo1} 
can be thought of as a uniqueness assumption
for the eigenvector of $f'_v$, and
Corollary~\ref{theo1} states in essence that 
the uniqueness of the eigenvector of $f'_v$
implies the uniqueness of the eigenvector
of $f$.

Let us state a corollary of Theorem \ref{th-geom} 
in the framework of Corollary \ref{theo1}.
If $X$ is a vector space endowed with a 
seminorm $\omega$ and $h$ is a homogeneous self-map of a cone 
$C$ of $X$, we generalize \eqref{definormh} and \eqref{eq31-bonsallcspr} by
setting:
\begin{equation}\label{6*}
\omega_C(h) = \sup_{\substack{x \in C \\ \omega(x) \neq 0}}
 \frac{\omega(h(x))}{\omega(x)} = 
\sup_{\substack{x \in C \\ \omega(x) = 1}} \omega(h(x)) \in [0, + \infty]
\end{equation}
and
\begin{equation}\label{6**}
\bonsall{C} (h) = \lim_{k \rightarrow \infty} \omega_C (h^k)^{1/k}.
\end{equation}
Again, when $C=X$, we omit $C$ in these notations.

\begin{theorem}\label{th-6.8}
Let $C$ be a normal cone with nonempty interior  in 
a Banach space $(X,\|\cdot\|)$.
Let $f : \Cint \rightarrow \Cint$ be homogeneous and order-preserving.
Let $S =\set{x \in \Cint}{f(x) = x }$, and assume
that $v \in S$.
Assume that $f$ is semidifferentiable at $v$ with $\bonsall{} (f'_v) < 1$,
where $\bonsall{}$ is defined as in {\rm (\ref{6*},\ref{6**})} with respect to
the seminorm $\omega_v$. Then,
for all $x \in \Cint$, 
\begin{equation}\label{6*3}
\quad \limsup_{k \rightarrow \infty} d(f^k(x), v)^{1/k} \leq \bonsall{}(f'_v)
\enspace,
\end{equation}
where $d$ is the Hilbert's projective metric, and
there is a scalar $\lambda > 0$ (depending on $x$)
such that
\begin{equation}\label{6*4}
\limsup_{k\to\infty} \tho(f^k(x), \lambda v)^{1/k} 
\leq \bonsall{}(f'_v) \enspace,
\end{equation}
where $\tho$ is the Thompson's metric.
\end{theorem}
\begin{proof}
Consider as in the proof of Corollary \ref{theo1},
$G =\Cint$ and $\psi \in C^*\backslash \{0\}$ such that $\psi(v)=1$.
Denote again $\Sigma = \set{x \in \Cint}{\psi(x) = 1}$.
Since $f$ is homogeneous, the functions
$\tilde{f}$ of \eqref{defg} and $g = \tilde{f} \mid_\Sigma$ satisfy :
\[f^k(x) = \psi (f^k(x)) \tilde{f}^k(x) \text{ and } \tilde{f}^k(x) = g^k \left(\frac{x}{\psi(x)} \right)
\enspace .
\]
Hence
\begin{equation}\label{6*5}
d(f^k(x),v) = d \left( g^k \left(\frac{x}{\psi(x)} \right), v \right).
\end{equation}
Moreover, by Lemma~\ref{preli-fpv},\ref{preli-fpv2},
$f'_v$ is additively homogeneous with respect to $v$,
and using~\eqref{gprimv}, we get by 
induction on $k$:
\begin{align}\label{induc-imme}
(\tilde{f}'_v)^k (x) & = (f'_v)^{k}(x) - \psi ((f'_v)^k(x))v 
 = (g'_v)^{k}(x - \psi(x)v)\enspace.
\end{align}
Indeed, for any $y\in G$, $\tilde{f}'_v(y)$ is 
the unique vector in $\psi^{-1}(0)$ of the form $f'_v(y)-\lambda v$ with
$\lambda\in\R$. In particular $(\tilde{f}'_v)^{k+1} (x)$ is the unique element of
$\psi^{-1}(0)$ of the form $f'_v((\tilde{f}'_v)^k (x))-\lambda v$.
If~\eqref{induc-imme} holds for $k$ then using the
property that $f'_v$ is additively homogeneous with respect to $v$,
we get that $(\tilde{f}'_v)^{k+1} (x)=(f'_v)^{k+1} (x)-\lambda' v$
for some $\lambda'\in\R$ and since it belongs to $\Sigma$ we 
deduce~\eqref{induc-imme} for $k+1$.
Hence
\begin{align*}
\omega_v ((f'_v)^k(x)) & = \omega_v((\tilde{f}'_v)^k(x)) 
 = \omega_v ((g'_v)^k (x - \psi(x) v)) \enspace .
\end{align*}
Therefore,
\begin{align*}
\omega((f'_v)^k) & = \sup_{\substack{x \in X \\ \omega_v(x) \neq 0}}
\frac{\omega_v((g'_v)^k(x - \psi(x)v))}{\omega_v(x - \psi(x)v)} 
 =  \sup_{\substack{x \in \psi^{-1}(0) \\ \omega_v(x) \neq 0}}
\frac{\omega_v( (g'_v)^k(x))}{\omega_v(x)} = \omega((g'_v)^k)
\end{align*}
which shows that
\begin{equation}\label{6*6}
\bonsall{}(f'_v) =\bonsall{}(g'_v)\enspace.
\end{equation}
{From}~\eqref{6*5} and~\eqref{6*6}, \eqref{6*3} is equivalent to
\begin{equation}\label{6*7}
\forall x\in \Sigma,\qquad
\limsup_{k \rightarrow \infty} d(g^k(x),v)^{1/k} \leq \bonsall{}(g'_v)
\enspace.
\end{equation}
Since $\omega_v$ is a norm on $\psi^{-1}(0)$
and $g:\Sigma\to\Sigma$ 
is nonexpansive with respect to $d$ (by Lemma \ref{ftog}),
\eqref{6*7} is obtained by applying Theorem~\ref{th-geom},
using the same transformations
as in the first paragraph of the proof of Theorem~\ref{theo00}.

We now prove~\eqref{6*4}. By homogeneity of $f$, it is enough
to consider the case when $x\in \Sigma$.
Then, since $d$ and $\tho$ are equivalent on $\Sigma$,
we get by \eqref{6*7} and~\eqref{6*6},
\begin{align}
\limsup_{k\to\infty} \tho (g^k(x),v)^{1/k} \leq \bonsall{}(g'_v) =\bonsall{}(f'_v)< 1 \enspace .
\label{thocvg}
\end{align} 
To derive~\eqref{6*4} from~\eqref{thocvg},
we write $f^k(x)$ as a function of the orbit of $x$ under $g$,
\begin{align}
f^k(x)= \psi(f\comp g^{k-1}(x))
\cdots \psi (f\comp g^0(x)) g^k (x) 
\label{e-cocycle} 
\end{align}
(this formula is readily checked by induction on $k\geq 1$,
using the homogeneity of $f$ and the definition
of $g$; this is an instance of the well known
``1-cocycle'' representation of iterates of homogeneous maps,
see e.g.~\cite{furstenberg63}).
We still denote by $\tho$ the Thompson's metric
on the open cone of strictly positive real numbers:
\[
\tho(\nu,\mu)=|\log \nu -\log \mu| ,\qquad \forall \nu,\mu>0 \enspace .
\]
Since $f$ is order preserving and homogeneous,
and since $x \leq y \implies \psi(x)\leq \psi(y)$,
and $\psi$ is linear, we have:
\begin{align}
\label{psicompf-nonexp}
\forall x,y\in \interior C,\qquad 
\tho(\psi\comp f(x),\psi\comp f(y)) \leq \tho(x,y) \enspace .
\end{align}
(This follows from the standard argument
of the proof of Lemma~\ref{lemma-nonexpan}:
we have  $\exp(-\tho(x,y))y \leq x\leq\exp(\tho(x,y))y$,
by definition of Thompson's metric, and applying
$\psi\comp f$, we get~\eqref{psicompf-nonexp}.)
Let $\mu_k= \psi( f\comp g^k(x))$. 
Since $f(v)=v$ and $\psi(v)=1$, 
it follows from~\eqref{psicompf-nonexp}
that
\begin{align*}
|\log \mu_k| = \tho(\mu_k, 1)
= \tho(\psi\comp f(g^k(x)),\psi\comp f(v))
\leq \tho(g^k(x),v) \enspace .
\end{align*}
Using~\eqref{thocvg},
we deduce:
\begin{align}
\limsup_{k\to\infty} |\log\mu_k|^{1/k} \leq \bonsall{}(g'_v) <1 \enspace .
\label{e-mukcvg}
\end{align}
Hence, the series $\sum_{k\geq 0} \log \mu_k$
is absolutely convergent, which implies
that the infinite product $\lambda =\prod_{k\geq 0} \mu_k$
is convergent (with $\infty>\lambda>0$). We get
from~\eqref{e-cocycle}:
\begin{align*}
\tho(f^k(x),\lambda v)&=
\tho\Big( (\prod_{0\leq m\leq k-1}\mu_m )g^k(x),\lambda v\Big)\\
&\leq 
\tho\Big( (\prod_{0\leq m\leq k-1}\mu_m) g^k(x),\lambda g^k(x)\Big)
+\tho(\lambda g^k(x),\lambda v)\\
&\leq (\sum_{m\geq k} |\log \mu_m|)
+ \tho(g^k(x),v) \enspace,
\end{align*}
and combining~\eqref{e-mukcvg} with~\eqref{thocvg},
we get~\eqref{6*4}.
\end{proof}

\section{Fixed points and eigenvectors of semidifferentiable nonexpansive maps
over proper cones}\label{sec-fix-proper}

As pointed out in Remark~\ref{remark63}, the map
$f:\continuous_0([0,1])\to\continuous_0([0,1])$ defined by $f(x)(t)=x(\frac{t}{2}) 
\vee (x(t)-1)$ and which has $v\equiv 0$ as a unique fixed point,
is such that $\id -f'_v$ does not have Property~\PF\ on $\continuous_0([0,1])$.
One can show however that $\id -f'_v$ has Property~\PF\ on 
the Banach space $\continuous_0^\gamma ([0,1])$ of H\"older continuous functions
$x:[0,1]\to\R$ with exponent $\gamma$, such that $x(0)=0$.
Indeed, since $f'_v$ is linear in this case, it suffices to prove,
as is noted in Section~\ref{sec-fredholm}, that $f'_v$ such that
$f'_v(x)(t)=x(\frac{t}{2})$ is Fredholm of index $0$. 
However, this is a special case of results which are proved for a 
much more general class of linear maps in \cite[Section 5]{nussbaum-period}.
Alternately, in this simple case, the reader can try a direct proof.
Since the positive cone of  $\continuous_0^\gamma([0,1])$ is not a normal
cone, but only a proper cone,
the results of Section~\ref{sec-eig-nonexp} cannot be used.
In order to treat this case, we thus need to prove a result
similar to that of Section~\ref{sec-eig-nonexp} 
but under the less restrictive condition that $C$ is proper.
The results of Section~\ref{sec-main} cannot be applied.
We are using rather Condition~\ref{cond2} below,
which will allow degree theory arguments (see~\cite{Nuss-Mallet}).

\begin{theorem}\label{th-nouv1}
Let $C$ be a proper cone with nonempty interior in
a  Banach space $(X, \| \cdot \|)$. Let $\psi \in C^*\setminus \{0\}$,
and denote $\Sigma = \set{x \in \Cint}{\psi(x) = 1}$.
Let $G$ be a relatively 
 open subset of $\Sigma$, and $g : G \rightarrow \Sigma$
be a continuous map,
that is nonexpansive with respect to Hilbert's projective metric $d$. 
Assume that $v \in G$ is a fixed point of $g$. 
Make the following assumptions:
\begin{enumerate}
\renewcommand{\theenumi}{\rm (A\arabic{enumi})}
\renewcommand{\labelenumi}{\theenumi}
\item \label{cond1} $g$ is semidifferentiable at $v$;
\item \label{cond2} There exists a relatively open neighborhood 
$U$ of $v$ in $\Sigma$, $U \subset G$, 
and a homogeneous generalized measure of noncompactness $\nu$
on $X$, such that 
$g|_{U}$ is a $k$-set contraction with $k <1$ with respect to $\nu$;
\item \label{cond3} The fixed point of $g'_v:\psi^{-1}(0)\to\psi^{-1}(0)$
is unique: $x = g'_v(x),\; x \in \psi^{-1}(0)\Rightarrow x=0$.
\end{enumerate}
Then, $v$ is the only fixed point of $g$ in $G$:
$g(w)=w,\; w\in G\;\implies w=v$.
\end{theorem}
The following lemma is the key result needed to prove Theorem~\ref{th-nouv1}.
\begin{lemma}\label{lem-nouv4}
Let assumptions and notations be as in Theorem \ref{th-nouv1}. 
If $w \in \Sigma$, define, for $0 \leq t \leq 1$, $g_t:G\to\Sigma$ by
$g_t(y) = (1-t) g(y) + tw$. Given $\epsilon > 0$, there exists 
$\delta > 0$ such that for $0 \leq t \leq \delta$, there exists 
$y_t \in G$ with $\|y_t - v\| \leq \epsilon$ and $g_t(y_t) = y_t$.
\end{lemma}
\begin{proof}
Define $h_t(x) := g_t(v + x) - v$ for $x \in G_0:=-v+G \subset \psi^{-1}(0)$
and $t\in [0,1]$. 
By Assumption~\ref{cond2} of Theorem \ref{th-nouv1}, 
$g|_{U}$ is a $k$-set contraction with respect to $\nu$,
so is $(g_t)|_{U}$ for $0\leq t\leq 1$
(indeed from~\eqref{gmnc7}, any translation of a $k$-set contraction 
is a $k$-set contraction, and from~\eqref{gmnc5} any homothetic transformation
with factor $\leq 1$ of a $k$-set contraction is a $k$-set contraction).
Then, $h_t$ is a $k$-set contraction when restricted to the
open neighborhood $-v+ U$ of $0$ in $\psi^{-1}(0)$. So
$(h_t)'_0 = (g_t)'_v= (1 - t)g'_v$, the semidifferential of $h_t$ at $0$,
is a $k$-set contraction
on $(\psi^{-1}(0), \| \cdot \|)$ (this follows from~\eqref{e-ftofp1}). 

Since $g_0(v)=v$, we
need to show that for $t  >0$, $t$ small, $h_t(x) = x$ has 
a solution in $B_\epsilon(0) := 
\set{z \in \psi^{-1}(0)}{\|z\| \leq \epsilon}$
(one may assume that $\epsilon$ is such that $B_\epsilon(0) 
\subset G_0$).
For any fixed $t > 0$, we consider the homotopy:
\[ \id - f_{t,\lambda}, \quad 0 \leq \lambda \leq 1,\]
with
\[ f_{t,\lambda}(x) =(1 - \lambda) h_t (x) + \lambda(h_t)'_0(x), \quad x \in B_\epsilon(0),\; 0\leq t\leq 1,\; 0 \leq \lambda \leq 1 \enspace .
\]
Recall that, since $g$ is semidifferentiable at $v$, with semidifferential 
$g'_v$:
\[
g(v+x) = v + g'_v(x) + R(x)
\]
where $\|R(x)\| \leq \eta (\|x\|) \|x\| \text{ and } \lim_{s \rightarrow 0^+} \eta(s) = 0.$
It follows that
\begin{eqnarray*}
\lefteqn{x - f_{t,\lambda}(x) }\\
&=&  x - (1 - \lambda) (g'_v(x) + R(x) + tw - tg(v +x)) -
 \lambda (1-t)g'_v(x) \\
&=& x - g'_v(x) - (1 - \lambda)R(x) - 
(1 - \lambda) t(w - g(v + x)) + \lambda t g'_v(x) \enspace .
\end{eqnarray*}
For $\|x\| = 1$, $x \in \psi^{-1}(0)$, $x - g'_v(x) \neq 0$
by Assumption~\ref{cond3}.
Because $g'_v=h'_0$ is a $k$-set contraction on $(\psi^{-1}(0), \| \cdot\|)$, 
it follows that there exists $c > 0$ such that
\[
\|x - g'_v(x) \| \geq c \text{ for } \|x \| = 1, \ x \in \psi^{-1}(0)
\enspace.
\]
Indeed, by Lemma~\ref{lem-ksetcontract}, $\id -g'_v$ has
Property~\PF, hence if, by way of contradiction, there exists a sequence
$\seq{x_n }{n \geq 1}$ of $\psi^{-1}(0)$ such that $\|x_n\|=1$ and 
$\lim_{n\to\infty} x_n-g'_v(x_n)=0$, there exists a convergent subsequence.
Thus, the limit satisfies $x-g'_v(x)=0$ and $\|x\|=1$, which contradicts
Assumption~\ref{cond3}.

By the homogeneity of $g'_v$, it follows that
\[
\|x - g'_v(x)\| \geq c\|x\|,
\text{ for } x \in \psi^{-1}(0)\enspace .
\]
Select $\epsilon_1 \leq \epsilon$, $\epsilon_1 > 0$,
so that $\eta(s) \leq c/2$ for $0 \leq s \leq \epsilon_1$.
It follows that
\[
\|x - g'_v(x) - (1 - \lambda) R(x)\| \geq \frac{c}{2} \|x\|,
\text{ for } 0 \leq \|x\| \leq \epsilon_1, \; x \in \psi^{-1}(0)
\enspace .
\]
Since $g(v+x)$ and $g'_v(x)$ are bounded in norm on 
$B_{\epsilon_1}(0)$, it follows that there exists $\delta > 0$ so that
\[
t \| w - g(v+x)\| + t\|g'_v(x)\| < \frac{c\epsilon_1}{2}, 
\text{ for } \|x\| = \epsilon_1, \; 0 \leq t \leq \delta
\enspace .
\]
Using this estimate, we see that, for $0 \leq t \leq \delta$,
$\|x\| = \epsilon_1$ and $0 \leq \lambda \leq 1$,
\[
x - f_{t,\lambda}(x)\neq 0
\enspace .
\]
Since $h_t$ and $(h_t)'_0$ are both continuous and $k$-set contractions,
so is $f_{t,\lambda}$ for $0\leq \lambda\leq 1$
(indeed, using~\eqref{gmnc2} and~\eqref{gmnc5},
one shows easily that the convex combination of $k$-set contractions
is a $k$-set contraction).
Hence, by the homotopy property for degree theory for $k$-set contractions
(see~\cite{MR0312341,nussbaumdegree72}
for the case where $\nu=\alpha$ the 
Kuratowski-Darbo generalized measure of noncompactness,
and~\cite{MR791918,Nuss-Mallet,nussbaummp10,nussbaummp11} for any $\nu$),
we have
\[
\deg (\id - h_t, H, 0) = \deg(\id -(h_t)'_0,H, 0) \enspace, \]
for all $0\leq t\leq \delta$,
where $H = \set{ z \in \psi^{-1}(0)}{\|z\| < \epsilon_1}$.

Because $g$ is nonexpansive with respect to $d$,
by Lemma~\ref{preli-fpv},\ref{preli-fpv5},
we know that $g'_v$ is nonexpansive on $\psi^{-1}(0)$,
with respect to the seminorm $\omega_v$ defined by~\eqref{def-omega}
(which is a norm on $\psi^{-1}(0)$).
It follows that $x-\sigma (h_t)'_0(x)=x - \sigma (1-t) g'_v(x) \neq 0$
for $0<t\leq \delta$,
$0 \leq \sigma \leq 1$ and $\|x\| = \epsilon_1$, $x \in \psi^{-1}(0)$.

By the homotopy property of degree theory again,
\begin{align*}
\deg(\id - (h_t)'_0, H, 0) & = \deg (\id - \sigma (h_t)'_0, H, 0) \qquad
\mrm{ for } 0 \leq \sigma \leq 1\\
 & = \deg (\id, H, 0) = 1 \enspace .
\end{align*}
It follows that,  for $0 < t \leq \delta$,  $\deg (\id - h_t, H, 0) = 1$
 so there exists $x_t \in H$ with $x_t = h_t(x_t)$.
\end{proof}

\begin{proof}[Proof of Theorem~\ref{th-nouv1}] 
Suppose, by way of contradiction, 
that there exists $w \in G$ with $w \neq v$ and $g(w) = w$. 
Select $r > 0$ so that $\|w - v\| > r$.
Define $g_t:G\to\Sigma$ by $g_t(y) = (1-t) g(y) + tw$. Take 
$\delta > 0$ (by Lemma \ref{lem-nouv4}) so that,  for $0\leq  t \leq \delta$, 
there exists $y_t \in G$ such that $\|y_t - v\| \leq r$ and $g_t(y_t) = y_t$. 
Note that $g_t(w) = w$ too,
and we know (see~\cite[Lemma~2.1, p.~45]{nussbaum88}) that
\[
d(g_t(z_1), g_t(z_2)) < d(z_1, z_2)
\]
if $0 < t \leq 1$ and $z_1,z_2 \in G$, $z_1 \neq z_2$. 
Taking $z_1 = y_t$, $z_2 = w$, $0 < t \leq \delta$, we obtain a contradiction.
\end{proof}

\begin{remark}
In our definition of a $k$-set contraction $f:D\to X$, $k<1$,
with respect to a homogeneous generalized measure of noncompactness $\nu$,
we do not assume that $f$ is continuous in the norm topology of $X$.
However in Theorem~\ref{th-nouv1}, we have added the hypothesis that $g$ is
continuous in the norm topology.
The reason for this restriction is that we use the theory of 
the fixed point index for $k$-set contractions, $k<1$, and the 
standard development of that theory requires maps to be continuous in the
norm topology. 
If the map $g$ in Theorem~\ref{th-nouv1} was only 
assumed to be nonexpansive with respect to Hilbert's projective metric $d$,
then the continuity of $g$ in the norm topology of $X$ would have followed
easily for a normal cone $C$ (from Proposition~\ref{prop-normequiv})
but not for a general proper cone $C$.
In fact, the assumption that $g$ in Theorem~\ref{th-nouv1} is continuous 
in the norm topology is not necessary. We sketch the argument below, but 
for reasons of length, details are omitted. 

In the following we 
shall always assume that $C$, $(X,\|\cdot\|)$, $\psi$, and  $\Sigma$
are as in the first two sentences of Theorem~\ref{th-nouv1}
and as usual $d$ will denote Hilbert's projective metric on $\Sigma$.
If $\seq{y_n }{n \geq 1}$ is a sequence of points in $\Sigma$
and $\lim_{n\to\infty} \|y_n-y\|=0$, where $y\in \Sigma$, it is not hard 
to prove that $\lim_{n\to\infty} d(y_n,y)=0$.
Conversely, if $T=\set{y_n}{n\geq 1}\subset \Sigma$ has a 
compact closure in the norm topology and there exists $y\in \Sigma$ such that
$\lim_{n\to\infty} d(y_n,y)=0$, one deduces that $\lim_{n\to\infty} \|y_n-y\|=0$.
Working in the norm topology, let $U$ be a relatively open subset of $\Sigma$
such that $\clo{U}\subset \Sigma$. Let $g: \clo{U}\to \Sigma$ be a map
which is continuous with respect to Hilbert's projective metric $d$
on $\Sigma$, and is also a $k$-set contraction, $k<1$, with respect
to a homogeneous generalized measure of noncompactness $\nu$.
Under these assumptions, we claim that $g$ is continuous in the norm
topology, so the assumption of norm continuity in Theorem~\ref{th-nouv1}
is superfluous.
Indeed, suppose $\seq{x_n }{n \geq 1}$ is a sequence in $\clo{U}$ 
and that $\lim_{n\to\infty} \|x_n-x\|=0$. By the remarks above,
$\lim_{n\to\infty} d(x_n,x)=0$. Since $g$ is continuous with respect to 
Hilbert's projective metric $d$, $\lim_{n\to\infty} d(g(x_n),g(x))=0$. 
In particular, there exists $n_0\geq 1$ such that, for all
$n\geq n_0$, $g(x_n)\in\set{y\in \Sigma}{\frac{1}{2} g(x)\leq y\leq 2 g(x)}$,
which is a norm closed subset of $\Sigma$. Let
$S=\set{x_n}{n\geq 1}\cup\{x\}$, then $S$ is compact in the norm topology,
thus $\nu(S)=0$, so $\nu(g(S))=0$. Consider $T:=\clo{(g(S))}$,
we get that $T$ is compact in the norm topology and $T\subset \Sigma$.
Assume by contradiction that $\lim_{n\to\infty} \|g(x_n)-g(x)\|\neq 0$.
Then, there exists $\epsilon>0$ and a 
subsequence $n_i\to\infty$ such that $\|g(x_{n_i})-g(x)\|\geq \epsilon$ for
all $i\geq 1$. By the norm compactness of $T$, we can also assume
that there exists $y\in \Sigma$ with $\|y-g(x)\|\geq \epsilon$ such that 
$\lim_{n\to\infty} \|g(x_{n_i})-y\|=0$. But our earlier remarks now 
imply that $\lim_{i\to\infty} d(g(x_{n_i}),y)=0$, a contradiction.
Hence  $\lim_{n\to\infty} \|g(x_n)-g(x)\|= 0$, which shows that 
$g$ is continuous in the norm topology.
\end{remark}

\begin{theorem}\label{th-nouv5}
Let $C$ be a proper cone with nonempty interior in
a  Banach space $(X, \| \cdot \|)$. Let $\psi \in C^*\setminus \{0\}$,
and denote $\Sigma = \set{x \in \Cint}{\psi(x) = 1}$.
Let $G$ be a relatively 
 open subset of $\Sigma$, and $g : G \rightarrow \Sigma$
be a continuous map,
that is nonexpansive with respect to Hilbert's projective metric $d$. 
Assume that $v \in G$ is a fixed point of $g$. 
Make the following assumptions:
\begin{enumerate}
\renewcommand{\theenumi}{\rm (A\arabic{enumi})}
\renewcommand{\labelenumi}{\theenumi}
\item \label{cond1'} $g$ is semidifferentiable at $v$;
\item \label{cond2'} There exists a relatively open neighborhood 
$U$ of $v$ in $\Sigma$, $U \subset G$, and an integer $n \geq 1$ 
such that $g^n|_U$ is a $k$-set contraction with $k <1$. 
If $n > 1$, assume also that $g'_v$ is uniformly continuous on 
bounded sets of $\psi^{-1}(0)$.
\item \label{cond3'} The fixed point of $(g'_v)^n:\psi^{-1}(0)\to\psi^{-1}(0)$
is unique: $x = (g'_v)^n(x),\; x \in \psi^{-1}(0)\Rightarrow x=0$.
\end{enumerate}
Then, $v$ is the only fixed point of $g^n$ in $G$ 
(and hence the only fixed point of $g$ in $G$):
$g^n(w)=w,\; w\in G\;\implies w=v$.
\end{theorem}
\begin{proof}
Let $n\geq 1$ be as in the theorem.
When $n>1$, 
the map $g^n$ is defined on some relatively open subset $G_n$
of $\Sigma$ (take $G_n=G\cap g^{-1}(G)\cap\cdots \cap
g^{-(n-1)}(G)$), and since $g$ is semidifferentiable at $v$, $g(v)=v$ and
$g'_v$ is uniformly continuous on  bounded sets of $\psi^{-1}(0)$,
applying the chain rule (Lemma~\ref{lemma-chain}) to $C_1=C_2=\psi^{-1}(0)$,
we get that $g^n$ is semidifferentiable at $v$, with
$(g^n)'_v = (g'_v)^n$.
Applying Theorem~\ref{th-nouv1} to $g^n$, we get the conclusion
of Theorem~\ref{th-nouv5}.
\end{proof}
\section{Eigenvectors of differentiable nonexpansive maps on cones}
\label{sec-eig-nonexp-diff}
In this section, we specialize the previous results to the case
where $f$ is differentiable at the fixed point, and compare
the results obtained in this way with the ones of~\cite{nussbaum88}.
\begin{corollary}[Uniqueness of eigenvectors in $\Sigma$, differentiable case]\label{coro1}
Let $C$ be a normal cone with nonempty interior  in 
a Banach space $(X,\|\cdot\|)$, let $\psi\in C^*\setminus\{0\}$, and denote
$\Sigma=\set{x\in \Cint}{\psi(x)=1}$.
Let $G$ be an open subset of $C$ and
$f : G  \rightarrow \Cint$ be a map such that
$f|_{G\cap \Sigma}$ is nonexpansive with respect to Hilbert's projective 
metric $d$.
Assume that $v\in  G\cap \Sigma$ is a fixed point of $f$: $f(v)=v$.
Make the following assumptions:
\begin{enumerate}
\renewcommand{\theenumi}{(A\arabic{enumi})}
\renewcommand{\labelenumi}{\theenumi}
\item\label{h0c4} $f$ is differentiable at $v$;
\item\label{h0c5} 
The linear  operator $(\id-f'_v)|_{\psi^{-1}(0)}: \psi^{-1}(0)\to X$ 
is semi-Fredholm with index in 
$\Z\cup\{-\infty\}$;
\item\label{h0c6} The fixed point of $f'_v$ in $\psi^{-1}(0)$ is unique:
$f'_v(x) = x,\, x \in \psi^{-1}(0) \Rightarrow x =0$;
\item\label{h0c7} $f'_v(C)\subset C$;
\item\label{h0c8} $f'_v(v)\leq v$.
\end{enumerate}
Then, the eigenvector of $f$ in $G\cap \Sigma$ is unique: 
$\exists \lambda>0,\; f(w) = \lambda w,\, w \in G\cap \Sigma\Rightarrow w =v$.
\end{corollary}
\begin{proof}
We only need to verify that $f$ satisfies Assumptions~\ref{h04}--\ref{h08}
of Theorem~\ref{theo01}. 
Clearly, Assumption~\ref{h0c4} of Corollary~\ref{coro1} implies
Assumption~\ref{h04} of Theorem~\ref{theo01}.
Under this condition,  $f'_v$ is linear, hence, Assumption~\ref{h05}
of Theorem~\ref{theo01} is equivalent to
Assumption~\ref{h0c5}
of Corollary~\ref{coro1} (see~\cite[Proposition~19.1.3]{hormander}
or~\cite[Chapter IV, Theorems~5.10 and~5.11]{kato}).
Assumption~\ref{h06} of Theorem~\ref{theo01} is
identical to Assumption~\ref{h0c6} of Corollary~\ref{coro1}.
Since $f'_v$ is linear, $f'_v$ is order preserving if,
and only if, $f'_v(x)\geq 0$ when $x\geq 0$, that is $f'_v(C)\subset C$,
which shows that Assumption~\ref{h07} of Theorem~\ref{theo01}
is equivalent to Assumption~\ref{h0c7} of Corollary~\ref{coro1}
when $f'_v$ is linear.
Also, since $f'_v$ is linear, 
Assumption~\ref{h0c8} of Corollary~\ref{coro1}
implies easily Assumption~\ref{h08} of Theorem~\ref{theo01}
(one can also use Lemma~\ref{preli-fpv33}).
\end{proof}

\begin{remark}\label{rk-roger1}
The assumptions of Corollary~\ref{coro1} arise when
specializing the ones of Theorem~\ref{theo01} to the case in
which $f'_v$ is linear. However Assumption~\ref{h0c5} can be 
replaced by the apparently more restrictive assumption
that $\id-f'_v$ is Fredholm of index $0$, or 
that $(\id-f'_v)|_{\psi^{-1}(0)}$ is Fredholm of index $-1$, 
without loss of generality.

Indeed, under the other assumptions, in particular when
\ref{h0c4},~\ref{h0c7} and \ref{h0c8} hold, and when $v$ is in the
interior of $C$, we get that $f'_v$ is linear continuous,
order preserving and satisfies $ \supeigen{C} (f'_v)\leq 1$,
where $ \supeigen{C}(f'_v)$ is defined  in~\eqref{def-collatz-number}
($\supeigen  C (f'_v) :=\inf\, \set{ \lambda >0}{\exists x \in \Cint, \; f'_v(x) \leq \lambda x }$).
From Lemma~\ref{collatz-lem72}, it follows that 
$\bonsall C (f'_v)\leq 1 $
and since $f'_v$ is linear, we get that $r(f'_v)\leq 1$.
Then for all $0<t<1$, $\id-tf'_v$ is one-one, and thus Fredholm
of index $0$. 

Moreover, since  $\ker \psi$ is closed and of codimension one,
it follows from standard properties of Fredholm operators
(see~\cite[\S 19.1]{hormander}) that $\id-f'_v: X\to X$  is semi-Fredholm
 if and only if 
 $(\id-f'_v)|_{\psi^{-1}(0)}: \psi^{-1}(0)\to X$ 
is semi-Fredholm, and the index of
$\id-f'_v$ is equal to the one of $(\id-f'_v)|_{\psi^{-1}(0)}$ plus $1$.
This implies that, under Assumption~\ref{h0c5},  $\id-f'_v$ is semi-Fredholm.

From the latter property,
the fact that all operators $\id-tf'_v$ with $t<1$ are Fredholm
of index $0$, and the continuity of the Fredholm index on the set of semi-Fredholm operators,
we deduce that $\id-f'_v$ is Fredholm of index $0$,
or equivalently that $(\id-f'_v)|_{\psi^{-1}(0)}$ is Fredholm of index $-1$.
Conversely, if  $(\id-f'_v)|_{\psi^{-1}(0)}$ is Fredholm of index $-1$, 
then~\ref{h0c5} holds trivially.
\end{remark}

We shall give now a corollary of Corollary~\ref{coro1}, which extends 
partially~\cite[Theorem~2.5]{nussbaum88}.
For a bounded linear operator $L:X\to X$ on a Banach space $(X,\|\cdot\|)$, we 
denote by $r(L)$ the spectral radius of $L$, which coincides
with the Bonsall spectral radius $\bonsall{X} (L)$
in~\eqref{eq31-bonsallcspr}. We
denote by $N(L)$ the null space of $L$.
Let $C$ be a proper cone of $X$. We say that $L$ satisfies
 the {\em weak Krein-Rutman} (\WKR) condition with respect to $C$ if 
$L(C)\subset C$, and either $r(L)=0$, or $r:=r(L)>0$ and
(i) there exists $u\in C\setminus\{0\}$ such that 
$N(r \, \id -L)=\set{\lambda u}{\lambda\in \R}$,
and (ii) $r \, \id -L$ is a semi-Fredholm operator.
Condition (\WKR) is easier to check
than the Krein-Rutman (KR) condition used in 
\cite[Definition~2.1]{nussbaum88}, where the condition (ii)
is replaced by the condition that $r \, \id -L$ is a Fredholm operator
with index $0$, and where it is also assumed
that, 
when $r>0$,
$L^*$ has an eigenvector $u^*\in C^*\setminus\{0\}$
with eigenvalue $r$ such that $u^*(u)>0$. 
However, if (\WKR) holds, then, by the same arguments as in Remark~\ref{rk-roger1}, if $r:=r(L)>0$, then, for all $0<t<1$, $r\id - tL$,
and so $r\id -L$, are necessarily Fredholm of index $0$, as requested
by the (KR) condition. 
Furthermore, a refinement of a theorem of Krein and Rutman (see~\cite{schaefer2}) implies that $L^*$ has an eigenvector $u^*$ in
$C^*$ with eigenvalue $r$.
Thus, the only difference in generality between condition (WKR) and condition (KR) is the requirement that $u^*(u)>0$. Even
if a variety of conditions imply that $u^*(u)>0$, it may happen that
$u^*(u)=0$.

\begin{corollary}[Uniqueness of eigenvectors, differentiable case]\label{coro2}
Let $C$ be a normal cone with nonempty interior  in 
a Banach space $(X,\|\cdot\|)$, let $\psi\in C^*\setminus\{0\}$, and denote
$\Sigma=\set{x\in \Cint}{\psi(x)=1}$.
Let $G$ be an open subset of $(X,\|\cdot\|)$ included in $C$ and
$f : G  \rightarrow \Cint$ be a map such that
$f|_{G\cap \Sigma}$ is nonexpansive with respect to Hilbert's projective 
metric $d$.
Assume that $v\in  G\cap \Sigma$ is a fixed point of $f$: $f(v)=v$.
Make the following assumptions:
\begin{enumerate}
\renewcommand{\theenumi}{\rm (A\arabic{enumi})}
\renewcommand{\labelenumi}{\theenumi}
\item\label{h0n4} $f$ is differentiable at $v$;
\item\label{h0n5} 
The linear operator $f'_v$ satisfies the \WKR\ condition
with respect to $C$;
\item\label{h0n6} If $u\in C\setminus\{0\}$ is an eigenvector of
$f'_v$ with eigenvalue $r(f'_v)>0$, then $\psi(u)>0$;
\item\label{h0n8} There exist $\delta>0$  such that  $\delta\leq 1$ and 
$t f(v)\leq f(tv)$ for all $1-\delta\leq  t\leq 1$.
\end{enumerate}
Then, the eigenvector of $f$ in $G\cap \Sigma$ is unique: 
$\exists \lambda>0,\; f(w) = \lambda w,\, w \in G\cap \Sigma\Rightarrow w =v$.
\end{corollary}
\begin{proof}
\newcommand{\ofcoro}{of Corollary~\ref{coro1}}
\newcommand{\ofcoron}{of Corollary~\ref{coro2}}
We only need to verify that $f$ satisfies 
Assumptions~\ref{h0c4}--\ref{h0c8}\ \ofcoro.
Assumption \ref{h0n4}\ \ofcoron\ corresponds to \ref{h0c4}\ \ofcoro.
Assumption~\ref{h0n5}\ \ofcoron\ implies that $f'_v(C)\subset C$,
by definition of the \WKR\ condition, that is
Assumption~\ref{h0c7}\ \ofcoro.
By Lemma~\ref{preli-fpv},~\ref{preli-fpv32}, 
we get that Assumption~\ref{h0n8}\ \ofcoron\ implies
Assumption~\ref{h0c8}\ \ofcoro, since $f'_v$ is linear.
It remains to show Asssumptions~\ref{h0c5} and \ref{h0c6}\ \ofcoro.
{From} Assumptions~\ref{h0c7} and \ref{h0c8}\ \ofcoro, 
$f'_v$ is homogeneous, order preserving and 
satisfies $ \supeigen{C} (f'_v)\leq 1$,
with $ \supeigen{C}(f'_v)$ as in~\eqref{def-collatz-number}.
From Lemma~\ref{collatz-lem72}, it follows that 
$\bonsall C (f'_v)\leq 1 $
and since $f'_v$ is linear, we get that $r(f'_v)\leq 1$.
Consider first the case where $r(f'_v)<1$. Then,
by Proposition~\ref{prop1}, $\id -f'_v$ has Property~\PF,
or is a semi-Fredholm operator
with index in $\Z\cup\{-\infty\}$, which shows
Assumption~\ref{h0c5}\ \ofcoro, and
$N(\id-f'_v)=\{0\}$, which implies Assumption~\ref{h0c6}\ \ofcoro.
Consider now the case where $r(f'_v)=1$. 
Then, by Assumption~\ref{h0n5}\ \ofcoron,
$\id -f'_v$ is a semi-Fredholm operator,
and since the dimension of $N(\id -f'_v)$ is finite (equal to $1$),
the index of  $\id -f'_v$ is in $\Z\cup\{-\infty\}$, hence
Assumption~\ref{h0c5} \ofcoro\ holds.
Let $x$ be a fixed point of $f'_v$ in $\psi^{-1}(0)$, that is
$x\in N(\id-f'_v)\cap \psi^{-1}(0)$.
By Assumption~\ref{h0n6}\ \ofcoron, we deduce that $x=0$.
This shows that Assumption~\ref{h0c6}\ \ofcoro\ holds.
\end{proof}

Corollary~\ref{coro2} extends partially~\cite[Theorem~2.5]{nussbaum88}.
We obtain the same conclusion with different assumptions: the
condition~\WKR\ is replaced by the stronger KR condition;
the map $f$ is assumed to be $\continuous_1$, whereas we only assume
$f$ to be differentiable at $v$; Condition~\ref{h0n8} is assumed
for all $v\in \Sigma\cap G$, whereas we only require it for the
fixed point $v$;
but we require the cone $C$ to be normal, whereas the result of
\cite{nussbaum88} is valid for a general proper cone $C$.

\section{Application to nonexpansive self-maps of AM-spaces with unit}\label{sec-order}
In this section, we give additive versions of the results of 
Section~\ref{sec-eig-nonexp}, motivated by the case of Shapley operators
of zero-sum games with a compact state space $K$.
The latter operators are order preserving and sup-norm nonexpansive maps
acting on a Banach space $\continuous(K)$. 

Since for any Banach space $(X,\|\cdot\|)$,
the corresponding metric space $(V,d)$ (with $V=X$ and $d(x,y)=\|x-y\|$)
satisfies Assumptions~\ref{h-1}, \ref{h-2} and \ref{h-3}, the
following additive versions of Theorems~\ref{theo00d} and~\ref{theo62} are 
obtained directly from Theorems~\ref{theo0} and~\ref{th-geom}
together with Remark~\ref{rem-r1}.

\begin{corollary}\label{cor-801}
Let $(X,\|\cdot\|)$ be an AM-space with unit, let $G$ be an open subset
of $X$ and let $F:G\to X$ be a nonexpansive map with respect to $\|\cdot\|$. 
Let $v\in G$ be a fixed point of $F$: $F(v)=v$. 
Make the following assumptions:
\begin{enumerate}
\renewcommand{\theenumi}{\rm (A\arabic{enumi})}
\renewcommand{\labelenumi}{\theenumi}
\item\label{h801-4} $F$ is semidifferentiable at $v$;
\item\label{h801-5} The map $\id-F'_v: X\to X$  has Property \PF;
\item\label{h801-6} The fixed point of $F'_v:X\to X$ is unique:
$F'_v(x) = x,\, x \in X\Rightarrow x =0$.
\end{enumerate}
Then, the fixed point of $F$ in $G$ is unique: 
$F(w) = w,\, w \in G\Rightarrow w =v$.
\end{corollary}

\begin{corollary}\label{cor-802}
Let $(X,\|\cdot\|)$ be an AM-space with unit, let $G$ be an open subset of $X$
 and let $F:G\to G$ be a nonexpansive map with respect to $\|\cdot\|$. 
If $F$ has a fixed point $v\in G$ and if  $F$ is semidifferentiable at $v$,
with $\bonsall{}(F'_v)<1$, the fixed point of $F$ in $G$ is unique.
Moreover, if $G$ is connected, we have:
\begin{align}
\label{e-geom-802}
\forall x\in G,\;\; \limsup_{k\to \infty} \|F^k(x)-v\|^{1/k}\leq
\bonsall{}(F'_v)
\enspace .
\end{align}
\end{corollary}

Recall that when $(X,\|\cdot\|)$ is an AM-space with unit, denoted by $e$, 
we denote by $\omega$ the seminorm $\omega_e$
defined in~\eqref{def-omega}.
Applying Theorem~\ref{theo0} to the metric spaces $(V,d)$ corresponding to the
Banach spaces $(\psi^{-1}(0),\|\cdot\|)$ and $(\psi^{-1}(0),\omega )$
with $\psi\in (X^+)^*\setminus\{0\}$, or replacing $X$ by $\psi^{-1}(0)$
and perhaps $\|\cdot\|$ by $\omega$ in Corollary~\ref{cor-801}, 
one would have obtained an additive version of Theorem~\ref{theo00}.
The following additive version of Theorem~\ref{theo01}, will be
derived from Corollary~\ref{cor-801}.

\begin{corollary}\label{cor-804}
Let $(X,\|\cdot\|)$ be an AM-space with unit, denoted by $e$, and
let $\psi\in (X^+)^*\setminus\{0\}$, and denote by
$\omega$ the seminorm $\omega_e$ defined in~\eqref{def-omega}.
Let $G$ be an open subset of $X$ and let $F:G\to X$
be a map such that $F|_{G\cap \psi^{-1}(0)}$ is
nonexpansive with respect to $\omega$. 
Let $v\in G\cap\psi^{-1}(0)$ be a fixed point of $F$: $F(v)=v$. 
Make the following assumptions:
\begin{enumerate}
\renewcommand{\theenumi}{\rm (A\arabic{enumi})}
\renewcommand{\labelenumi}{\theenumi}
\item\label{h804-4} $F$ is semidifferentiable at $v$;
\item\label{h804-5} 
The map $(\id-F'_v)|_{\psi^{-1}(0)}: \psi^{-1}(0)\to X$ has Property \PF;
\item\label{h804-6} The fixed point of $F'_v$ in $\psi^{-1}(0)$ is unique:
$F'_v(x) = x,\, x \in \psi^{-1}(0) \Rightarrow x =0$;
\item\label{h804-7} $F'_v$ is order preserving;
\item\label{h804-8} $F'_v$ is additively subhomogeneous.
\end{enumerate}
Then, the additive eigenvector of $F$ in $G\cap \psi^{-1}(0)$ is unique: 
$\exists \lambda\in \R,\; F(w) = \lambda e+ w,\, w \in G\cap \psi^{-1}(0)
\Rightarrow w =v$.
\end{corollary}
\begin{proof}
Since the assumptions and conclusions of the corollary depend only on 
$\psi^{-1}(0)$, we may assume that $\psi(e)=1$.
Let $\tilde{F}:G\to \psi^{-1}(0),\, x\mapsto F(x)-\psi(F(x)) e$ and
denote $H=\tilde{F}|_{G\cap \psi^{-1}(0)}$. We shall prove that $H$
satisfies the assumptions of Corollary~\ref{cor-801} for the 
Banach space $(\psi^{-1}(0),\omega)$.
Since $F|_{G\cap \psi^{-1}(0)}$ is
nonexpansive with respect to $\omega$, and
$\omega(x+a e)=\omega(x)$ for all $x\in G$, $H$ is
nonexpansive with respect to $\omega$. 
Since $F(v)=v$ and $v\in G\cap \psi^{-1}(0)$, $H(v)=v$.
Since $F$ is semidifferentiable at $v$ and the map 
$R:X\to \psi^{-1}(0),\, x\mapsto x-\psi(x) e$ is linear, thus
differentiable at any point, $\tilde{F}=R\circ F$ is semidifferentiable
at $v$, by Lemma~\ref{lemma-chain}, and $\tilde{F}'_v=R\circ F'_v$.
This implies that $H$ is semidifferentiable at $v$ with
$H'_v=\tilde{F}'_v|_{\psi^{-1}(0)}:\psi^{-1}(0)\to \psi^{-1}(0)$.
We thus get Assumption~\ref{h801-4} of Corollary~\ref{cor-801}.
Taking $C=X^+$, and using Assumptions~\ref{h804-7} and~\ref{h804-8}
of the Corollary~\ref{cor-804}, we obtain by the same arguments as in the
proof of Theorem~\ref{theo01}, the following inequality:
\begin{equation}\label{eq-cor804}
\| x-F'_v(x)\|\leq 2 \, \| x-H'_v(x)\| \quad \forall x\in \psi^{-1}(0)
\enspace .
\end{equation}
Then, using $H'_v(x)=F'_v(x)-\psi(F'_v(x)) e$ and~\eqref{eq-cor804},
we obtain, by the same arguments as in the
proof of Theorem~\ref{theo01}, Assumptions~\ref{h801-5} and~\ref{h801-6}
of Corollary~\ref{cor-801} with $X$ replaced by $\psi^{-1}(0)$.
If $w\in G\cap \psi^{-1}(0)$ satisfies $F(w) = \lambda e +w$ for some 
$\lambda\in\R$, we get $H(w)=w$, and by Corollary~\ref{cor-801}, we 
obtain $w=v$.
\end{proof}

\begin{remark}
As for Theorem~\ref{theo01}, Corollary~\ref{cor-804} can be applied in 
the following situations.
{From} Lemma~\ref{preli-fpv-add}, Assumption~\ref{h804-7} 
of Corollary~\ref{cor-804} is fulfilled
as soon as $F$ is order preserving in a neighborhood of $v$,
and Assumption~\ref{h804-8} is fulfilled as soon as $F$ is 
additively subhomogeneous in a neighborhood of $v$.
Moreover, these properties imply, by Lemma~\ref{lemma-nonexpan-add}, 
that $F|_{G\cap \psi^{-1}(0)}$ is nonexpansive with respect to $\omega$.
\end{remark}

The following additive version of Corollary~\ref{theo1} will be
derived from Corollary~\ref{cor-804}.

\begin{corollary}\label{cor-add}
Let $(X,\|\cdot\|)$ be an AM-space with unit, denoted by $e$,
and let $F:X\to X$ be
an additively homogeneous and order preserving map. 
Let $S=\set{x\in X}{F(x)=x}$, and assume that $v\in S$.
Make the following assumptions:
\begin{enumerate}
\renewcommand{\theenumi}{\rm (A\arabic{enumi})}
\renewcommand{\labelenumi}{\theenumi}
\item\label{h4p} $F$ is semidifferentiable at $v$;
\item\label{h5p} The map $\id-F'_v: X \to X$ has Property \PF;
\item\label{h6p} if $F'_v(x) = x$ for some $x \in X$, then 
$x \in \{ \lambda e \mid \lambda \in \R\}$;
\end{enumerate}
Then, $S=\set{v+\lambda e}{\lambda \in \R}$.
\end{corollary}
\begin{proof}
Let us check that $F$ satisfies the assumptions of 
Corollary~\ref{cor-804}.
Consider $G=X$, $\psi\in (X^+)^*\setminus\{0\}$ such that $\psi(e)=1$.
Since $F$ is additively homogeneous and $v\in S$, $v+\lambda e\in S$
for all $\lambda\in\R$. Moreover, since $F$ is semidifferentiable at $v$,
$F$ is semidifferentiable at $v+\lambda e$ for all $\lambda \in \R$,
with $F'_{v+\lambda e}= F'_v$. Taking $\lambda=-\psi(v)$, we
get $\psi(v+\lambda e)= 0$ and $v+\lambda e$ satisfies all the
assumptions of the corollary. Hence, replacing $v$ by $v+\lambda e$, we may
assume that $\psi(v)=0$.

We have : $v\in G\cap \psi^{-1}(0)$ and $F(v)=v$.
Since $F$ is order preserving and additively homogeneous, 
$F|_{ \psi^{-1}(0)}$ is
nonexpansive with respect to $\omega$ (see Lemma~\ref{lemma-nonexpan-add}).
Assumptions~\ref{h804-4} and~\ref{h804-5}  of Corollary~\ref{cor-804} are
implied by Assumptions~\ref{h4p} and~\ref{h5p}  
 of Corollary~\ref{cor-add}.
Assumptions~\ref{h804-7} and~\ref{h804-8} of Corollary~\ref{cor-804}
are deduced from
Lemma~\ref{preli-fpv-add}, \ref{preli-fpv-add1}
and~\ref{preli-fpv-add2}, using the fact that $F$
 is order preserving and additively homogeneous.
This completes the proof of the assumptions of Corollary~\ref{cor-804}.

Let $x\in S$, and denote $\lambda=\psi(x)$.
Since $F$ is additively homogeneous,  $y=x-\lambda e$ satisfies 
$F(y)=y$ and $y\in \psi^{-1}(0)$.
{From} Corollary~\ref{cor-804}, this implies $y=v$, hence $x=\lambda e+ v$.
Since we already proved above the converse implication, this 
yields the conclusion of the corollary.
\end{proof}

The following is the additive version of Theorem~\ref{th-6.8}. 
\begin{corollary}\label{cor-addnew}
Let $(X,\|\cdot\|)$ be an AM-space with unit, denoted by $e$,
and denote by
$\omega$ the seminorm $\omega_e$ defined in~\eqref{def-omega}.
Let $F:X\to X$ be an additively homogeneous and order preserving map,
denote by $S=\set{x\in X}{F(x)=x}$, and assume
that $v\in S$. Make the following assumptions:
\begin{enumerate}
\renewcommand{\theenumi}{\rm (A\arabic{enumi})}
\renewcommand{\labelenumi}{\theenumi}
\item\label{h4pnew} $F$ is semidifferentiable at $v$;
\item\label{h6pnew} $\bonsall{}(F'_v) <1$, where $\bonsall{}$ is defined
with respect to the seminorm $\omega$, as in~\eqref{6*}, \eqref{6**}
(with $C=X$).
\end{enumerate}
Then, for all $x\in X$, 
\begin{align}
\limsup_{k\to \infty} \omega(F^k(x)-v)^{1/k} \leq \bonsall{}(F'_v)
\enspace ,
\label{e-geom-hilbertadditive}
\end{align}
and there is a scalar $\lambda$ (depending on $x$),
such that 
\begin{align}
\limsup_{k\to \infty} \|F^k(x)-\lambda e -v \|^{1/k} \leq \bonsall{}(F'_v)
\enspace .
\label{e-geom-additive}
\end{align}
\end{corollary}
\begin{proof}
We define $H$ and $\psi$
as in the proof of Corollary~\ref{cor-804},
so that $H$ leaves invariant the Banach space $\psi^{-1}(0)$
equipped with the norm $\omega$. 
Moreover, since $F$ is additively homogeneous,
we get that $H^k(x)=F^k(x)-\psi(F^k(x)) e$ for all $x\in X$ and $k\geq 1$.
Then, Corollary~\ref{cor-802} implies
that 
\[
\limsup_{k\to \infty} \omega(F^k(x)-v )^{1/k} 
= \limsup_{k\to \infty} \omega(H^k(x)-v)^{1/k}
\leq \bonsall{} (H'_v) = 
\bonsall{}(F'_v) \enspace 
\]
(recall that $\omega(x+a e)=\omega(x)$ for all $x\in X$ and $a\in\R$),
which shows~\eqref{e-geom-hilbertadditive}.
Now, we use the additive analogue
of the $1$-cocycle formula~\eqref{e-cocycle},
namely
\[
F^k(x)= \big(\psi(F\comp H^{k-1}(x))
+\dots+  \psi (F\comp H^0(x)) \big)e + H^k (x) 
\]
and by a straightforward adaptation of the argument
of the second part of the proof of Theorem~\ref{th-6.8},
we conclude that~\eqref{e-geom-additive} holds
for some scalar $\lambda\in \R$.
\end{proof}

\begin{remark}
One may have derived Corollary~\ref{cor-add} from 
Corollary~\ref{theo1}
by using the Kakutani-Krein theorem
and exp-log transformations as follows (see Section~\ref{am-spaces-def}).
Let $\imath: X\to \continuous(K)$ (where $K$ is  a compact set),
$C=\continuous^+(K)$, $\Cint$ and $\log: \Cint \to \continuous(K)$ be defined as in
Section~\ref{am-spaces-def}. Denote $\jmath=\imath^{-1} \comp \log
: \Cint\to X$. Then, $\jmath^{-1}=\exp\comp \imath$, where
$\exp=\log^{-1}$.
To a  map $F:X\to X$, one associates the map $f:\Cint \to \Cint$,
$f= \jmath^{-1} \comp  F \comp \jmath$.
Any additive property (homogeneity, subhomogeneity, order preserving
property) of $F$ is transformed into its multiplicative version for $f$.
Moreover, since $\imath$ is an isometry, $F$ is Lipschitz continuous
for the norm of $X$ if, and only if, $f$ is  Lipschitz continuous for the
Thompson metric of $C$ and this implies that $f$ is locally
Lipschitz continuous for the sup-norm of $\continuous(K)$.
Hence, the differentiability of the $\exp$, $\log$, $\imath$ 
and $\imath^{-1}$ transformations, implies, from Lemma~\ref{lemma-chain},
that, when $F$ is Lipschitz continuous,
the semidifferentiability of $f$ is equivalent to that of $F$.
This allows us to derive Corollary~\ref{cor-add} from Corollary~\ref{theo1},
since in this case $F$ is nonexpansive (by Lemma~\ref{lemma-nonexpan-add}),
hence Lipschitz continuous with respect to $\|\cdot\|$.
However, in order to derive Corollary~\ref{cor-804} from Theorem~\ref{theo01},
and Corollary~\ref{cor-addnew} from Theorem~\ref{th-6.8} 
one should have generalized first Theorem~\ref{theo01}
and Theorem~\ref{th-6.8} to the case where 
$\psi$ is a nonlinear order preserving homogeneous
map from $\Cint$ to $\R^+$ (the linearity of $\psi$ is not
preserved by taking ``log-glasses'').
\end{remark}

\section{An example of stochastic game}\label{subsec-games}
As a simple illustration of the present results, consider
the following zero-sum two player game, which may be
thought of as a variant (with additive rewards)
of the Richman games~\cite{loeb} or of the stochastic tug-of-war games~\cite{peres} arising in the discretization of the infinity Laplacian~\cite{oberman}.

Let $G=(V,E)$ denote a (finite) directed graph
with set of nodes $V$ and set of arcs $E\subset V\times V$. 
Loops, i.e., arcs of the form $(i,i)$ are allowed. We
assume that every node has at least one successor
(for every $i\in V$, there is at least one $j\in V$ such that $(i,j)\in E$).
We associate to every arc $(i,j)\in E$ a payment $A_{ij}\in \R$. 
Two players, called ``Max'' and ``Min'', will move a token on this digraph,
tossing an unbiased coin at each turn, to decide who plays the turn. 
The player (Max or Min) who just won the right to play
the turn must choose a successor node $j$ (so that $(i,j)\in E$)
and move the token to this node,
Then, Player Max receives the payment $A_{ij}$ from Player Min.
We denote by $v_i(k)$ the {\em value} of this game in $k$ turns,
provided the initial state is $i\in V$.
We refer the reader to~\cite{sudderth,FilarVrieze,sorinneymanbook} for
background on zero-sum games, including the definition and
properties of the value.
In particular, standard dynamic programming arguments 
show that the value of the game in $k$ turns does exist, 
and that the {\em value vector} $v(k):=(v_i(k))\in \R^V$
satisfies
\[
v(k)=F(v(k-1)),\qquad v(0)=0. 
\]
where $F$ (the Shapley operator) is the map $\R^V\to \R^V$ given by 
\begin{align}
F_i(x)= \frac 1 2\big( \max_{j\in V, \;(i,j)\in E} (A_{ij}+x_j) + 
\min_{j\in V, \;(i,j)\in E} (A_{ij}+x_j)\big),
\qquad \forall i\in V \enspace .
\label{e-attainminmax}
\end{align}
The map $F$ is additively homogeneous and order preserving.
We are interested in the {\em mean payoff} vector
\[
\chi(F) := \lim_{k\to \infty}v(k)/k = \lim_{k\to\infty} F^k(0)/k;
\]
hence, $\chi_i(F)$ represents the mean payoff per time unit
starting from the initial state $i$, when the number of turns $k$ tends
to infinity. We shall consider, for simplicity, the
case in which $F$ has an additive eigenvector $u\in \R^n$
with associated eigenvalue $\mu\in\R$, meaning that
$F(u)=u+\mu e$ where $e$ denotes the unit vector of $\R^n$.
Then,
\[
\chi(F)= \mu e \enspace .
\]
Actually, the generalized Perron-Frobenius theorem in~\cite{arxiv1} implies that the additive
eigenpair $(u,\mu)$ does exist if the graph of the game $G$ is strongly connected. 

The vector $u$, which is sometimes called {\em bias}, or {\em potential}
in the dynamic programming literature, can be interpreted as an
invariant terminal payoff. Indeed, consider the auxiliary
game in which all the rewards $A_{ij}$ are replaced by $A_{ij}-\mu$,
and a terminal payment $u_i$ is paid by Min to Max if the terminal
state is $i$. Then, the equation $F(u)=\mu e+ u$, or $-\mu e+ F(u)=u$ means
that the value of this modified game is independent
of the number of turns (the operator $x\mapsto -\mu e + F(x)$ being
interpreted as the dynamic programming operator of this modified game).
An interest of a bias vector is
that it determines stationary optimal strategies for both players,
by selecting the actions attaining the maximum and minimum
in the expression of $F(u)$.

The map $F$ is semidifferentiable.
To see this, let $E^+_i(x)$ and $E^-_i(x)$ denote the set of nodes $j$ attaining
the maximum and the minimum in~\eqref{e-attainminmax}.
An application of Theorem~\ref{th-semi} with Remark~\ref{rk-finite}
shows that the semidifferential
of $F$ at point $u$ does exist and is given by
\[
(F'_u(x))_i:=   \frac 1 2\big( \max_{j\in E^+_i(u)} x_j + 
\min_{j\in E^-_i(u)} x_j\big).
\]

Consider now as an example the digraph of Figure~\ref{fig-stoch}. The corresponding Shapley operator is given by
\[
F(x)= \left(\begin{array}{c}
\frac{1}{2}\big(
\max(3+x_1,4+x_2,x_3)+
\min(3+x_1,4+x_2,x_3)
\big)\\
\frac{1}{2}\big(
\max(x_1,3+x_2,-7+x_3)+
\min(x_1,3+x_2,-7+x_3)
\big)\\
\frac{1}{2}\big(
\max(3+x_1,2+x_2)+
\min(3+x_1,2+x_2)
\big)
\end{array}\right) \enspace .
\]
The vector $u=(5,0,4)$ can be checked to be an additive eigenvector
of $F$, with additive eigenvalue $\mu=1$, i.e., $F(u)=\mu e +u$,
so that the mean payoff per time unit is equal to $1$ for all initial states
(Max is winning $1$ per time unit).
We get
\[
F'_u(x) = 
\left(\begin{array}{c}
\frac{1}{2}\big(
x_1+ 
\min(x_2,x_3)
\big)\\
\frac{1}{2}\big(
x_1+x_3
\big)\\
\frac{1}{2}\big(
x_1+x_2
\big)
\end{array}\right) \enspace .
\]
Let $\mathsf{t}(x):=\max_{1\leq i\leq n}x_i$, $\mathsf{b}(x):=\min_{1\leq i\leq n} x_i$,
and $\omega(x):=\omega_e(x)=\mathsf{t}(x)-\mathsf{b}(x)$.  One
readily checks that 
\[
\omega(F'_u(x))=\frac{1}{2}\big(x_1+\max(x_2,x_3)\big)
- \frac{1}{2}\big(x_1+\min(x_2,x_3)\big) \leq \frac{1}{2} \omega(x)\enspace.
\]
Hence, $F'_u$ is a contraction of rate $1/2$ in the seminorm $\omega$, which
implies that $\bonsall{}(F'_u)\leq 1/2$.
In particular, every fixed point $x$ of $F'_u$ satisfies $x_1=x_2=x_3$.
Hence,
Corollary~\ref{cor-add} shows that the set of additive eigenvectors
of $F$ is precisely $S=\{u+\lambda e \mid \lambda \in \R\}$
(in other words, the bias vector is unique up to
an additive constant).  
Moreover, Corollary~\ref{cor-addnew}
implies that for all $x\in \R^3$, 
\[
\limsup_{k\to\infty} \omega(F^k(x)-u)^{1/k} \leq 1/2 \enspace ,
\]
and that there is a constant $\lambda\in \R$, depending
on $x$, such that
\[
\limsup_{k\to\infty} \|F^k(x)-\lambda e - u\|^{1/k} \leq 1/2 \enspace .
\]
\begin{figure}[htbp]
\begin{center}
\begin{picture}(0,0)%
\includegraphics{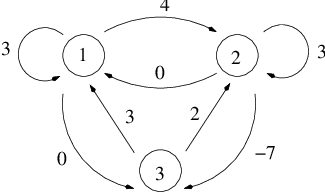}%
\end{picture}%
\setlength{\unitlength}{2072sp}%
\begingroup\makeatletter\ifx\SetFigFont\undefined%
\gdef\SetFigFont#1#2#3#4#5{%
  \reset@font\fontsize{#1}{#2pt}%
  \fontfamily{#3}\fontseries{#4}\fontshape{#5}%
  \selectfont}%
\fi\endgroup%
\begin{picture}(4937,2905)(976,-3491)
\end{picture}%
\end{center}
\caption{An additive version of Richman games}
\label{fig-stoch}
\end{figure}
\bibliographystyle{alpha}
\bibliography{paper}
\end{document}